\documentclass{article}

\usepackage{amsthm, amsmath, amsfonts, mathrsfs, amssymb, bbm}

\usepackage[colorlinks=true, citecolor=violet, linkcolor=blue, filecolor=magenta, urlcolor=cyan]{hyperref}
\usepackage{geometry}

\makeatletter
\newcommand*\rel@kern[1]{\kern#1\dimexpr\macc@kerna}
\newcommand*\widebar[1]{%
	\begingroup
	\def\mathaccent##1##2{%
		\rel@kern{0.8}%
		\overline{\rel@kern{-0.8}\macc@nucleus\rel@kern{0.2}}%
		\rel@kern{-0.2}%
	}%
	\macc@depth\@ne
	\let\math@bgroup\@empty \let\math@egroup\macc@set@skewchar
	\mathsurround\z@ \frozen@everymath{\mathgroup\macc@group\relax}%
	\macc@set@skewchar\relax
	\let\mathaccentV\macc@nested@a
	\macc@nested@a\relax111{#1}%
	\endgroup
}
\makeatother

\makeatletter
\def\keywords{\xdef\@thefnmark{}\@footnotetext}
\makeatother

\usepackage{cite}
\usepackage{cleveref}
\usepackage{authblk}
\usepackage{xcolor}

\newtheorem{thm}{Theorem}[section]
\newtheorem*{thm*}{Theorem}
\newtheorem{prop}[thm]{Proposition}
\newtheorem{coro}[thm]{Corollary}

\newtheorem{lem}[thm]{Lemma}
\theoremstyle{definition}
\newtheorem{defn}[thm]{Definition}

\theoremstyle{remark}
\newtheorem{rmk}[thm]{Remark}

\makeatletter
\def\cref@thmoptarg[#1]#2#3#4{%
    \ifhmode\unskip\unskip\par\fi%
    \normalfont%
    \trivlist%
    \let\thmheadnl\relax%
    \let\thm@swap\@gobble%
    \thm@notefont{\fontseries\mddefault\upshape}%
    \thm@headpunct{.}
    \thm@headsep 5\p@ plus\p@ minus\p@\relax%
    \thm@space@setup%
    #2
    \@topsep \thm@preskip               
    \@topsepadd \thm@postskip           
    \def\@tempa{#3}\ifx\@empty\@tempa%
      \def\@tempa{\@oparg{\@begintheorem{#4}{}}[]}%
    \else%
      \refstepcounter[#1]{#3}
      \@namedef{cref@#3@alias}{#1}
      \def\@tempa{\@oparg{\@begintheorem{#4}{\csname the#3\endcsname}}[]}%
    \fi%
    \@tempa}%
\makeatother

\DeclareMathAlphabet\matheuvm{U}{zeur}{m}{n}
\newcommand{\di}{\matheuvm{D}}
\newcommand{\diff}{\operatorname{d}}
\newcommand{\Leb}{\mathcal{L}}
\newcommand{\citeinfo}[2]{\cite[#1]{#2}}

\newcommand{\bset}{\matheuvm{B}}

\title{Absolute continuity of Wasserstein barycenters on manifolds with a lower Ricci curvature bound}
\author{Jianyu MA \quad
	}
\affil{Institut de Mathématiques de Toulouse}
\date{\today}

\begin{document}
\maketitle


\begin{abstract}
 	Given a complete Riemannian manifold $M$ with a lower Ricci curvature bound,
	we consider barycenters in the Wasserstein space $\mathcal{W}_2(M)$ of probability measures on $M$.
	We refer to them as Wasserstein barycenters,
	which by definition are probability measures on $M$.
	The goal of this article is to 
	present a novel approach to proving their absolute continuity.
	We introduce a new class of displacement functionals
	exploiting the Hessian equality for Wasserstein barycenters.
	To provide suitable instances of such functionals,
	we revisit Souslin space theory, Dunford-Pettis theorem and
	the de la Vall\'ee Poussin criterion for uniform integrability.
	Our method shows that if a probability measure $\mathbb{P}$ on $\mathcal{W}_2(M)$
	gives mass to absolutely continuous measures on $M$,
	then its unique barycenter is also absolutely continuous.
	This generalizes the previous results on compact manifolds by Kim and Pass \cite{kim2017wasserstein}.
\end{abstract}


\renewcommand\contentsname{}
\vspace{-1cm}
\tableofcontents


\section{Introduction}

Barycenter is the notion of mean for probability measures on
metric spaces.
Given a probability measure $\mu$ on the Euclidean space $\mathbb{R}^m$,
if its first and second moments are finite,
then its mean $\int_{\mathbb{R}^m} x \diff \mu(x)$
can be equivalently defined as the unique point where
the infimum $\inf_{y \in \mathbb{R}^m} \int_{\mathbb{R}^m} \| y - x \|^2 \diff \mu(x)$ is reached.
This formulation in terms of minimization and metric
is still valid for general metric spaces, and it leads
to our definition of barycenter (see \Cref{defn:barycenter}).
It is worth noting that the existence of barycenters
is not guaranteed \emph{a priori} for general metric spaces.
We restrict our discussions to proper metric spaces
to ensure that barycenters exist.

Wasserstein spaces are metric spaces extensively studied in
the field of optimal transport theory.
Their geometric properties have gained constant attention.
By Wasserstein barycenter we mean a
barycenter of some probability measure on a given Wasserstein space.
In the simplest case, given two measures $\mu, \nu$ in the Wasserstein space $\mathcal{W}_2(M)$
over a Riemannian manifold $M$,
all minimal geodesics from $\mu$ to $\nu$ are made of
barycenters of $(1 - \lambda) \delta_{\mu} + \lambda \, \delta_{\nu}$ with $\lambda$ varying in $[0, 1]$.
Their absolute continuity (in possibly generalized settings)
was previously studied as the regularity of displacement
in \cite{mccann1997convexity,bernard2007optimal,fathi2010optimal,figalli2008absolute,villani2009optimal}.
A more general case was first studied by Agueh and Carlier
\cite{agueh2011barycenters}, where barycenters of $\mathbb{P}: = \sum_{i=1}^{n} \lambda_i \, \delta_{\mu_i}$
on the Wasserstein space $(\mathcal{W}_2(\mathbb{R}^m), W_2)$ were considered.
In this setting, Wasserstein barycenters are solutions
to the following minimization problem:
\[
	\min_{\nu} \sum_{i=1}^{n} \lambda_i \, W_2(\nu, \mu_i)^2, \quad
	\text{for }\nu \in \mathcal{W}_2(\mathbb{R}^m) .
\]
They proved the existence of barycenters constructively using a dual formulation
and showed that if at least one of $\mu_i$'s is absolutely continuous
with bounded density function,
then the unique barycenter is also absolutely continuous.
Kim and Pass \cite{kim2017wasserstein} conducted the same study for
Wasserstein barycenters on compact Riemannian manifolds $M$ with similar conclusions.
Their generalization is applicable to general probability measures
$\mathbb{P}$ on $\mathcal{W}_2(M)$ that give mass to the set
of absolutely continuous measures with a uniform upper density bound.
The absolute continuity of Wasserstein barycenters
plays an indispensable role in their study of Jensen's type inequalities
for Wasserstein barycenters.
There is also a generalization \cite{jiang2017absolute} 
of Agueh and Carlier's results
to compact Alexandrov spaces with curvature bounded below.

When $\mathbb{P}$ has the form $\sum_{i=1}^n \lambda_i \, \delta_{\mu_i}$
with $\mu_1$ absolutely continuous,
Kim and Pass' proof of the absolute continuity of the
(unique) barycenter $\widebar \mu$ of $\mathbb{P}$ remains valid for non-compact manifolds $M$.
For a general measure $\mathbb{P}$ giving mass to absolutely continuous measures,
the strategy is to approximate $\mathbb{P}$ with finitely supported measures $\mathbb{P}_j$
whose barycenters $\widebar \mu_j$ are already shown to be absolutely continuous.
Thanks to the law of large numbers for Wasserstein barycenters (\Cref{thm:law_of_large_numbers_Wasserstein_barycenter}),
$\widebar \mu_j$ converges to $\widebar \mu$ weakly.
However, this is not sufficient to ensure that $\widebar \mu$ is also absolutely continuous.
To overcome this difficulty, Kim and Pass \cite{kim2017wasserstein} imposed
a uniform upper density bound on $\widebar \mu_j$'s,
which forced them to include the assumption on $\mathbb{P}$.

In our work, instead of following their quantitative approach, we seek for
proper integral functionals $F$ on $\mathcal{W}_2(M)$ that admit finite values only for absolutely continuous measures.
The continuity of these functionals has been studied in various sources, including
\cite{buttazzo1991functionals}, \cite[Theorem 29.20]{villani2009optimal},
\cite[Chapter 7]{Santambrogio2015}, and
\cite[Chapter 15]{ambrosio2021lectures}.
We summarize their assumptions and conclusions in \Cref{lem:lower_semi_continuity_entropy}.
Additionally, we aim to control the value of $F$ at $\widebar \mu_j$ by
those at the support of $\mathbb{P}_j$,
which enables us to use the convergence $\mathbb{P}_j \rightarrow \mathbb{P}$ effectively.
Classic references, such as Villani's monograph \cite{villani2009optimal},
focus on the $\lambda$-convexity of $F$,
a widely studied property that would satisfy our requirements if we tolerate
some independent constants in its inequality expression of convexity (\Cref{prop:entropy_estimation}).
Functionals defined in this way generalize the entropy functional
$ f \cdot \operatorname{Vol} \mapsto \int_M f \log f \diff \operatorname{Vol}$,
which is an important example in the study
of synthetic treatment of Ricci curvature lower bounds developed in
\cite{lott2009ricci,sturm2006geometryI,sturm2006geometryII}.
\Cref{prop:entropy_estimation} reveals how Ricci curvature affects the properties
of Wasserstein barycenters and suggests possible
extensions of our current work to general metric measures spaces.

The methodology previously described leads us to
\Cref{thm:absolute_continuity_main_theorem} on
the absolute continuity of Wasserstein barycenters,
where an extra assumption on $\mathbb{P}$ is needed.
With the help of a generalized de la Vall\'ee Poussin criterion (\Cref{lem:modified_de_la_Vallee_Poussin_theorem}),
this assumption can be further simplified:
we ask that $\mathbb{P}$ gives mass to a compact
subset in some weak topology of absolutely continuous measures.
Although this topology is barely mentioned in the literature of optimal transport,
it generates the same Borel sets as the topology induced by the Wasserstein metric
according to the theory of Souslin space.
This helps us to state our main result with a natural assumption on $\mathbb{P}$:

\begin{thm*}
	Let $(M, \matheuvm{g})$ be a complete Riemannian manifold
	with a lower Ricci curvature bound.
	If a probability measure $\mathbb{P} \in \mathcal{W}_2(\mathcal{W}_2(M))$
	gives mass to the set of absolutely continuous probability
	measures on $M$, then its unique  barycenter is absolutely continuous.
\end{thm*}

\subsection*{Structure of the paper}


In \Cref{sec:wasserstein_barycenter},
we introduce notation and definitions for Wasserstein barycenters,
and then extend Kim and Pass' proof of their absolute continuity to non-compact manifolds.
In \Cref{sec:hessian_equality}, we present
the Hessian equality for Wasserstein barycenters
(\Cref{thm:hessian_equality_Wasserstein_barycenter}),
which is used in \Cref{sec:ricci_curvature}
to justify our displacement functionals (\Cref{prop:entropy_estimation}).
\Cref{sec:measure_theory} primarily concerns Polish spaces,
and we use the Souslin space theory to provide appropriate instances of the
previously defined displacement functionals.
Our main result, \Cref{thm:final_theorem_manifolds},
is a consequence of the intermediate result \Cref{thm:absolute_continuity_main_theorem} 
after proving several auxiliary results.







\section{Wasserstein barycenters}
\label{sec:wasserstein_barycenter}

\subsection{Notation and definitions}
\label{sec:notation}

\begin{defn}[Barycenter]
	\label{defn:barycenter}
	Let $(E,d)$ be a metric space and let $\mu$ be a probability measure on $E$ such that
	$\int_{E} d(x_0, y)^2 \diff \mu(y) < \infty$ for some point $x_0 \in E$.
	We call $z \in E$ a barycenter of $\mu$ if
	\[
		\int_{E} d(z, y)^2 \diff \mu(y)
		= \min_{x \in E} \int_{E} d(x, y)^2 \diff \mu(y).
	\]
\end{defn}

A metric space is proper
if its bounded closed subsets are also compact.
Barycenters always exist in proper spaces since a minimizing sequence is
bounded and thus pre-compact.
We refer to Ohta \cite{ohta2012barycenters} for more details and
some other properties
of barycenters in a proper space.


A metric space is Polish if it is complete and separable.
Any proper space $(E,d)$ is Polish,
so are the Wasserstein spaces built over it.
In this article, we consider the ($2$-)Wasserstein space $(\mathcal{W}_2(E), W_2)$
of probability measures on $E$ with
\begin{align}
	\mathcal{W}_2(E) : & = \left \{ \mu \text{ is a probability measure on } E\, \big|\,
	\exists\, x_0 \in E,
	\int_{E} d(x_0, y)^2 \diff \mu(y) < \infty \right\}, \nonumber                       \\
	\label{equa:Wasserstein_distance}
	W_2(\mu, \nu)^2:   & =  \inf_{ \pi \in \Pi(\mu, \nu)}
	\int_{E \times E} d(x, y)^2 \diff \pi(x, y),
\end{align}
where $\Pi(\mu, \nu)$ is the set of probability measures on $E \times E$
with marginals $\mu$ and $\nu$.
The infimum in (\ref{equa:Wasserstein_distance}) is always reached
by some measure $\pi \in \Pi(\mu, \nu)$, and
we call it an optimal transport plan between $\mu$ and $\nu$.

Since Wasserstein spaces are complete and separable,
we can construct the Wasserstein space
$(\mathcal{W}_2(\mathcal{W}_2(E)), \mathbb{W}_2)$ over the
Wasserstein space $(\mathcal{W}_2(E), W_2)$.
Symbols $W_2$ and $\mathbb{W}_2$ will always denote Wasserstein metrics
in the rest of the paper.
A Wasserstein space $\mathcal{W}_2(E)$ is not proper unless the
base space $E$ is compact \cite[Remark 7.19]{ambrosio2008gradient}.
Classic references on this topic are \cite{villani2009optimal},
\cite{Santambrogio2015}, and \cite{villani2021topics}.

As mentioned before, Wasserstein barycenters are barycenters
of measures on Wasserstein spaces.
We refer to the following result by Le Gouic and Loubes
as the law of large numbers for Wasserstein barycenters since
we can set $\mathbb{P}_j$ to be empirical measures for the law $\mathbb{P}$.
\begin{thm}[Law of large numbers for Wasserstein barycenters, \cite{le2017existence}]
	\label{thm:law_of_large_numbers_Wasserstein_barycenter}
	Let $(E, d)$ be a proper space.
	Fix a probability measure $\mathbb{P} \in \mathcal{W}_2(\mathcal{W}_2(E))$ on $\mathcal{W}_2(E)$.
	Given a sequence of measures $\mathbb{P}_j \in \mathcal{W}_2(\mathcal{W}_2(E))$
	with their corresponding barycenters $\widebar \mu_j \in \mathcal{W}_2(E)$,
	if $\mathbb{W}_2(\mathbb{P}_j , \mathbb{P}) \rightarrow 0$ as $j \rightarrow \infty$, then
	$W_2(\widebar \mu_j, \widebar \mu) \rightarrow 0$ for some barycenter $\widebar \mu$ of $\mathbb{P}$
	up to extracting a subsequence of $\widebar \mu_j$.
\end{thm}

For two topological spaces $E_1$ and $E_2$, we denote by
$p_1$ and $p_2$ the canonical projection maps defined on $E_1 \times E_2$,
where $p_1$ maps $(x, y) \in E_1 \times E_2$ to $x \in E_1$ and
$p_2$ maps $(x, y)$ to $y \in E_2$.
Recall that these projection maps are continuous and open
(mapping open sets to open sets).
The map $p_1$ (respectively $p_2$) is closed
if $E_2$ (respectively $E_1$)
is compact \cite[Proposition 8.2]{bredon2013topology}.
By convention, $\operatorname{Id}$ denotes the identity map
$ x \mapsto x$.
The following lemma is useful when compactness arguments are needed,
whose proof is based on the previous property of projection maps.

\begin{lem}
	\label{lem:compact_barycenter}
	Let $(E, d)$ be a proper space.
	Given an integer $n \ge 1$, let $\lambda_i > 0, 1 \le i \le n$,
	be $n$ positive real numbers such that $\sum_{i=1}^n \lambda_i = 1$.
	The set
	\[
		\Gamma := \left\{ (x_1, \ldots, x_n, z) \in E^{n+1} \,\bigg|\, 
			\sum_{i=1}^n \lambda_i\,d(z, x_i)^2
			= \min_{y \in E} \sum_{i=1}^n \lambda_i\,d(y, x_i)^2
	\right\}
	\]
	is closed.
	Denote by $\operatorname{bary}( \boldsymbol{A} )$
	the set of all barycenters of the measures
	$\sum_{i=1}^n \lambda_i\, \delta_{x_i}$ when
	$(x_1, \ldots, x_n)$ runs through
	a subset $\boldsymbol{A} \subset E^n$.
	If $\boldsymbol{A}$ is compact,
	then $\operatorname{bary}( \boldsymbol{A} )$ is compact.
\end{lem}

\begin{proof}
	For $\boldsymbol{x}:=(x_1, x_2, \ldots, x_n) \in E^n$, we define
	$\eta(\boldsymbol{x}) := \sum_{i=1}^n \lambda_i \,\delta_{x_i} \in \mathcal{W}_2(E)$.
	The map $\eta: (E,d) \rightarrow (\mathcal{W}_2(E), W_2)$ is continuous
	by definition of Wasserstein metric:
	for $\boldsymbol{x}, \boldsymbol{y} \in E^n$,
	\[
		W_2(\eta(\boldsymbol{x}), \eta(\boldsymbol{y}))^2 \leq
		\sum_{i=1}^n \lambda_i \, d(x_i, y_i)^2.
	\]
	It follows from the triangle inequality that
	the map $\boldsymbol{x} \in E^n \mapsto
	\min_{y \in E} W_2(\eta(\boldsymbol{x}), \delta_y)$
	is also continuous, which implies that the set $\Gamma$ is closed
	as $W_2(\eta(\boldsymbol{x}), \delta_y)^2 = \sum_{i=1}^n \lambda_i\,d(y, x_i)^2$.

	Note that $\operatorname{bary}( \boldsymbol{A} )
		= p_2 \left( \Gamma \cap (\boldsymbol{A} \times E) \right)$,
	where $p_2: \boldsymbol A \times E \rightarrow E$ is the canonical projection map.
	If $\boldsymbol A$ is compact, $p_2$ is a closed map
	and thus $\operatorname{bary}( \boldsymbol{A} )$ is closed as $\Gamma$ is closed.
	The set $\operatorname{bary}( \boldsymbol{A} )$
	is bounded since barycenters are located within
	the union of $n$ bounded balls with centers $x_i$.
\end{proof}

For a metric space $E$, denote by $\mathcal{B}(E)$
the $\sigma$-algebra of its Borel sets.
If $E$ is separable, then the support of any Borel measure $\mu$ on $E$
exists \cite[Proposition 7.2.9]{bogachev2007measure}.
This support, denoted by $\operatorname{supp}(\mu)$,
is the closed set
whose complement is the union of all open sets $U \subset E$ satisfying $\mu(U) = 0$.
We shall apply the following widely used measurable selection theorem to construct
Wasserstein barycenters in the next subsection.
Its proof could be found in \cite[Theorem 6.9.3]{bogachev2007measure},
\cite{fremlin2006measurable}, and \cite{koumoullis1983ramsey}.

\begin{thm}[Kuratowski and Ryll-Nardzewski measurable selection theorem]
	\label{thm:measurable_selection_theorem}
	Let $E$ be a Polish metric space, and
	let \( \Psi \) be a map on a measurable space
	$( \Omega , \mathcal { B } )$ with values in the set of nonempty
	closed subsets of $E$.
	Suppose that for every open set $U \subset E$, we have
	\[
		\left\{ \omega \in \Omega \,\mid\, \Psi ( \omega ) \cap U
		\neq \emptyset \right\} \in \mathcal { B }.
	\]
	Then \( \Psi \) has a selection that is measurable with respect to
	the pair of \( \sigma \)-algebras \( \mathcal { B } \) and $\mathcal { B } ( E )$.
\end{thm}

The notion of conditional measures \cite[Definition 10.4.2]{bogachev2007measure}
will be used to prove \Cref{prop:absolutely_continous_barycenter_discrete_marginals}.

\begin{defn}[Conditional probability measures]
	\label{defn:regular_conditional_measure}
	Let $E$ be a Polish metric space and
	let $n \ge 2$ be a positive integer.
	Denote by $ \boldsymbol x^\prime = (x_2, \ldots, x_n) \in E^{n-1}$ the last $n-1$ components
	of a point $\boldsymbol x = (x_1, x_2, \ldots, x_n) \in E^n$.
	Given a Borel probability measures $\gamma$ on $E^n$,
	define the measure $\pi : = {p_2}_{\#} \gamma$ on $E^{n-1}$,
	where $p_2$ is the projection
	$\boldsymbol x \in E \times E^{n-1} \mapsto \boldsymbol x^\prime \in E^{n-1}$.
	We call $\gamma(\cdot, \cdot): \mathcal{B}(E^n) \times E^{n-1} \rightarrow \mathbb{R}$
	a conditional measure for $\gamma$,
	written as $ \label{equa:conditional_measure}
		\diff \gamma(\boldsymbol x) =
		\gamma(\diff \boldsymbol x , \boldsymbol x^\prime)\, \diff \pi(\boldsymbol x^\prime) $,
	if
	\begin{enumerate}
		\item for all $\boldsymbol x^\prime \in E^{n-1}$, $\gamma(\cdot , \boldsymbol x^\prime)$ is a
		      Borel probability measure on $E^n$,
		\item for $\pi$-almost every $\boldsymbol x^\prime \in E^{n-1}$,
		      $\gamma(\cdot, \boldsymbol x^\prime)$ is concentrated on $E \times \{\boldsymbol x^\prime\}$,
		\item for any Borel set $\boldsymbol R \subset E^n$, the function
		      $\boldsymbol  x^\prime \mapsto \gamma(\boldsymbol R , \boldsymbol x^\prime)$ is measurable, and
		\item for any Borel set $\boldsymbol S \subset E^{n-1}$,
		      $\gamma[\boldsymbol R \cap (E \times \boldsymbol S)] =
			      \int_{\boldsymbol S} \gamma(\boldsymbol R , \boldsymbol x^\prime) \diff \pi(\boldsymbol x^\prime)$.
	\end{enumerate}
\end{defn}

Under our assumption that $E$ is Polish,
conditional measures always exist
\cite[Corollary 10.4.10]{bogachev2007measure}.
For $\pi$-almost every $\boldsymbol x^\prime$, the measure $\gamma(\cdot, \boldsymbol x^\prime)$
is unique \cite[Lemma 10.4.3]{bogachev2007measure}
and coincides with the disintegration \cite[452E]{fremlin2000measure}
of $\gamma$ that is consistent with the projection $p_2$.

\vspace{0.6cm}
Finally, throughout this document, we assume that (Riemannian) manifolds are connected
and smooth without boundary.
These assumptions enable us to apply the results
by McCann \cite[Proposition 6]{mccann2001polar} and
Cordero-Erausquin et al.\@ \cite{cordero2001riemannian}.
In most propositions, we also assume that the manifolds are complete.
We always denote by $(M, \matheuvm{g})$ such a manifold, by
$d_{\matheuvm{g}}$ its intrinsic geodesic metric, by $\exp: TM \rightarrow M$
the exponential map on its tangent bundle, and by $\operatorname{Vol}$ the volume measure on it.
Denote by $\Leb^m$ the Lebesgue measure on $\mathbb{R}^m$.
$\mathbb{N}^* : = \mathbb{N} \setminus \{0\}$ is the set of natural numbers
$\mathbb{N}$ with $0$ excluded.


\subsection{Construction, existence and uniqueness of Wasserstein barycenters}
\label{sec:existence_and_uniqueness}

This subsection covers several fundamental properties of Wasserstein barycenters.
We construct them via optimal transport theory and measurable barycenter selection maps.
It is crucial to comprehend further features of these maps,
as highlighted in Section \ref{sec:Lipschitz_continuous_transport_map}.
Once the construction is explained,
we discuss the problem of existence and uniqueness of Wasserstein barycenters.
These properties are closely related to the law of large numbers for Wasserstein barycenters.

We begin with the existence of measurable barycenter selection maps.

\begin{lem}[Measurable barycenter selection maps]
	\label{thm:measurable_barycenter_selection_for_discrete_measures}
	Let $(E, d)$ be a proper space.
	Given an integer $n \ge 1 $, let $\lambda_i > 0, 1 \le i \le n$,
	be $n$ positive real numbers such that $\sum_{i=1}^n \lambda_i = 1$.
	There exists a measurable barycenter selection map
	$B: E^n \rightarrow E$ such that
	$B(x_1, \ldots, x_n)$ is a barycenter of $\sum_{i=1}^n \lambda_i \,\delta_{x_i} \in \mathcal{W}_2(E)$.
\end{lem}

\begin{proof}

	As in \Cref{lem:compact_barycenter}, for a subset $\boldsymbol{A} \subset E^n$,
	denote by $\operatorname{bary}(\boldsymbol{A}) \subset E$ the set of barycenters
	of $\sum_{i=1}^n \lambda_i \,\delta_{x_i}$ when $\boldsymbol{x}=(x_1, \ldots, x_n)$
	runs through $\boldsymbol{A}$.
	For the existence of a measurable function $B: E^n \rightarrow E$ such that
	$B(\boldsymbol{x}) \in \operatorname{bary}(\{ \boldsymbol{x} \})$,
	we shall apply the Kuratowski and Ryll-Nardzewski measurable selection theorem
	(\Cref{thm:measurable_selection_theorem}).
	By \Cref{lem:compact_barycenter},  the set $\Gamma : =\{
		(\boldsymbol{x}, z) \in E^{n+1} \mid z \in \operatorname{bary}(\{ \boldsymbol{x} \}) \}$ is closed.
	Let $C \subset E$ be a compact set, then
	\[
		\{ \boldsymbol x \, | \, \operatorname{bary}(\{ \boldsymbol{x} \})
		\cap C \neq \emptyset \} =
		p_1(\{ (\boldsymbol x, z) \in  E^n \times C \,\mid\,
		(\boldsymbol x, z) \in \Gamma \}),
	\]
	where $p_1: E^n \times C \rightarrow E^n$ is the canonical projection map.
	Since $C$ is compact and $\Gamma$ is closed,
	$p_1(\Gamma \cap (E^n \times C))$ is a closed set.
	As $(E, d)$ is a proper space,
	any open subset $U \subset E$ is a countable union of compact sets.
	Indeed, fix a point $z \in E$ and define
		$C_j := \{x \in E \mid d(x, z) \le j \text{ and }
		d(x, E \setminus U) \ge \frac{1}{j} \} $ for $j = 1, 2, \ldots$,
		then each $C_j$ is compact and $U = \cup_{j \ge 1} C_j$.
	Therefore, as a union of countably many closed sets,
	the set 
	\[
		\{ \boldsymbol x \, | \, \operatorname{bary}(\{ \boldsymbol{x} \}) \cap U \neq \emptyset \}
		 = \bigcup_{j \ge 1} \{ \boldsymbol x \, 
		| \, \operatorname{bary}(\{ \boldsymbol{x} \}) \cap C_j \neq \emptyset \}
	\]
	is measurable.
	Since $\operatorname{bary}(\{ \boldsymbol{x} \})$ is compact for $\boldsymbol{x} \in E^n$
	by \Cref{lem:compact_barycenter},
	the assumptions of \Cref{thm:measurable_selection_theorem} are satisfied
	by the map $\boldsymbol{x} \mapsto \operatorname{bary}(\{\boldsymbol{x}\})$.
	This proves the lemma.
\end{proof}

To construct Wasserstein barycenters of finitely many measures,
we first recall the following particular type of multi-marginal optimal
transport plans.

\begin{defn}[Multi-marginal optimal transport plans]
	Let $(E, d)$ be a proper space.
	Given an integer $n \ge 2 $,
	let $\lambda_i > 0, 1 \le i \le n$, be $n$ positive real numbers such
	that $\sum_{i=1}^n \lambda_i = 1$ and let $\mu_i \in \mathcal{W}_2(E), 1 \le i \le n$, be $n$
	probability measures on $E$.
	Denote by $\Pi$ the set of probability measures on $E^n$
	with marginals $\mu_1, \ldots, \mu_n$ in this order.
	We call $\gamma \in \Pi$ a multi-marginal optimal transport plan
	(of its marginals) if
	\begin{equation}
		\label{equa:multi_marginal_optimal_plan_definition}
		\int_{E^n} 
		\min_{y \in E} \sum_{i=1}^n \lambda_i\, d(y, x_i)^2
		\diff \gamma(x_1, \ldots, x_n)
		= \min_{\theta \in \Pi}
		\int_{E^n}
		\min_{y \in E} \sum_{i=1}^n \lambda_i\, d(y, x_i)^2
		\diff \theta(x_1, \ldots, x_n).
	\end{equation}
\end{defn}

In what follows, the marginal measures $\mu_i$ and constants $\lambda_i$ will be clear
from the context.
In the proof of \Cref{lem:compact_barycenter},
it is shown that $\min_{y \in E} \sum_{i=1}^n \lambda_i\, d(x_i, y)^2$ is continuous
with respect to $(x_1, \ldots, x_n) \in E^n$.
Hence, we can conclude the existence of a multi-marginal
optimal transport plan $\gamma$ in the same way
as the classic existence of optimal couplings between two measures
\cite[Theorem 4.1]{villani2009optimal}.
Now we are ready to construct Wasserstein barycenters.

\begin{prop}[Construction of Wasserstein barycenters of $\sum_{i=1}^n \lambda_i \, \delta_{\mu_i}$]
	\label{prop:construction_Wasserstein_barycenter}
	Let $(E,d)$ be a proper space.
	Given an integer $n \ge 2$,  let $\lambda_i > 0, 1 \le i \le n$,
	be $n$ positive real numbers such that $\sum_{i=1}^n \lambda_i = 1$.
	Let $\mu_1, \ldots, \mu_n \in \mathcal{W}_2(E)$ be $n$ probability measures
	and let $\gamma$ be a multi-marginal optimal transport plan
	of them, i.e., satisfying (\ref{equa:multi_marginal_optimal_plan_definition}).
	If $B: E^n \rightarrow E$ is a measurable map such that
	$B(x_1, \ldots, x_n)$ is a barycenter of $\sum_{i=1}^n \lambda_i \,\delta_{x_i}$,
	then
	\begin{enumerate}
		\item $\widebar{\mu} : = B_{\#} \gamma$ is a barycenter
		      of $\mathbb{P}:=\sum_{i=1}^n \lambda_i \,\delta_{\mu_i}$;
		\item $(B, p_i)_{\#} \gamma$ is an optimal transport plan
		      between $\widebar \mu$ and $\mu_i$, where $p_i$ denotes
		      the canonical projection $(x_1, \ldots, x_n) \in E^n \mapsto x_i \in E$;
		\item if $X, X_1, \ldots, X_n$ are $n + 1$ random variables
			from a probability space $(\Omega, \mathcal{B}, P)$ to $(E, d)$ with law $\widebar \mu, \mu_1, \ldots, \mu_n$
			such that $\mathbb{E} \,d(X, X_i)^2 = W_2(\widebar \mu, \mu_i)^2$, i.e., $(X, X_i)$ is an optimal transport coupling between $\widebar \mu$ and $\mu_i$,
			then for $P$-almost every $\omega \in \Omega$,
			$X(\omega)$ is a barycenter of $\sum_{i=1}^n \lambda_i\, \delta_{X_i(\omega)}$.
	\end{enumerate}
\end{prop}

\begin{proof}
	Given an arbitrary probability measure $\nu \in \mathcal{W}_2(E)$,
	thanks to the gluing lemma \cite[Lemma 7.1]{villani2021topics},
	there are $n+1$ random variables $X, X_1, \ldots X_n$ valued in $E$
	with laws $\nu, \mu_1, \ldots \mu_n$ such that
	$\mathbb{E} \, d(X, X_i)^2 = W_2(\nu, \mu_i)^2$.
	Since $\mu_i = {p_i}_{\#} \gamma$, we have
	\begin{align*}
		\sum_{i=1}^{n}\lambda_i \, W_2(\widebar{\mu}, \mu_i)^2
		 & \leq \sum_{i=1}^{n} \int_{E^n} \lambda_i\, d(B(\boldsymbol x), x_i)^2
		\diff \gamma (\boldsymbol{x}) =  \int_{E^n} 
		\min_{y \in E} \sum_{i=1}^n \lambda_i\, d(y, x_i)^2
		\diff \gamma (\boldsymbol{x})                                            \\
		 & \le \mathbb{E}\, \min_{y \in E} \sum_{i=1}^n \lambda_i\, d(y, X_i)^2
		\leq \mathbb{E} \sum_{i=1}^n \lambda_i\, d(X, X_i)^2                     \\
		 & = \sum_{i=1}^{n}\lambda_i\, W_2(\nu, \mu_i)^2,
	\end{align*}
	where we sequentially applied the definitions of $\widebar{\mu} = B_{\#} \gamma$,
	$W_2(\widebar{\mu}, \mu_i)$, $\gamma$, and $X, X_1, \ldots X_n$.
	Since $\nu$ is arbitrary, it follows that $\widebar \mu$ is a Wasserstein barycenter.
	By setting $\nu = \widebar{\mu}$ in the above inequality,
	we actually obtain an equality.
	Firstly, this equality implies that
	$\sum_{i=1}^{n}\lambda_i \, W_2(\widebar{\mu}, \mu_i)^2 \leq \sum_{i=1}^{n}
	\int_{E^n} \lambda_i\, d(B(\boldsymbol{x}), x_i)^2 \diff \gamma (\boldsymbol{x})$
	is indeed always an equality,
	which proves the second statement.
	Secondly, it also implies
	that the law of $(X_1, \ldots, X_n)$ is a multi-marginal optimal transport plan
	and $\min_{y \in E} \sum_{i=1}^n \lambda_i\,d(y, X_i(\omega))^2 =
		\sum_{i=1}^n \lambda_i\, d(X(\omega), X_i(\omega))^2$
	for $P$-almost every $\omega \in \Omega$,
	which proves the third statement.
\end{proof}

The general existence of Wasserstein barycenters
was first established in \cite{le2017existence}.
Recall that finitely supported measures
are dense in Wasserstein spaces \cite[Theorem 6.18]{villani2009optimal},
so the above construction of Wasserstein barycenters together with the law of large numbers
for Wasserstein barycenters (\Cref{thm:law_of_large_numbers_Wasserstein_barycenter})
implies the following theorem.

\begin{thm}[Existence of Wasserstein barycenters]
	\label{thm:exsitence_Wasserstein_barycenter}
	If $(E,d)$ is a proper space,
	then any $\mathbb{P} \in \mathcal{W}_2(\mathcal{W}_2(E))$ has a barycenter.
\end{thm}

Note that in the law of large numbers for Wasserstein barycenters
(\Cref{thm:law_of_large_numbers_Wasserstein_barycenter}),
we may need to pass to a subsequence of Wasserstein barycenters
$\widebar \mu_j$ and
the limit barycenter $\widebar \mu$ is not known in advance.
Hence, \Cref{thm:law_of_large_numbers_Wasserstein_barycenter} will be
enhanced if we can assert the uniqueness of barycenters
under some additional assumptions, as follows.

\begin{prop}[Uniqueness of Wasserstein barycenters]
	\label{prop:uniqueness_barycenter_Wasserstein}
	Let $(E,d)$ be a proper space.
	If a probability measure $\mathbb{P} \in \mathcal{W}_2(\mathcal{W}_2(E))$ gives mass
	to a Borel subset $\mathcal{A} \subset \mathcal{W}_2(E)$ such that
	for $\mu \in \mathcal{A}$ and $\nu \in \mathcal{W}_2(E)$,
	any optimal transport plan between $\mu$ and $\nu$ is
	induced by a measurable map $T$ pushing $\mu$ forward to $\nu$,
	i.e., $\nu = T_{\#} \mu$ and $W_2(\mu, \nu)^2 = \int_E d(x, T(x))^2 \diff \mu$,
	then $\mathbb{P}$ has a unique barycenter in $\mathcal{W}_2(E)$.
\end{prop}

\begin{proof}
	The uniqueness follows from the strict convexity
	of the squared distance function to a point in $\mathcal{W}_2(E)$,
	as shown by \cite[Theorem 7.19]{Santambrogio2015}
	and \cite[Theorem 3.1]{kim2017wasserstein}.
	We recall the proof for the sake of completeness.

	Observe that any convex combination of probability measures in the space
	$\mathcal{W}_2(E)$ is still a probability measure in it.
	Fix $\mu \in \mathcal{A}$ and consider the squared Wasserstein distance function $W_2(\mu, \cdot)^2$
	with respect to this convex structure.
	For $\lambda \in [0, 1]$ and
	two different probability measures $\nu_1,\nu_2 \in \mathcal{W}_2(E)$,
	by definition of Wasserstein metric we have
	\begin{equation}
		\label{equa:convexity_Wassersetein_distance}
		W_2(\mu, \lambda\, \nu_1 + (1- \lambda)\nu_2)^2 \leq
		\lambda\, W_2(\mu, \nu_1)^2 + (1-\lambda) W_2(\mu, \nu_2)^2.
	\end{equation}

	By our assumptions, there are two measurable maps $T_1, T_2: E \rightarrow E$
	such that $\gamma_1 : = (\operatorname{Id}  \times T_1)_{\#}\mu$ and
	$\gamma_2 : = (\operatorname{Id}  \times T_2)_{\#}\mu$
	are optimal transport plans between $\mu$ and the two measures $\nu_1$ and $\nu_2$ respectively.
	We claim that (\ref{equa:convexity_Wassersetein_distance}) cannot
	be an equality unless $\lambda = 0$ or $\lambda =1$.
	Indeed, if (\ref{equa:convexity_Wassersetein_distance}) is
	an equality for some $0 < \lambda < 1$,
	then by setting $\gamma := \lambda\, \gamma_1 + (1-\lambda)\, \gamma_2$ we have
	\begin{align*}
		\lambda\, W_2(\mu, \nu_1)^2 +(1 - \lambda) W_2(\mu, \nu_2)^2
		 & = W_2(\mu, \lambda\, \nu_1 + (1 - \lambda)\nu_2)^2             \\
		 & \leq \int_{E \times E} d(x,y)^2 \diff \gamma(x,y)              \\
		 & =	\lambda\, W_2(\mu, \nu_1)^2 + (1-\lambda) W_2(\mu, \nu_2)^2,
	\end{align*}
	and thus $\gamma$ is an optimal plan between $\mu$ and $\lambda \,\nu_1 + (1-\lambda)\nu_2$.
	By assumptions, there exists a measurable map $T: E \rightarrow E$ such
	that $\gamma = (\operatorname{Id} \times T)_{\#} \mu$.
	Denote by $\operatorname{graph}(S) \subset E^2$ the graph of a map $S: E \rightarrow E$.
	Note that if $S$ is a measurable map, then
	$\operatorname{graph}(S) = \{ (x, y) \in E^2 \mid d(S(x), y) = 0 \}$
	is a Borel subset of $E^2$.
	Since $\gamma[ \operatorname{graph}(T) ] 
	= \lambda\, \gamma_1[ \operatorname{graph}(T) ] + (1- \lambda)
	\gamma_2[ \operatorname{graph}(T) ] = 1$
	and $ 0 < \lambda < 1$,
	we have $\gamma_1[ \operatorname{graph}(T) ] = \gamma_2[ \operatorname{graph}(T) ] = 1$.
	Hence, for $i \in \{ 1,2 \}$,
	$\mu(\{ x \in E \mid T_i(x) = T(x) \}) = \gamma_i[ \operatorname{graph}(T) \cap
	\operatorname{graph}(T_i) ] = 1$.
	It follows that both $T_1$ and $T_2$ coincide with $T$ 
	almost everywhere with respect to $\mu$
	and thus $\gamma_1 = \gamma_2$, which is a contradiction since $\nu_1 \neq \nu_2$.

	This shows that $W_2(\mu, \cdot)^2$ is strictly convex on
	$\mathcal{W}_2(E)$ for $\mu \in \mathcal{A}$.
	Since $\mathbb{P}(\mathcal{A}) > 0$, the map
	\[
		\nu \in \mathcal{W}_2(E) \mapsto \int_{\mathcal{W}_2(E)} W_2(\mu, \nu)^2 \diff \mathbb{P}(\mu)
	\]
	is also strictly convex on $\mathcal{W}_2(E)$
	by the linearity and positivity of the above integral.
	It follows that the Wasserstein barycenter of $\mathbb{P}$
	asserted by \Cref{thm:exsitence_Wasserstein_barycenter} is unique.
\end{proof}

\begin{rmk}
	\label{rmk:unique_induced_transport_plan}
	Under the assumptions of \Cref{prop:uniqueness_barycenter_Wasserstein},
	the optimal transport plan between
	$\mu \in \mathcal{A}$ and $\nu \in \mathcal{W}_2(M)$ is unique.
	By setting $\nu_1 = \nu_2 = \nu$, (\ref{equa:convexity_Wassersetein_distance})
	becomes an equality for any $\lambda \in [0, 1]$.
	It is shown above that any two optimal transport plans $\gamma_1$ and $\gamma_2$
	between measures $\mu$ and $\nu$ coincide.
\end{rmk}

There are many setups in which we can
apply \Cref{prop:uniqueness_barycenter_Wasserstein}.
We typically choose
$\mathcal{A}$ as the set of absolutely continuous measures
with respect to some given reference measure.
The following lemma ensures that $\mathcal{A}$ is a Borel set
of $(\mathcal{W}_2(E), W_2)$.

\begin{lem}
	\label{lem:measurable_set_absolutely_continous_measures}
	Let $E$ be a metric space with a $\sigma$-finite Borel measure $\mu$ on $E$.
	Assume that $\mu$ is outer regular, i.e.,
	for any Borel set $N \in \mathcal{B}(E)$,
	$ \mu(N) = \inf \{ \mu(O) \mid O \text{ open neighborhood of N } \}$.
	Denote by $\mathcal{A}$ the set of probability measures in $\mathcal{W}_2(E)$ that are
	absolutely continuous with respect to $\mu$.
	For $\epsilon, \delta > 0$, define the set
	\[
		\mathcal{E}_{\epsilon, \delta}
		: = \left\{ \nu \in \mathcal{W}_2(E)  \mid
		\forall\, N \in \mathcal{B}(E),\, \mu(N) < \delta \implies \nu(N) \le \epsilon \right\}.
	\]
	It is a closed set with respect to
	the weak convergence topology of $\mathcal{W}_2(E)$, and we have
	\[
		\mathcal{A} = \bigcap_{k \in \mathbb{N}}\bigcup_{l \in \mathbb{N}} \mathcal{E}_{2^{-k}, 2^{-l}}.
	\]
	In particular, if $E$ is a proper space and $\mu$ is a locally finite
	Borel measure, i.e., $\mu$ gives finite mass to 
	some open neighborhood of every point in $E$,
	then for the Wasserstein space topology, $\mathcal{E}_{\epsilon, \delta}$ is a closed set and
	$\mathcal{A}$ is a Borel set.
\end{lem}

\begin{proof}
	Our proof is based on \cite[Proposition 2.1, Remark 2.2]{kim2017wasserstein}
	though we use different assumptions.

	Suppose that $\nu_j \in \mathcal{E}_{\epsilon, \delta}$ converges weakly to
	$\nu \in \mathcal{W}_2(E)$.
	For any $N \in \mathcal{B}(E)$ such that $\mu(N) < \delta$,
	there exists an open set $O$ such that $N \subset O$ and $ \mu(O) < \delta$
	since $\mu$ is outer regular.
	By the characterization of weak convergence of probability measures
	on metric spaces \cite[Corollary 8.2.10]{bogachev2007measure},
	we have
	\[
		\nu(N) \le \nu(O) \le \liminf_{j \rightarrow \infty} \nu_j(O) \le \epsilon
	\]
	and thus $\mathcal{E}_{\epsilon, \delta}$ is closed with respect to weak
	convergence topology of $\mathcal{W}_2(E)$.

	The inclusion $\mathcal{A} \supset
		\bigcap_{k \in \mathbb{N}}\bigcup_{l \in \mathbb{N}} \mathcal{E}_{2^{-k}, 2^{-l}}$
	follows from the definition of a measure $\nu$ being absolutely continuous
	with respect to $\mu$: $\forall\, N \in \mathcal{B}(E),\, \mu(N) = 0 \implies \nu(N) = 0$.
	Fix a measure $\nu \in \mathcal{A}$.
	Since $\mu$ is $\sigma$-finite,
	we can apply the Radon-Nikodym theorem to write $\nu = f \cdot \mu$.
	The reverse inclusion $\mathcal{A} \subset
		\bigcap_{k \in \mathbb{N}}\bigcup_{l \in \mathbb{N}} \mathcal{E}_{2^{-k}, 2^{-l}}$
	follows from the absolute continuity
	of Lebesgue integral \cite[Theorem 2.5.7, Proposition 2.6.4]{bogachev2007measure}.

	Given a proper space $E$ and a locally finite Borel measure $\mu$,
	$\mu$ gives finite mass to compact sets, and every open subset of $E$ is $\sigma$-compact.
	It follows that $\mu$ is outer regular \cite[Theorem 6 of \S 2.7]{swartz1994measure}
	and also $\sigma$-finite.
	Since Wasserstein convergence implies weak convergence,
	the set $\mathcal{E}_{\epsilon, \delta}$ is closed for the Wasserstein metric.
	It follows that $\mathcal{A}$ is a Borel set of $\mathcal{W}_2(E)$.
\end{proof}

\begin{rmk}
	On a metric space, any finite Borel measure is outer regular, see
	\cite[Definition 7.1.5, Theorem 7.1.7]{bogachev2007measure} or
	\cite[Theorem 1.1]{billingsley1999convergence}.
	However, this is not true for $\sigma$-finite Borel measures.
	For example, define the Borel measure $\mu$ on $\mathbb{R}$ such that
	for $N \in \mathcal{B}(\mathbb{R})$, $\mu$ counts the number
	of rational points in $N$.
	This measure is $\sigma$-finite but not outer regular since
	$\mu$ never gives finite mass to open sets.
	As for the assumption regarding the $\sigma$-compactness of open sets
	in the above cited theorem \cite[Theorem 6 of \S 2.7]{swartz1994measure},
	for metric spaces it can be replaced by assuming that $\mu$ gives
	finite mass to a sequence of open sets $O_i, i \ge 1$ such that
	$E = \bigcup_{i \ge 1} O_i$.
	We also mention that there exists a $\sigma$-finite and locally finite but not outer regular
	Borel measure on a locally compact Hausdorff space
	\cite[problem 5 of Exercise \S 1, INT IV.119]{bourbaki2004extension}.
\end{rmk}

Thanks to \Cref{prop:uniqueness_barycenter_Wasserstein} and
\Cref{lem:measurable_set_absolutely_continous_measures},
the Wasserstein barycenter of $\mathbb{P}$ is unique for the following spaces,
provided that $\mathbb{P}$ gives mass to the set of
absolutely continuous measures with respect to the
corresponding canonical reference measure:
\begin{enumerate}
	\item complete Riemannian manifolds, see Villani \cite[Theorem 10.41]{villani2009optimal} or Gigli \cite[Theorem 7.4]{gigli2011inverse};
	\item compact finite dimensional Alexandrov spaces, see Bertrand \cite[Theorem 1.1]{bertrand2008existence};
	\item for $K \in \mathbb{R}$ and $N \ge 1$, non-branching $\operatorname{CD}(K, N)$ spaces,
	      see Gigli \cite[Theorem 3.3]{gigli2012optimal};
	\item for $K \in \mathbb{R}$ and $N \ge 1$, $\operatorname{RCD^*}(K, N)$ spaces,
	      see Gigli, Rajala and Sturm \cite[Theorem 1.1]{gigli2016optimal};
	\item for $K \in \mathbb{R}$ and $N \ge 1$,
	      essentially non-branching $\operatorname{MCP}(K, N)$ spaces,
	      see Cavalletti and Mondino \cite[Theorem 1.1]{cavalletti2017optimal};
	\item ($2$-)essentially non-branching spaces with qualitatively non-degenerate
	      reference measures, see Kell \cite[Theorem 5.8]{kell2017transport}.
\end{enumerate}

The above spaces are listed in (nearly) ascending order of generality.
For the metric measure spaces, we assume that the metric space is proper
and the reference measure is locally finite.
The references cited above demonstrate that
the unique optimal transport plan (\Cref{rmk:unique_induced_transport_plan})
between an absolutely continuous probability measure
and a given probability measure is induced by a measurable map,
allowing us to apply \Cref{prop:uniqueness_barycenter_Wasserstein}.

Since the existence and uniqueness of Wasserstein barycenters (under mild assumptions)
on Riemannian manifolds $M$ are established, we are ready to prove the absolute
continuity of Wasserstein barycenters with respect to $\operatorname{Vol}$.
In the next subsection, we consider Wasserstein spaces $\mathcal{W}_2(M)$
over Riemannian manifolds and show that
the unique Wasserstein barycenter of finitely many measures
is absolutely continuous if one of those measures is so.


\subsection{The absolute continuity of Wasserstein barycenters of finitely many measures}

Let $(M, \matheuvm{g})$ be a complete Riemannian manifold and
let $\mathbb{P} = \sum_{i=1}^n \lambda_i\, \delta_{\mu_i}$ 
be a probability measure on $\mathcal{W}_2(M)$
with positive real numbers $\lambda_i$ and compactly supported measures $\mu_i$ in $\mathcal{W}_2(M)$.
Assuming that $M$ is compact and $\mu_1$ is absolutely continuous
(with respect to $\operatorname{Vol}$),
Kim and Pass \cite[Theorem 5.1]{kim2017wasserstein} proved that
unique barycenter $\widebar{\mu}$ of $\mathbb{P}$ is absolutely continuous.
For completeness and also for a rigorous foundation of our later arguments,
we provide a proof for general non-compact manifolds.

The proof strategy is to investigate different cases
for measures $\mu_i, 2 \le i \le n$, step by step.
In the simplest case when $\mu_i = \delta_{x_i}, 2 \le i \le n$, are Dirac measures,
the unique barycenter $\widebar{\mu}$ is the push-forward of
$\mu_1 \otimes \delta_{x_2} \otimes \ldots \otimes \delta_{x_n}$ by a measurable barycenter selection map $B$
(\Cref{prop:construction_Wasserstein_barycenter}).
To deduce more properties of $B$, we shortly review $c$-concave functions.


\subsubsection{\texorpdfstring{$c$}{c}-concave functions}

For $x,y \in M$, we define the function $c(x,y): = \frac{1}{2} d_{\matheuvm{g}}(x,y)^2$
as the half of the squared distance between $x$ and $y$ in $M$,
and define $d_y^2(\cdot) := d_{\matheuvm{g}}(\cdot, y)^2$ to avoid ambiguity
when fixing the point $y$.

\begin{defn}[$c$-transforms and $c$-concave functions]
	Let $(M, \matheuvm{g})$ be a Riemannian manifold.
	Let \( X \) and \( Y \) be two non-empty compact subsets of \( M \).
	A function \( \phi \) : \( X \rightarrow \mathbb { R }\)
	is $c$-concave if there exists a function 
	\( \psi : Y \rightarrow \mathbb { R } \) such that
	\begin{equation}
		\label{defn:c_transform}
		\phi ( x ) = \inf _ { y \in Y } c ( x , y ) - \psi ( y ), \quad \forall x \in X.
	\end{equation}
	We write it as \( \phi = \psi ^ { c } \) and call $\phi$ the $c$-transform of $\psi$.
	The set of all $c$-concave functions with respect to $X$ and $Y$
	is denoted by $\mathcal{I}^c(X, Y)$.
\end{defn}


The significance of $c$-concave functions in optimal transport theory
is highlighted by the following theorem of McCann \cite{mccann2001polar},
which extends Brenier's seminal theorem \cite[Theorem 2.12]{villani2021topics} to Riemannian manifolds.
Recall that given a $c$-concave function $\phi$ on a compact set $\widebar{\mathcal{X}}$
with $\mathcal{X} \subset M$ open, its gradient $\nabla \phi$
exists on $\mathcal{X}$ almost everywhere with respect to $\operatorname{Vol}$
since $\phi$ is Lipschitz \cite[Lemma 4]{mccann2001polar}.

\begin{thm}[Optimal transport on manifolds, \citeinfo{Theorem 3.2}{cordero2001riemannian}]
	\label{thm:optimal_transport_manifold}
	Let $(M, \matheuvm{g})$ be a complete Riemannian manifold.
	Fix two measures $\mu, \nu \in \mathcal{W}_2(M)$ with compact support
	such that $\mu$ is absolutely continuous.
	Given two bounded open subsets $\mathcal{X} , \mathcal{Y} \subset M$
	containing the supports of \( \mu \) and \( \nu \) respectively,
	there exists $\phi \in \mathcal{I}^c(\widebar{\mathcal{X}}, \widebar{\mathcal{Y}})$
	such that $(\operatorname{Id}, F)_{\#} \mu$ is the unique optimal transport
	plan between \( \mu \) and \( \nu \), where the function $F: = \exp( - \nabla \phi)$
	is $\mu$-almost everywhere well-defined.
\end{thm}

The following lemma shows that the definition of barycenters for measures
$\sum_{i=1}^n \lambda_i\, \delta_{x_i}$ on $M$ involves $c$-concave functions.

\begin{lem}
	\label{lem:barycenter_and_c_concave_function}
	Let $(M, \matheuvm{g})$ be a complete Riemannian manifold.
	Given an integer $n \ge 2$,
	let $\lambda_i > 0, 1 \le i \le n$, be $n$ positive real
	numbers such that $\sum_{i=1}^n \lambda_i = 1$.
	We define
	\begin{equation}
		\label{equa:barycenter_potential}
		f: (x_1, x_2, \ldots, x_n) \in M^n \mapsto 
		\min_{w \in M} \sum_{i=1}^n \lambda_i \,c(w, x_i) =
		\frac{1}{2} \min_{w \in M} \sum_{i=1}^n \lambda_i \,d_{\matheuvm{g}} (w, x_i)^2.
	\end{equation}
	Fix a non-empty compact subset $X \subset M$ and
	$n-1$ points $x_i \in M, 2 \le i \le n$.
	Denote by $Y$ the set of all barycenters of $\sum_{i=1}^n \lambda_i\, \delta_{x_i}$
	when $x_1$ runs through $X$.
	Define $f_1: x_1 \in X  \mapsto f(x_1, \ldots, x_n) / \lambda_1$ and
	 $g_1: y \in Y \mapsto - 1/ \lambda_1 \sum_{i = 2}^n
		\lambda_i\, c(y, x_i) $,
	then $f_1  = g_1^c \in \mathcal{I}^c(X, Y)$ 
	and $g_1 = f_1 ^c \in \mathcal{I}^c(Y, X)$.
\end{lem}

\begin{proof}
	The set $Y \subset M$ is compact by \Cref{lem:compact_barycenter}.
	Using the given definition of $Y$, we can replace
	the minimum over $M$ in (\ref{equa:barycenter_potential})
	by the minimum over $Y$,
	which shows the equality $f_1 = g_1^c \in \mathcal{I}^c(X, Y)$.

	Since $ f_1 (x) + g_1(y) \le c(x, y)$
	for any $(x, y) \in X \times Y$, we have
	\begin{equation}
		\label{equa:g_1_equality_one_direction}
		g_1(y) \leq f_1 ^c(y) := \inf_{x \in X} c(x, y) - f_1(x) .
	\end{equation}
	Fix an arbitrary point $y \in Y$. Our definition of $Y$ implies
	the existence of $x_1 \in X$ such that $y$ is
	a barycenter of $\sum_{i=1}^n \lambda_i\, \delta_{x_i}$.
	For such a pair $(x_1, y) \in X \times Y$,
	$f_1 (x_1) + g_1(y) = c(x_1, y)$ by the definitions of $f_1$ and $g_1$.
	It follows from the inequalities $f_1 (x_1) + f_1^c(y) \leq c (x_1, y)
		= f_1 (x_1) + g_1(y)$ and (\ref{equa:g_1_equality_one_direction})
	that $g_1(y) = f_1^c (y)$.
	Since $y$ is arbitrarily chosen, we conclude that
	$g_1 = f_1^c \in \mathcal{I}^c(Y, X)$.
\end{proof}

The $c$-concave function $g_1 \in \mathcal{I}^c(Y, X)$
defined in \Cref{lem:barycenter_and_c_concave_function}
has simple expression unlike its $c$-transform $f_1$.
Furthermore, thanks to the following lemma by Kim and Pass \cite[Lemma 3.1]{kim2015multi},
we conclude that $g_1$ is smooth.
This differential property of $g_1$ (to be used in \Cref{lem:Lipschitz_optimal_transport_map})
is crucial to prove the absolute continuity of Wasserstein barycenters.

\begin{lem}[Barycenters and cut loci,
	\citeinfo{Lemma 3.1 and proof of Theorem 6.1}{kim2015multi}]
	\label{lem:barycenter_discrete_measure_out_of_cut_locus}
	Let $(M, \matheuvm{g})$ be a complete Riemannian manifold.
	Given an integer $n \ge 1$,
	let $\lambda_i > 0, 1 \le i \le n$, be $n$ positive real
	numbers such that $\sum_{i=1}^n \lambda_i = 1$
	and let $x_i \in M, 1 \le i \le n$, be $n$ points of $M$.
	For $1 \le i \le n$, $x_i$ is out of the cut locus of
	any barycenters of $\sum_{i=1}^n \lambda_i \, \delta_{x_i}$.
\end{lem}

It will be shown in \Cref{lem:Lipschitz_optimal_transport_map} that the map
$\exp(-\nabla g_1)$ is an optimal transport map $F$ as stated in \Cref{thm:optimal_transport_manifold}.
Given the above \Cref{lem:barycenter_discrete_measure_out_of_cut_locus},
the following lemma further illustrates how to differentiate such maps.

\begin{lem}
	\label{lem:differentiate_smooth_concave_function}
	Let $(M, \matheuvm{g})$ be a complete Riemannian manifold.
	Fix an open set $U \subset M$, a point $x \in U$, and
	a $\mathcal{C}^2$ smooth function $\phi$ defined on $U$.
	Define $F: = \exp(-\nabla \phi)$ on $U$.
	Assume that the (fixed) point $y: = F(x)$ is out of the cut locus of $x$.
	If the two functions,  $\phi$ and  $d^2_{y} / 2$,
	have the same gradient at $x$, then
	\begin{equation}
		\label{equa:differentiate_smooth_concave_function}
				\di_x F = [\di_{- \nabla \phi(x)} \exp_x] \circ (\operatorname{Hess}_x d^2_{y} / 2 - \operatorname{Hess}_x \phi).
	\end{equation}
	In the above formula,
	\begin{enumerate}
		\item $\operatorname{Hess}_x$ denotes the Hessian operator at $x$
		      and its values are maps from $T_xM$ to $T_xM$;
		\item $\di_{- \nabla \phi(x)} \exp_x : T_{-\nabla \phi(x)}T_xM \rightarrow T_{y}M$
		      denotes the differential of the exponential map
		      $\exp_x: T_xM \rightarrow M$ at $-\nabla \phi (x)$;
		\item the composition is valid since
		      $T_{-\nabla \phi(x)} T_x M$ can be canonically identified with $T_xM$.
	\end{enumerate}
\end{lem}

\begin{proof}
	The formula (\ref{equa:differentiate_smooth_concave_function}) is already proven
	in \cite[Proposition 4.1]{cordero2001riemannian},
	whose proof can be simplified thanks to our assumptions. 
	Define $y : = F(x)$.
	By the assumption that $y$ is not in the cut locus of $x$,
	$\operatorname{Hess}_x d^2_{y}  /2$ is  well-defined.
	Shrink the neighborhood $U$ of $x$ if necessary so that
	for $(w, z) \in U \times U$, $w$ is not in the cut loci of $y$ and $z$
	\cite[(2) of Proposition 4.1 in Chapter III]{sakai1996riemannian}.
	Define the following function $g$ on $U \times U$,
	\[
		g(w,z) := \exp_{w}\left(-\nabla d_y^2(w) / 2 +
		\Pi_{z \rightarrow w}\left[\nabla d_y^2(z) / 2 - \nabla \phi (z) \right]\right),
	\]
	where $\Pi_{z \rightarrow w}: T_{z}M \rightarrow T_wM$ denotes the parallel
	transport of tangent vectors along the minimal geodesic from $z$ to $w$.
	For $z \in U$, since $\Pi_{z \rightarrow z}$
	is the identity map on $T_zM$, $g(z, z) = F(z)$.
	For $ w \in U$,
	$g(w, x) = \exp_w(-\nabla d_y^2(w) / 2) \equiv y$ is a constant,
	where we used the assumption $\nabla d^2_y(x)  / 2 = \nabla \phi (x)$ for the first equality
	and used that $w$ is not in the cut locus of $y$ for the second one.
	It follows that
	\begin{align}
		\label{equa:d_xF_first_line}
		\di_x F & = \partial_w g(x, x) + \partial_z g(x, x) = \partial_z g(x, x)                                                 \\
		\label{equa:d_xF_second_line}
		        & = [\di_{- \nabla \phi(x)} \exp_x ] \circ (\operatorname{Hess}_x d^2_y / 2 - \operatorname{Hess}_x \phi),
	\end{align}
	where we used $F(z) = g(z, z)$,
	the chain rule and $g(w, x) \equiv y$
	in the line (\ref{equa:d_xF_first_line}),
	and applied the relation between covariant derivative and parallel transport
	\cite[Corollary 4.35]{lee2018introduction}
	to the definition of Hessian
	as the covariant derivative of gradients \cite[Proposition 2.2.6]{petersen2016riemannian}
	in the line (\ref{equa:d_xF_second_line}).
\end{proof}


\subsubsection{Lipschitz continuous optimal transport maps of Wasserstein barycenters}
\label{sec:Lipschitz_continuous_transport_map}

To better illustrate our approach towards the absolute continuity
of Wasserstein barycenters of finitely many measures,
we recall the following result corresponding to the case of two measures.

\begin{prop}[Regularity of displacement interpolations,
		\citeinfo{Theorem 8.5, Theorem 8.7}{villani2009optimal}]
	\label{prop:Lipschitz_continuity_displacement_interpolation}
	Let $(M, \matheuvm{g})$ be a complete Riemannian manifold.
	Let $t \in [0, 1] \mapsto \mu_t \in \mathcal{W}_2(M)$ be a minimal geodesic in the
	Wasserstein space $\mathcal{W}_2(M)$ such that
	both $\mu_0$ and $\mu_1$ have compact support.
	For any $0 < \lambda < 1$, $\mu_{\lambda}$ is the barycenter
	of $(1-\lambda) \delta_{\mu_0} + \lambda\,\delta_{\mu_1}$.
	The optimal transport map from $\mu_\lambda$ to $\mu_0$
	is Lipschitz continuous,
	and it follows that $\mu_\lambda$ is absolutely continuous
	provided that $\mu_0$ is absolutely continuous.
\end{prop}

The Lipschitz continuity in
\Cref{prop:Lipschitz_continuity_displacement_interpolation}
can be shown as a consequence of Mather's shortening lemma
\cite[Chapter 8]{villani2009optimal}.
Since Lipschitz maps send Lebesgue negligible sets to Lebesgue negligible sets,
the last statement on absolute continuity follows.
Another approach to the Lipschitz continuity
is given by Bernard and Buffoni \cite{bernard2007optimal}
using the Hamilton-Jacobi equation,
which is generalized to non-compact settings by Fathi and Figalli \cite{fathi2010optimal}.
For the case when both $\mu_0$ and $\mu_1$ are absolutely continuous
measures on Euclidean spaces, McCann \cite[Proposition 1.3]{mccann1997convexity}
presented a concise proof of the Lipschitz continuity.
See relevant references in Villani \cite[Bibliographical notes of Chapter 8]{villani2009optimal}.
The goal of this subsection is to generalize \Cref{prop:Lipschitz_continuity_displacement_interpolation}.

We deduce the following Lipschitz continuity from the $c$-concave
functions defined in \Cref{lem:barycenter_and_c_concave_function},
which are related to barycenter selection maps
and thus barycenters of $\lambda_1\,\delta_{\mu_1} + \sum_{i=2}^n \lambda_i\, \delta_{\delta_{x_i}}$.
Recall that a measurable barycenter selection map $B: M^n \rightarrow M$
(\Cref{thm:measurable_selection_theorem})
sends $(x_1, \ldots, x_n) \in M^n$ to a barycenter of $\sum_{i=1}^n \lambda_i\, \delta_{x_i}$.
In the following propositions, the constant $\lambda_i, 1 \le i \le n$
for $B$ are given in the context.

\begin{lem}[Lipschitz continuous maps $F= \exp(-\nabla g_1)$]
	\label{lem:Lipschitz_optimal_transport_map}
	Let $(M, \matheuvm{g})$ be a complete Riemannian manifold.
	Given an integer $n \ge 2$, let $\lambda _ { i } > 0,1 \le i \le n$, be $n$ positive real numbers
	such that $\sum _ { i = 1 } ^ { n } \lambda _ { i } = 1$.
	Fix a non-empty compact subset $X \subset M$ and
	a point $\boldsymbol x^\prime = (x_2, \ldots, x_n) \in M^{n-1}$.
	Denote by $Y$ the compact set of all barycenters of $\sum_{i=1}^n \lambda_i\, \delta_{x_i}$
	when $x_1$ runs through $X$.
	Define the function $g_1: y \in M  \mapsto - 1/\lambda_1 \sum_{i = 2}^n \lambda_i\, c(y, x_i)$.
	It is smooth in a neighborhood
	of $Y$ and thus $F := \exp(-\nabla g_1) : Y \rightarrow M$
	is a well-defined Lipschitz continuous function.
	We have $F(Y) = X$ and the following characterization of $F$:
	\begin{equation}
		\label{equa:g_1_and_barycenter}
		z \in Y \text{ and } x_1 = F(z) \iff
		x_1 \in X \text{ and } z \text{ is a barycenter of }
		\sum_{i=1}^n \lambda_i\, \delta_{x_i}.
	\end{equation}
	Given a measure $\mu_1 \in \mathcal{W}_2(M)$ with support $X$
	and a measurable barycenter selection map $B: M^n \rightarrow M$,
	$\widebar \mu : = B_{\#}(\mu_1 \otimes \delta_{x_2} \otimes
		\cdots \otimes \delta_{x_n})$ is a barycenter of
	$\lambda_1\,\delta_{\mu_1} + \sum_{i=2}^n \lambda_i \, \delta_{\delta_{x_i}}$
	and $(\operatorname{Id}, F)_{\#} \widebar{\mu}$
	is an optimal transport plan between $\widebar \mu$ and $\mu_1$.
\end{lem}

\begin{proof}
	\Cref{lem:barycenter_discrete_measure_out_of_cut_locus} implies
	the differential property of $g_1$ and thus the Lipschitz continuity of $F$.
	Since $g_1$ restricted to $Y$ is a $c$-concave function
	(\Cref{lem:barycenter_and_c_concave_function})
	and $\nabla g_1$ exists on $Y$,
	by defining $g_1^c: x \in X \mapsto \min_{y \in Y}\{ c(x, y) - g_1(y) \}$,
	a well-known property of $c$-concave functions proven
	by McCann \cite[Lemma 7]{mccann2001polar} shows that
	\begin{equation*}
		\label{equa:c_concave_g_1_equality}
		z \in Y \text{ and } x_1 = \exp(-\nabla g_1)(z) =
		F(z) \iff
		(x_1, z) \in X \times Y \text{ and } g_1^c(x_1) + g_1(z) = c(x_1, z).
	\end{equation*}
	Note that though McCann's lemma is proven for compact manifolds,
	the arguments of its proof only depend on the existence of gradient $\nabla g_1$ and
	the compactness of $X$ and $Y$.
	For $x_1 \in X$, we have
	$g_1^c(x_1) = 1 / \lambda_1 \inf_{w \in M} \sum_{i=1}^n \lambda_i\, c(w, x_i)^2$ 
	(\Cref{lem:barycenter_and_c_concave_function}) and thus
	\[
		z \in Y \text{ and }
		g_1^c(x_1) + g_1(z) = c(x_1, z) \iff
		\sum_{i=1}^n \lambda_i \, d_{\matheuvm{g}}(z, x_i)^2 = \inf_{w \in M}
		\sum_{i=1}^n \lambda_i\, d_{\matheuvm{g}}(w, x_i)^2,
	\]
	which implies the characterization (\ref{equa:g_1_and_barycenter}).
	$F(Y) = X$ follows from (\ref{equa:g_1_and_barycenter}) and the definition of $Y$.

	Since $\gamma:= \mu_1 \otimes \delta_{x_2} \otimes \cdots \otimes \delta_{x_n}$
	is the only measure on $M^{n}$ with marginals $\mu_1, \delta_{x_2}, \ldots, \delta_{x_n}$
	in this order,
	it is the (unique) multi-marginal optimal transport plan of its marginals.
	\Cref{prop:construction_Wasserstein_barycenter} shows that
	$\widebar \mu = B_{\#} \gamma$ is a Wasserstein barycenter.
	Denote by $p_1: M \times M^{n-1} \rightarrow M$ the canonical projection map.
	Since $ p_1(x_1, \boldsymbol x^\prime) = x_1 = F(B(x_1, \boldsymbol x^\prime))$ for $x_1 \in X$
	by the characterization (\ref{equa:g_1_and_barycenter}),
	\Cref{prop:construction_Wasserstein_barycenter} shows that
	$(B, p_1)_{\#} \gamma = (B, F \circ B)_{\#} \gamma =
	(\operatorname{Id}, F)_{\#} \widebar \mu$
	is an optimal transport plan between $\widebar \mu$ and $\mu_1$.
\end{proof}

\Cref{lem:Lipschitz_optimal_transport_map} implies that
any barycenter selection map on $X \times \{ \boldsymbol x^\prime \}$
is injective.
The following lemma by Kim and Pass \cite[Lemma 3.5]{kim2015multi}
generalizes this injectivity,
and it will help us to generalize \Cref{lem:Lipschitz_optimal_transport_map}.

\begin{lem}
	\label{lem:barycenter_injective_on_support}
	Let $(M, \matheuvm{g})$ be a complete Riemannian manifold.
	Given an integer $n \ge 2$,
	let $\lambda _ { i } > 0,1 \le i \le n$, be $n$ positive real numbers
	such that $\sum _ { i = 1 } ^ { n } \lambda _ { i } = 1$
	and let $\mu_i \in \mathcal{W}_2(M), 1 \le i \le n$, be $n$ measures
	with compact support.
	If $\gamma$ is a multi-marginal optimal transport plan
	with marginals $\mu_1, \ldots, \mu_n$, then
	\[
		\boldsymbol x, \boldsymbol y \in \operatorname{supp}(\gamma),\quad
		\boldsymbol x \neq \boldsymbol y \implies
	 \operatorname{bary}(\{ \boldsymbol x \}) \cap
	 \operatorname{bary}(\{ \boldsymbol y \}) = \emptyset,
 \]
	 where
	$\operatorname{bary}(\{ x_1, \ldots, x_n \})$ is the set
	of barycenters of $\sum_{i=1}^n \lambda_i \, \delta_{x_i}$.
\end{lem}

To avoid being lengthy,
we skip the proof of above lemma \cite[Lemma 3.5]{kim2015multi},
which is based on $c$-cyclical monotonicity and
\Cref{lem:barycenter_discrete_measure_out_of_cut_locus}.
Though the proof in the given reference 
is for the case when $\lambda_1 = \cdots = \lambda_n = 1 / n$, there
is no essential difficulty to apply it to the stated case
\cite[proof of Theorem 6.1]{kim2015multi}.
The following proposition constructs an optimal transport map
from $\widebar \mu: = B_{\#} \gamma$ to $\mu_1$ when $\mu_i, 2 \le i \le n$ are discrete measures
and thus generalizes \Cref{lem:Lipschitz_optimal_transport_map}.
The optimal transport map may fail to be a Lipschitz map, but it is
a disjoint union of Lipschitz maps.
Recall that given (at most) countably many disjoint subsets 
$Y_j \subset M, j \in J \subset \mathbb{N}$ with functions $F_j: Y_j \rightarrow M$,
the \emph{disjoint union} $F$ of $F_j, j \in J$ is the function defined
on $\cup_{j \in J} F_j$ such that $F |_{Y_j} = F_j$.
We shall use conditional measures (\Cref{defn:regular_conditional_measure})
to deduce further conclusions from $F_j$'s Lipschitz continuity.

\begin{prop}
	\label{prop:absolutely_continous_barycenter_discrete_marginals}
	Let $(M, \matheuvm{g})$ be an $m$-dimensional complete Riemannian manifold.
	Given an integer $n \ge 2$,
	let $\lambda _ { i } > 0,1 \le i \le n$, be $n$ positive real numbers
	such that $\sum _ { i = 1 } ^ { n } \lambda _ { i } = 1$.
	Let $\mu_{1} \in \mathcal{W}_2(M)$ be a measure with compact support and
	let $\mu_i \in \mathcal{W}_2(M), 2 \le i \le n$, be
	$n - 1$ discrete measures, i.e., measures concentrated on
	at most countably many points.
	Given a multi-marginal optimal transport plan $\gamma$ of $\mu_1, \ldots, \mu_n$
	in this order and a measurable barycenter selection map $B : M^n \rightarrow M$,
	the measure $\widebar \mu := B_{\#} \gamma$ is a barycenter
	of $\sum_{i=1}^n \lambda_i \, \delta_{\mu_i}$.
	This barycenter $\widebar \mu$ assigns full mass to a disjoint union of 
	at most countably many compact sets, and on each of them
	 \Cref{lem:Lipschitz_optimal_transport_map} defines a
	 Lipschitz continuous map with compact subset $X \subset M$
	 and point $\boldsymbol{x^\prime} \in M^{n-1}$
	 such that $X \times \{\boldsymbol x^\prime \}$ is contained in the support of $\gamma$.
	Denote by $F$ the disjoint union of the Lipschitz maps.
	$(\operatorname{Id}, F)_{\#} \widebar \mu$ is an optimal transport
	plan between $\widebar \mu$ and $\mu_1$.

	For positive real numbers $\delta, \epsilon > 0$, we define the set
	\[
		\mathcal{E}_{\epsilon, \delta}
		: = \left\{ \mu \in \mathcal{W}_2(M) \mid
		\forall\, N \in \mathcal{B}(M),\, \operatorname{Vol}(N) < \delta
		\implies \mu(N) \le \epsilon \right\}.
	\]
	If there is a common Lipschitz constant $C$ of the
	Lipschitz maps, then $\mu_1 \in \mathcal{E}_{\epsilon, \delta} \implies
		\widebar \mu \in \mathcal{E}_{\epsilon, \delta / C^m}$.
\end{prop}

\begin{proof}
	\Cref{prop:construction_Wasserstein_barycenter} shows that
	$\widebar \mu $ is a Wasserstein barycenter.
	Let us reveal more details of $\gamma$.
	Denote by $p_1$ and $p_2$ the canonical projections
	sending $\boldsymbol x = (x_1, \boldsymbol x^\prime) \in M \times M^{n-1}$ to $x_1 \in M$
	and $\boldsymbol x^\prime \in M^{n-1}$ respectively.
	The measure $\pi: = {p_2}_{\#} \gamma$ on $M^{n-1}$ is discrete
	since its marginals $\mu_2, \ldots, \mu_n$ are so.
	Denote by $\{  \boldsymbol x_j^\prime \}_{j \in J} \subset M^{n-1}$ ($J \subset \mathbb{N}$)
	the set of all atoms of $\pi$.
	For each $j \in J$, we introduce the following definitions.
	Define $\pi_j := \pi(\{  \boldsymbol x^\prime_j \}) > 0$ and
	define $X_j: = p_1(\operatorname{supp} \gamma \cap (M \times \{ \boldsymbol x_j^\prime\}))$.
	Applying \Cref{lem:Lipschitz_optimal_transport_map} to
	$X_j$ and $ \boldsymbol x_j^\prime \in M^{n-1}$,
	we obtain a compact set $Y_j$ and a Lipschitz continuous map
	$F_j: Y_j \rightarrow M$ such that $F_j(Y_j) = X_j$.
	Since $\pi$ assigns full mass to the union $\bigcup_{j \in J} \{  \boldsymbol x_j \}$,
	$\gamma$ assigns full mass to the union $\bigcup_{j \in J} X_j \times \{  \boldsymbol x_j^\prime \}$.
	As in \Cref{lem:barycenter_injective_on_support},
	we denote by $\operatorname{bary}(\{ (x_1, \ldots, x_n) \})$ the set of barycenters
	of $\sum_{l=1}^n \lambda_l \, \delta_{x_l}$.


	We claim that $Y_i \cap Y_k = \emptyset$ for two different indices $i, k \in J$.
	Indeed, if $z \in Y_i \cap Y_k$ for $i, k \in J$,
	then by the characterization of $F_i, F_k$ in \Cref{lem:Lipschitz_optimal_transport_map},
	$z \in \operatorname{bary}(\{(F_i(z), \boldsymbol x_i^\prime )\}) \cap
	\operatorname{bary}(\{(F_k(z), \boldsymbol x_k^\prime )\})$.
	Since $\operatorname{supp} \gamma = \bigcup_{j \in J} X_j \times \{ \boldsymbol x_j^\prime \}$
	and $F_j(Y_j) = X_j$,
	\Cref{lem:barycenter_injective_on_support} forces that
	$\boldsymbol x^\prime_i = \boldsymbol x^\prime_k$ and thus $i = k$.
	Define $F$ as the disjoint union of $F_j, j \in J$, i.e., 
	$F |_{Y_j} = F_j$.
	Since $ p_1(x, \boldsymbol x^\prime_j) = x = F(B(x, \boldsymbol x^\prime_j)) $
	for $x \in X_j$, \Cref{prop:construction_Wasserstein_barycenter} implies that
	$(B, p_1)_{\#} \gamma = (B, F \circ B)_{\#} \gamma = (\operatorname{Id}, F)_{\#} \widebar \mu$
	is an optimal transport plan between $\widebar \mu$ and $\mu_1$.
	Since $\bigcup_{j \in J} Y_j$ is the domain of $F$
	and $F_{\#} \widebar \mu = \mu_1$,
	$\widebar \mu$ assigns full mass to
	a union of at most countably many compact sets that satisfies our description.

	We claim that $\mu_1( X_i \cap X_k) = 0$ for two different indices $i, k \in J$.
	Consider the conditional measure 
	such that $\diff \gamma(\boldsymbol x) = \gamma(\diff \boldsymbol x,
		\boldsymbol x^\prime) \diff \pi(\boldsymbol x^\prime)$.
	For $j \in J$, define $\nu_j : = \frac{1}{\pi_j} \mu_1 |_{X_j}$
	and $\widebar \nu_j : = B_{\#} \gamma(\cdot, \boldsymbol x_j^\prime)$.
	Note that for $j \in J$ and $\boldsymbol R \in \mathcal{B}(M^n)$,
	$\gamma[\boldsymbol R \cap (M \times \{\boldsymbol x_j ^ \prime\})] =
		\gamma(\boldsymbol R, \boldsymbol x_j^\prime)\, \pi_j$
	by \Cref{defn:regular_conditional_measure},
	so $\gamma(\cdot, \boldsymbol x_j^\prime)$ is concentrated on
	$X_j \times  \{ \boldsymbol x_j^\prime \}$.
	Since $\gamma$ assigns full mass to the union $\bigcup_{j \in J} X_j \times \{  \boldsymbol x_j^\prime \}$,
	we obtain the following equality by choosing $\boldsymbol R$ of the form $A \times M^{n-1}$
	with $A \in \mathcal{B}(M)$,
	\[
		\gamma(A \times M^{n-1}, \boldsymbol x_j^\prime) =
		\frac{1}{\pi_j} \gamma[A \times \{\boldsymbol x_j ^ \prime\}]
		= \frac{1}{\pi_j} \gamma[(A \cap X_j) \times M^{n-1}]
		= \frac{1}{\pi_j} \mu_1 (A \cap X_j),
	\]
	which implies that the first marginal of $\gamma(\cdot, \boldsymbol x_j^\prime)$ is $\nu_j$
	as $A$ is arbitrarily chosen.
	Furthermore, for a measurable map $f: M^n \rightarrow M$,
	\[
		\forall N \in \mathcal{B}(M), \quad
		[f_{\#} \gamma](N) =
		\gamma(f^{-1}(N))
		= \sum_{j \in J} \gamma(f^{-1}(N), \boldsymbol x_j^\prime)\,
		\pi_j
		= \sum_{j \in J} [f_{\#} \gamma(\cdot, \boldsymbol x_j^\prime)](N)\,
		\pi_j.
	\]
	Taking $f = p_1$ and $f = B$, we obtain
	$\mu_1 = \sum_{j \in J} \pi_j\, \nu_j$ and
	$\widebar \mu = \sum_{j \in J} \pi_j\, \widebar \nu_j$.
	Hence, given $i \in J$, $\mu_1(X_i) = \sum_{j \in J} \mu_1 |_{X_j}(X_i)$
	and thus $\mu_1(X_i \cap X_k) = 0$ for $k \in J$ different from $i$.

	Assume that $C$ is a common Lipschitz constant of all $F_j, j \in J$.
	For any Borel set $N \in \mathcal{B}(M)$,
	there exist Borel sets $W_j, j \in J$ such that
	$F_j(N \cap Y_j) \subset  W_j \subset X_j$ and
	$\operatorname{Vol}(W_j) \le C^m \operatorname{Vol}(N \cap Y_j)$
	\cite[Proposition 12.6, Proposition 12.12, Remark after Proposition 12.12]{taylor2006measure}
	(c.f.\@ \cite[Proof of Theorem 8.7]{villani2009optimal}).
	For $j \in J$, since 
	$\gamma(\cdot, \boldsymbol x^\prime_j)$ is the product measure of its marginals,
	\Cref{lem:Lipschitz_optimal_transport_map} shows
	that ${F_{j}}_{\#} \widebar \nu_j = \nu_j$.
	It follows that $\widebar \nu_j(N \cap Y_j) \le \widebar \nu_j(F_j^{-1}(W_j)) \le \nu_j(W_j)$
	for $j \in J$ and thus
	\begin{equation}
		\label{equa:Lipschitz_comparison_by_parts}
		\widebar \mu (N)
		=  \sum_{j \in J} \pi_j \, \widebar \nu_j(N \cap Y_j)
		\le \sum_{j \in J} \pi_j \frac{1}{\pi_j} \mu_1 |_{X_j} (W_j)
		=  \sum_{j \in J} \mu_1( W_j)
		=  \mu_1( \bigcup_{j \in J} W_j),
	\end{equation}
	where we used $W_j \subset X_j$
	and $\mu_1(X_i \cap X_k) = 0$ if $i \neq k \in J$.
	Since $Y_j, j \in J$ are disjoint,
	$\operatorname{Vol}(\bigcup_{j \in J} W_j)
		\le C^m \sum_{j \in J} \operatorname{Vol}(N \cap Y_j) \le C^m \operatorname{Vol}(N)$.
	Assuming that $\mu_1 \in \mathcal{E}_{\epsilon, \delta}$,
	then for any $N \in \mathcal{B}(M)$ with $\operatorname{Vol}(N) < \delta / C^m$,
	we have $\operatorname{Vol}(\bigcup_{j \in J} W_j) < \delta$
	and thus $ \mu_\mathbb{P} (N) \le \mu_1( \bigcup_{j \in J} W_j) \le \epsilon$ by
	(\ref{equa:Lipschitz_comparison_by_parts}).
	Therefore, the asserted implication
	$\mu_1 \in \mathcal{E}_{\epsilon, \delta} \implies
		\mu_\mathbb{P} \in \mathcal{E}_{\epsilon, \delta / C^m}$ is proven.
\end{proof}

\begin{rmk}
	Figuratively speaking, the sets $X_j, j \in J$ create a tiling
	of the support of $\mu_1$ and the points $\boldsymbol x^\prime_j, j \in J$
	pull them apart (via barycenter selection maps)
	into disjoint sets $Y_j, j \in J$, which
	contain different pieces of the support of $\widebar \mu$ separately.
\end{rmk}

\subsubsection{Proof of absolute continuity}

Consider the probability measure $\mathbb{P} = \sum_{i=1}^n \lambda_i\, \delta_{\mu_i}$
with positive real numbers $\lambda_i$ and compactly supported measures
$\mu_i \in \mathcal{W}_2(M)$.
We can approximate each $\mu_i$ for $2 \le i \le n$ with discrete measures
to apply \Cref{prop:absolutely_continous_barycenter_discrete_marginals}.
If $\mu_1$ is absolutely continuous, then $\mathbb{P}$ has a unique
barycenter $\widebar \mu$, which is approximated by
the barycenters of the approximating sequence (\Cref{thm:law_of_large_numbers_Wasserstein_barycenter}).
Recall that the sets $\mathcal{E}_{\epsilon, \delta}$
(\Cref{lem:measurable_set_absolutely_continous_measures})
fully characterize absolutely continuous measures and are closed with respect to weak convergence.
Hence, to prove the absolute continuity of $\widebar \mu$,
it remains to find a common Lipschitz constant $C$ for $F$
defined as in \Cref{lem:Lipschitz_optimal_transport_map}
valid for any element of the whole approximating sequence.
Then we get the result thanks to \Cref{prop:absolutely_continous_barycenter_discrete_marginals}.
Note that the domain $Y$ of $F$ is at least varying along the approximating sequence,
so the existence of $C$ is not simply a direct consequence of compactness. More precisely,
we shall prove:

\begin{thm}[Absolute continuity of the barycenter of $\sum_{i=1}^n \lambda_i \,\delta_{\mu_i}$]
	\label{thm:absolute_continuity_discrete}
	Let $(M, \matheuvm{g})$ be a complete Riemannian manifold.
	Given an integer $n \ge 2$,
	let $\lambda _ { i } > 0,1 \le i \le n$, be $n$ positive real numbers
	such that $\sum _ { i = 1 } ^ { n } \lambda _ { i } = 1$ and
	let $\mu_i \in \mathcal{W}_2(M), 1 \le i \le n$, be $n$ probability
	measures with compact support.
	If $\mu_1$ is absolutely continuous,
	then the unique barycenter $\widebar{\mu}$ of $\sum_{i=1}^n \lambda_i \, \delta_{\mu_i}$
	is absolutely continuous with compact support.
\end{thm}

\begin{proof}
	The uniqueness and compact support of $\widebar \mu$ follow from
	\Cref{sec:existence_and_uniqueness} and \Cref{lem:compact_barycenter}.
	We approximate each $\mu_i$ for $2 \leq i \leq n $ in $(\mathcal{W}_2(M), W_2)$ by
	a sequence of discrete measures $\{\mu_i^{j}\}_{ j \ge 1}$
	whose supports are contained in the compact support of $\mu_i$.
	Then $\mathbb{P}_j := \lambda_1 \delta_{\mu_1} + \sum_{i = 2}^n
		\lambda_i\, \delta_{\mu_i^j}$ converges to $\mathbb{P}$
	in $\mathcal{W}_2(\mathcal{W}_2(M))$.
	By the law of large numbers for Wasserstein barycenters
	(\Cref{thm:law_of_large_numbers_Wasserstein_barycenter}),
	the unique barycenter $\widebar{\mu}_j$ of $\mathbb{P}_j$
	converges in $(\mathcal{W}_2(M), W_2)$ to the unique barycenter $\widebar{\mu}$ of $\mathbb{P}$.

	Denote by $\gamma_j$ a multi-marginal
	optimal transport plan of marginal measures $\mu_1, \mu_2^j, \ldots, \mu_n^j$ in this order.
	Fix an index $j$, a non-empty compact subset $X \subset M$ and
	a point $\boldsymbol x^\prime :=(x_2, \ldots, x_n) \in M^{n-1}$
	such that $X \times \{ \boldsymbol x^\prime \} \subset \operatorname{supp} \gamma_j$.
	Applying \Cref{lem:Lipschitz_optimal_transport_map} to $X$ and $\boldsymbol x^\prime$,
	we obtain a Lipschitz continuous function $F = \exp(- \nabla g_1)$ on
	a compact set $Y$.
	We claim that there exists a Lipschitz constant
	$C$ of $F$ on $Y$ independent of $j, X$ and $\boldsymbol x^\prime$.
	Recall that $g_1(y) : = - 1 / \lambda_1 \sum_{i = 2}^n \lambda_i\, c(y, x_i)$
	is smooth in a neighborhood of $Y$.
	Given $z \in Y$,
	since $z$ is a barycenter of $\sum_{i=1}^n \lambda_i\, \delta_{x_i}$ (\Cref{lem:Lipschitz_optimal_transport_map})
	with $x_1: = F(z)$,
	we have $\sum_{i=1}^n \nabla d^2_{x_i}(z) = 0$ thanks to \Cref{lem:barycenter_discrete_measure_out_of_cut_locus}
	and thus $\nabla d^2_{x_1} / 2 (z) = \nabla g_1(z)$.
	Moreover, \Cref{lem:barycenter_discrete_measure_out_of_cut_locus} enables
	us to apply \Cref{lem:differentiate_smooth_concave_function}, which implies
	\begin{align}
		\di_z F =
		\di_z \exp(- \nabla g_1)
		 & = [\di_{- \nabla g_1(z)} \exp_z]
		\circ (\operatorname{Hess}_z d^2_{x_1} / 2 - \operatorname{Hess}_z g_1) \nonumber \\
		\label{equa:differential_exp_g_1}
		 & = [\di_{- \nabla g_1(z)} \exp_z]
		\circ \frac{1}{2 \lambda_1} \sum_{i=1}^n \lambda_i \operatorname{Hess}_z d^2_{x_i}.
	\end{align}
	In (\ref{equa:differential_exp_g_1}),
	$\sum_{i=1}^n \lambda_i\,\operatorname{Hess}_z d^2_{x_i}$ is positive semi-definite since
	$z$ reaches the global minimum of $ w \in M \mapsto \sum_{i=1}^n \lambda_i\, d_{\matheuvm{g}}(w, x_i)^2$.
	We now bound (\ref{equa:differential_exp_g_1}) using compactness as follows.
	By \Cref{lem:compact_barycenter} and our construction of $\mathbb{P}_j$,
	the union of the supports of $\widebar{\mu}$, $\mu_i$, $\widebar{\mu}_j$ and $\mu_i^j$
	for $ 1 \le i \le n$ and $j \ge 1$ is compact.
	Hence, independent of $z$, $j$ and $\boldsymbol x^\prime$,
	$\di_{- \nabla g_1(z)} \exp_z$ is uniformly bounded (in norm)
	and $\sum_{i=1}^n \lambda_i\,\operatorname{Hess}_z d^2_{x_i}$ is uniformly
	bounded from above by the Rauch comparison theorem for Hessians
	of distance functions, which is applicable here and provides a constant upper bound
	thanks to the compactness,
	see \cite[Lemma 3.12 and Corollary 3.13]{cordero2001riemannian}
	or \cite[Theorem 6.4.3]{petersen2016riemannian}.
	This shows the existence of the claimed Lipschitz constant $C$.
	We remark that the absolute continuity of $\mu_1$ is not needed for
	the existence of $C$.

	Applying \Cref{prop:absolutely_continous_barycenter_discrete_marginals} to
	measures $\mu_1, \mu_2^j, \ldots, \mu_n^j$, we have for $\epsilon, \delta > 0$,
	$\mu_1 \in \mathcal{E}_{\epsilon, \delta} \implies
		\widebar \mu_j \in \mathcal{E}_{\epsilon, \delta / C^m}$ since
	$\widebar \mu_j$ is the unique barycenter of $\mathbb{P}_j$.
	As $\widebar \mu_j$ converges to $\widebar \mu$ weakly,
	\Cref{lem:measurable_set_absolutely_continous_measures}
	shows that all measures $\widebar \mu_j$ for $j \ge 1$ and $\widebar \mu$ are
	absolutely continuous since $\mu_1$ is so.
\end{proof}


\section{Hessian equality for Wasserstein barycenters}
\label{sec:hessian_equality}

In this section, we prove the Hessian equality
for Wasserstein barycenters of finitely many
measures (\Cref{thm:hessian_equality_Wasserstein_barycenter}).
A similar property is named as the \emph{second order balance} (inequality)
by Kim and Pass \cite[Theorem 4.4]{kim2017wasserstein},
but being an equality instead of an inequality
is crucial for our proof of \Cref{prop:entropy_estimation}.
Let us take a special case to illustrate this equality.
Consider the reduced case in \Cref{lem:Lipschitz_optimal_transport_map}.
Namely, take $n$ positive numbers $\lambda_i > 0$ such that $\sum_{i=1}^n \lambda_i = 1$
and denote by $\widebar \mu$ the barycenter of $\sum_{i=1}^n \lambda_i\, \delta_{\mu_i}$,
where $\mu_1$ is absolutely continuous with compact support
and $\mu_i = \delta_{x_i}, 2 \le i \le n$, are Dirac measures.
Let us set $\phi_1(z) : =g_1(z): = - 1 / \lambda_1 \sum_{i=2}^n \lambda_i\, c(z, x_i)$
and $\phi_i(z) : = c(z, x_i), 2 \le i \le n$. Thanks to \Cref{lem:compact_barycenter} and
\Cref{lem:barycenter_discrete_measure_out_of_cut_locus},
if $z$ is in the support of $\widebar{\mu}$,
then $z$ is not in the cut locus of any $x_i$,
which implies $\exp(- \nabla \phi_i)_{\#} {\widebar \mu} = \mu_i$ for $2 \le i \le n$.
Besides, by definition of the $\phi_i$'s, $\sum_{i=1}^n \lambda_i\, \phi_i \equiv 0$; therefore $\sum_{i=1}^n \lambda_i \nabla \phi_i (z) = 0$. Consequently we get $\sum_{i=1}^n \lambda_i \operatorname{Hess}_z \phi_i = 0$,
which is the Hessian equality we are referring to.



\subsection{Approximate differentiability}

We justify the definition of \emph{density point} for Riemannian manifolds
by comparing it to its usual Euclidean counterpart.
Denote by $\widebar B(x, r)$ the closed metric ball centered at $x$ with radius $r$.

\begin{lem}[Density points]
	\label{lem:density_points_mainfolds}
	Let $(M, {\matheuvm{g}})$ be a Riemannian manifold and
	let $A$ be a Borel subset of $M$.
	We call $x \in M$ a density point of $A$ (with respect to $\operatorname{Vol}$) if
	\[
		\lim_{r \downarrow 0} \frac{\operatorname{Vol}[\widebar B(x, r) \setminus A]}
		{\operatorname{Vol}[\widebar B(x, r)]} = 0.
	\]
	This definition is equivalent to the standard one with respect to the Lebesgue measure after pulling $x$ and $A$ back to the Euclidean space through an arbitrary chart around $x$.
	In particular, almost every point of $A$ is a density point of $A$ with respect to $\operatorname{Vol}$.
\end{lem}

\begin{proof}
	Denote by $m$ the dimension of $M$.
	In a (smooth) local chart $(\varphi,U)$ with $U$ a small enough neighborhood of $x \in M$,
	the metric of $M$ is bounded (from both sides) by the metric of $\mathbb{R}^m$
	with constant scales $0 < c_1 < c_2$.
	It follows that
	$ c_1^m\Leb^m(\varphi(N)) \le \operatorname{Vol}(N) \le c_2^m\Leb^m(\varphi(N))$
	for any measurable subset $N \subset U$
	\cite[Proposition 12.6 and 12.7]{taylor2006measure}.
	Hence, $x$ is a density point of $A$ if and only if
	\begin{equation}
		\label{equa:local_chart_density_point}
		\lim_{r \downarrow 0} \frac{\Leb^m[\varphi(\widebar B(x, r)) \setminus \varphi(A \cap U)]}
		{\Leb^m[\varphi(\widebar B(x, r))]} = 0.
	\end{equation}
	Applying again the relation between the metric of $M$ and the metric of $\mathbb{R}^m$,
	for any $r > 0$, we have $\widebar B(\varphi(x), c_1\,r)
	\subset \varphi(\widebar B(x, r)) \subset \widebar B(\varphi(x), c_2\,r)$.
	Therefore, (\ref{equa:local_chart_density_point}) is equivalent to
	that $\varphi(x)$ is a density point of $\varphi(A)$ with respect to $\Leb^m$.
\end{proof}

We now recall the definition of \emph{approximate derivatives} first on Euclidean space 
(see \cite[5.8(v)]{bogachev2007measure} and \cite[3.1.2]{federer2014geometric}
for more detailed discussions),
then on manifolds.

\begin{defn}[Approximate derivatives on Euclidean spaces]
	\label{defn:approximate_derivative_Euclidean_spaces}
	Let $m, n \ge 1$ be two positive integers.
	Given a function $F: \Omega \rightarrow \mathbb{R}^n$ defined on a subset $\Omega$ of $\mathbb{R}^m$,
	$l \in \mathbb{R}^n$ is an \emph{approximate limit} of $F$ at a point
	$x \in \mathbb{R}^m$, for which we write $l = \operatorname{ap} \lim_{y \rightarrow x} F(y)$,
	if there exists a Borel set $\Omega_x \subset \Omega$ such that
	$x$ is a density point of $\Omega_x$ and $\displaystyle \lim_{y \in \Omega_x, y \rightarrow x} F(y) = l$.
	The approximate derivatives of $F$ are defined via the approximate limits
	of its difference quotients as follows.

	A linear map $L: \mathbb{R}^m \rightarrow \mathbb{R}^n$
	is called the \emph{approximate derivative} of a function $F: \Omega \rightarrow \mathbb{R}^n$
	at a point $x \in \Omega \subset \mathbb{R}^m$ if
	\begin{equation}
		\label{equa:approximate_derivative}
		\operatorname{ap} \lim_{y \rightarrow x}
		\frac{|F(y) - F(x) - L(y - x)|}{|y - x|} = 0.
	\end{equation}
	The approximate derivative $L$ will be denoted by $\operatorname{ap} \di_x F$.
\end{defn}

The previous definition can be extended to the Riemannian setting as follows:

\begin{lem}[Approximate derivatives on manifolds]\label{defn:approximate_derivative_manifolds}
	Let $(M, {\matheuvm{g}})$ be an $m$-dimensional Riemannian manifold $M$
	and let $ f:  A \rightarrow \mathbb{R}^n$ be a function
	defined on a subset $A$ of $M$.
	Given an arbitrary local chart $(\varphi, U)$ around a point $x \in A$,
	$f$ is said to be approximately differentiable at $x$ if
	the approximate derivative $\operatorname{ap} \di_{\varphi(x)}
		[ f \circ \varphi^{-1} |_{\varphi( A \cap U)}]$ exists.
	The approximate derivative of $f$ at $x$ is then defined as
	\[
		\operatorname{ap} \di_x f:=
		\operatorname{ap} \di_{\varphi(x)}  [f \circ \varphi^{-1} |_{\varphi(A \cap U)}] 
		\circ \di_x \varphi:
		T_xM \rightarrow \mathbb{R}^n,
	\]
	where $\di_x \varphi: T_x M \rightarrow T_{\varphi(x)} \mathbb{R}^m$
	denotes the differential map of $\varphi$ at $x$
	and the tangent space $T_{\varphi(x)} \mathbb{R}^m$ is canonically identified
	with $\mathbb{R}^m$ in the above composition of functions.
	In particular, a constant function has null approximate derivative at
	density points located in its domain.
\end{lem}

\begin{proof}
	In Euclidean space, approximate derivatives are unique when they exist
	\cite[Theorem 6.3]{evans2018measure}.
	Since density points are well-defined for Riemannian manifolds by \Cref{lem:density_points_mainfolds}
	and coordinate changes for $M$ are smooth diffeomorphisms,
	it follows from (\ref{equa:approximate_derivative}) that
	the existence of approximate derivative at a given point is independent of the choice of the  chart
	and the change of variables rule applies.
	To show our last statement, note that
	$L = 0$ satisfies (\ref{equa:approximate_derivative}) whenever
	$F: = f \circ \varphi^{-1}$ is a constant function.
\end{proof}

\subsection{Approximate Hessian of locally semi-concave functions}

The properties of locally semi-concave functions provide a valuable toolbox
for analyzing optimal transport maps on manifolds.
In this section, we examine the weak second-order regularity of these functions.

In a Riemannian manifold $(M, {\matheuvm{g}})$,
a subset $C$ of $M$ is said to be a \emph{geodesically convex}
(or \emph{simple and convex}) set if,
given any two points in $C$,
there is a unique minimizing geodesic contained within $C$ that joins those two points.
A function $f: C \rightarrow \mathbb{R}$ defined on a geodesically convex set $C \subset M$
is said to be \emph{geodesically convex} (respectively \emph{geodesically concave}) if the composition
$f \circ \gamma$ of $f$ and any geodesic curve $\gamma$ contained within $C$ is convex (respectively concave).
It is noteworthy that for any point $x \in M$,
there exists an open ball centered at $x$ that is geodesically convex \cite{whitehead1932convex, kobayashi1996foundations},
and such a ball is referred to as a geodesically convex ball.

\begin{defn}[Semi-concavity]
	\label{defn:semi-concavity}
	Let $(M, {\matheuvm{g}})$ be a Riemannian manifold.
	Fix an open subset $O \subset M$.
	A function $\phi: O \rightarrow \mathbb{R}$ is semi-concave
	at $x \in O$ if there exists an open and geodesically convex set $C(x)$ centered at $x$
	and a $\mathcal{C}^2$ function $V : C(x) \rightarrow \mathbb{R}$
	such that $\phi + V$ is
	geodesically concave throughout $C(x)$.
	The function $\phi$ is locally semi-concave on $O$ if it is
	semi-concave at each point of $O$.
\end{defn}


Bangert \cite[(2.3) Satz]{bangert1979analytische} proved that
the notion of local semi-concavity is independent of the Riemannian metric.
This property also follows from the following characterization of locally
semi-concave functions (with a linear module),
whose proof for the Euclidean case is detailed in
\cite[Proposition 4.3, Proposition 4.8]{vial1983strong} and
\cite[Theorem 5.1]{clarke1995proximal}.
In \cite[Appendix A]{fathi2010optimal}, it is adopted as the definition of local semi-concavity.
Denote by $\langle \cdot, \cdot \rangle$ and $\| \cdot \|_2$ respectively
the Euclidean inner product and its associated norm.
To stress that certain points are
coordinate representations of manifold points,
we denote them by tilde symbols $\widetilde x$ and $\widetilde z$.

\begin{prop}[Characterization of local semi-concavity,
		\citeinfo{Proposition 10.12}{villani2009optimal}]
	\label{prop:characterization_local_semi_concavity}
	Let $(M, {\matheuvm{g}})$ be an $m$-dimensional Riemannian manifold.
	Fix an open subset $O$ of $M$.
	A function $f: O \rightarrow \mathbb{R}$ is locally semi-concave
	if and only if for each point in $O$,
	there exist a chart $(\varphi, U)$ defined around the point
	and a positive constant $C > 0$ such that
	$\forall\,\widetilde x \in \varphi(U)$,
	$\exists\, l_{\widetilde x} \in \mathbb{R}^m$, $\forall\, \widetilde {z} \in \varphi(U)$,
	\begin{equation*}
		(f \circ \varphi^{-1}) \left( \widetilde z \right)
		\label{equa:super_differentiability_characterization}
		\leq (f \circ \varphi^{-1}) ( \widetilde x ) + \langle l_{\widetilde x},\,\widetilde z - \widetilde x\rangle
		\ + C \,\|\widetilde z - \widetilde x\|_2^2.
	\end{equation*}
\end{prop}

Hence, a function is locally semi-concave if and only if it is so when expressed in local charts
\cite[discussion after Lemma A.9]{fathi2010optimal}.
We shall apply this chart-independence, along with Alexandrov's theorem,
to establish the weak second-order regularity of locally semi-concave functions.

In the following theorem,
we revisit Alexandrov's theorem stated via approximate derivatives.
The proofs of this theorem can be found in
\cite[Theorem 14.1]{villani2009optimal} and
\cite[Theorem D.2.1]{niculescu2018convex}.
To maintain clarity of notation, for a function
$f : U \rightarrow \mathbb{R}$ defined on an open subset $U \subset \mathbb{R}^m$,
we define its Euclidean gradient $\nabla^E f(x) \in \mathbb{R}^m$ at $x \in U$
as the (column) vector $(\partial_1 f(x), \partial_2 f(x), \ldots, \partial_m f(x))$
when all of these partial derivatives of $f$ exist at $x$.
By contrast, the symbol $\nabla f$ is reserved to
denote the gradient of functions $f: U \rightarrow \mathbb{R}$ defined on some open subset $U$
of a Riemannian manifold,
which is a (possibly not continuous) vector field defined
at points where $f$ is differentiable.

\begin{thm}[Alexandrov's theorem]\label{thm:Alexandrov-theorem}
	Let $f: U \subset \mathbb{R}^m \rightarrow \mathbb{R}$ be a semi-concave function.
	Then the Euclidean gradient $\nabla^E f$ of $f$ is defined $\Leb^m$-almost everywhere on $U$:
	\[
		\nabla^E f: A \longrightarrow  \mathbb{R}^m \qquad \text{ with }\,
		A \in \mathcal{B}(\mathbb{R}^m) \text{ and } \Leb^m (U \setminus A) = 0.
	\]
	For $\Leb^m$-almost everywhere on $A$,
	the function $\nabla^E f$ is approximately differentiable
	and its approximate derivative $(\partial_{ij}^2 f)_{1 \le i, j \le m}$ forms a symmetric matrix.
	Moreover, at every point $x$ where such approximate derivative of $ \nabla^E f$ exists,
	$f$ admits a second-order Taylor expansion:
	\begin{equation}
		\label{equa:Alexandrov_theorem_Taylor_expansion}
		f(z) = f(x) + \langle \nabla^E f(x), z-x\rangle +
		\frac{1}{2} \langle \operatorname{ap} \di_x  \nabla^E f (z-x),z-x\rangle + o(\|z-x\|_2^2).
	\end{equation}
\end{thm}

\begin{rmk}
	\label{rmk:equivalent_formulations_Alexandrov_theorem}
	In the literature, the weak second-order regularity in Alexandrov's theorem is
	expressed in different formulations, including the one that
	differentiates super-gradients of semi-concave functions
	\cite[Theorem D.2.1, Theorem D.2.2]{niculescu2018convex}.
	Their equivalence to (\ref{equa:Alexandrov_theorem_Taylor_expansion})
	is proven in \cite[Theorem 14.25]{villani2009optimal}.
	Compared to these equivalent formulations,
	our \Cref{thm:Alexandrov-theorem} further requires $x$ to be a density point of $A$
	for the existence of $\operatorname{ap} \di_x  \nabla^E f$.
	However, under our assumption that $U$ is an open set,
	the condition $\Leb^m (U \setminus A) = 0$ implies that
	every point of $A$ is a density point.
\end{rmk}

To extend our results to the Riemannian setting,
we provide a concise review of the Riemannian Hessian.
For a $\mathcal{C}^2$ function defined on a Riemannian manifold $(M, \matheuvm{g})$,
the Hessian at a point $x \in M$ can be interpreted either as a self-adjoint linear map
from the tangent space $T_xM$ to itself
or as a symmetric bilinear form on $T_xM \times T_xM$.
These two interpretations are related by duality
through the Riemannian metric $\matheuvm{g}$ at $x$ \cite[Proposition 2.2.6]{petersen2016riemannian}.
While we shall primarily adopt the linear map perspective in the subsequent sections,
we shall utilize the bilinear form viewpoint in the following two paragraphs.
This choice is motivated by the fact that the chart-based expression of the Hessian
is simpler when viewed as a bilinear form.

In what follows, the Hessian of a $\mathcal{C}^2$ function on a Riemannian manifold
is a particular instance of a continuous $(0, 2)$-tensor $S$.
Namely, for any two given charts $\varphi, \psi$ defined on a common open subset
$U \subset M$, there exist two bilinear forms $S_{\varphi}$ and $S_{\psi}$
whose coefficients are continuous functions such that
$\forall\, \widetilde x \in \varphi(U) \subset \mathbb{R}^m,\, \forall\, u,v \in \mathbb{R}^m$,
\begin{equation*}
	\label{eq: tensor}
	[S_{\varphi}{(\widetilde x)}] \,(u,v)
	= [S_{\psi}{(T(\widetilde x))}]
	(\di_{\widetilde x} T(u), \di_{\widetilde x} T(v)),
\end{equation*}
where $T = \psi \circ \varphi^{-1}$ is assumed to be a smooth (transition) map defined on $\varphi(U)$.
In the case of the Hessian of a $\mathcal{C}^2$ function $f$, its expression in a chart $\varphi$ is given by
\[
	\operatorname{Hess}_{\widetilde x} (f \circ \varphi^{-1})(\partial_i, \partial_j) =
	\partial^2_{ij} (f\circ \varphi^{-1})(\widetilde x) -
	\sum_{k=1}^m \Gamma_{ij}^k(\widetilde x) \, \partial_k (f \circ \varphi^{-1})(\widetilde x),
\]
where $\partial_i$ are the coordinate vectors associated with the given coordinate system
\cite[p.60 of Chapter 3]{lee2012introduction},
and $\Gamma_{ij}^k$ are the Christoffel symbols of the chart,
see \cite[Chapter 2]{petersen2016riemannian} for more details.

In the particular case of a chart $\varphi$ inducing a normal coordinate system at $x_0 \in M$
\cite[\S 2 of Chapter II]{sakai1996riemannian},
i.e., $\varphi^{-1}(u)=\exp_{x_0} (u)$ after identifying $T_{x_0}M$ with $\mathbb{R}^m$
by choosing an orthonormal basis of $T_{x_0}M$,
the matrix made with the metric components $g_{ij}$ is the identity at $\widetilde x_0 = \varphi(x_0)$,
and all its first-order partial derivatives (and thus the Christoffel symbols)
vanish at $\widetilde x_0$  \cite[2.89 bis]{gallot2004riemannian}.
Hence, the above formula at the point $\widetilde x_0$ is simplified into
\begin{equation}
	\label{equa:Hessain_in_normal_coordinate}
	\operatorname{Hess}_{\widetilde x_0} (f \circ \varphi^{-1})(\partial_i, \partial_j) =
	\partial^2_{ij} (f\circ \varphi^{-1})(\widetilde x_0).
\end{equation}
Since the metric matrix $(g_{ij})_{1 \le i,j \le m}$ at $\widetilde x_0$ is the identity,
if we consider $\operatorname{Hess}_{\widetilde x_0} (f \circ \varphi^{-1})$
as a linear map from $\mathbb{R}^m \cong T_{x_0} M$ to itself,
then it coincides with the derivative of $\nabla^E (f \circ \varphi^{-1})$ at $\widetilde x_0$.

As a consequence, we are led to the following definition of Hessian
for semi-concave functions on a Riemannian manifold.

\begin{defn}[Hessian of semi-concave functions]
	\label{defn:approximate_Hessian}
	Let $(M, \matheuvm{g})$ be an $m$-dimensional complete Riemannian manifold,
	$f: O \rightarrow \mathbb{R}$ be a semi-concave function defined on an open subset $O \subset M$,
	and $A\subset O$ be the subset of points where $f$ is differentiable.

	The function $f$ is said to have an \emph{approximate Hessian}
	or simply a \emph{Hessian} at a point $x \in A$
	if there exists a chart $(\varphi,U)$ inducing a \emph{normal} coordinate system around $x$
	such that $\nabla^E (f \circ \varphi^{-1})$ is approximately differentiable at $\varphi(x)$,
	and its approximate derivative is symmetric.
	Then the Hessian of $f$ at $x$
	is the function $\operatorname{Hess}_x f$ from $T_xM$ to $T_xM$ defined by
	\begin{equation}
		\label{equa:Hessian}
		\operatorname{Hess}_x f (u)
		:=  (\di_x\varphi)^{-1}\circ \operatorname{ap} \di_{\varphi(x)}
		\nabla^E (f \circ \varphi^{-1}) \circ \di_x\varphi(u),
		\quad \forall\, u\in T_xM.
	\end{equation}
\end{defn}

\begin{rmk}
	To justify \Cref{defn:approximate_Hessian},
	first note that if $(\psi,V)$ is another chart defined in a neighborhood of $x$,
	then $\nabla^E (f \circ \varphi^{-1})$ is approximately differentiable at $\varphi(x)$
	if and only if $\nabla^E (f \circ \psi^{-1})$ is approximately differentiable at $\psi(x)$;
	indeed both vector fields are related by the formula
	\begin{equation}
		\label{equa:gradient_in_dfferent_charts}
		^t(\di_{\psi(z)} T) \cdot [\nabla^E (f \circ \varphi^{-1}) (\varphi(z))]
		= \nabla^E (f \circ \psi^{-1})(\psi(z)),
	\end{equation}
	where $z$ is close to $x$, $T := \varphi \circ \psi^{-1}$ is a $\mathcal{C}^\infty$
	diffeomorphism defined around $\psi(x)$
	and $^t(\di_{\psi(z)} T)$ is the transpose of $T$'s differential at $\psi(z)$.
	See the proof of Lemma \ref{defn:approximate_derivative_manifolds} for a similar argument.
	Moreover, in our definition (\ref{equa:Hessian}) of $\operatorname{Hess}_x f (u)$,
	we can justify the independence of charts (inducing normal coordinate systems) in two different ways.
	Since the Hessian of a $\mathcal{C}^2$ function defined on manifolds is a tensor,
	the required independence is guaranteed by its simplified local expressions 
	(\ref{equa:Hessain_in_normal_coordinate}) in normal coordinate systems.
	Alternatively, we suppose that $(\psi, V)$ also induces a normal coordinate system around $x$,
	which implies that the transition map $T = \varphi \circ \psi^{-1}$ is linear.
	By applying the chain rule to (\ref{equa:Hessian}) for the chart $(\psi, V)$,
	the independence follows from the linearity of
	$\di_{\psi(z)} T = T$ and the equality
	(\ref{equa:gradient_in_dfferent_charts}).
\end{rmk}

To summarize the content of this part,
we have obtained the following analog of Alexandrov's theorem
for locally semi-concave functions on Riemannian manifolds.

\begin{prop}
	\label{prop:existence_Hessian_semi_concave_functions}
	Let $(M, \matheuvm{g})$ be a complete Riemannian manifold.
	Fix an open subset $O \subset M$ and a locally semi-concave function $f: O \rightarrow \mathbb{R}$.
	For $\operatorname{Vol}$-almost every $x \in O$,
	there exists a function $\operatorname{Hess}_x f: T_xM \rightarrow T_xM$,
	called the Hessian of $f$ at $x$, such that
	\begin{itemize}
		\item $\operatorname{Hess}_x f$ is a self-adjoint operator on $T_xM$;
		\item the function $f$ satisfies the following second-order expansion at $x$,
		      \begin{equation}
			      \label{equa:second_order_expansion}
			      f(\exp_{x} u) = f(x) + \di_xf( u) +
				  \frac{1}{2} \matheuvm{g}_x(\operatorname{Hess}_x f(u), u)+o(\|u\|^2),
		      \end{equation}
		      for  $u \in T_xM$.
	\end{itemize}
\end{prop}

\subsection{Differentiating optimal transport maps}

In this part, we collect some properties of optimal transport maps between absolutely continuous measures on a Riemannian manifold, which are taken from \cite[Sections 4 \& 5]{cordero2001riemannian}.
These properties will be used in \Cref{sec:ricci_curvature}.
To justify them, we remark that
our definition of Hessian enjoys the second-order expansion (\ref{equa:second_order_expansion}),
which allows us to apply properties proven
for the Hessian defined in \cite[Definition 3.9]{cordero2001riemannian}.
See \Cref{rmk:equivalent_formulations_Alexandrov_theorem} and
\cite[Discussion after Definition 3.9]{cordero2001riemannian} for more details.

\begin{prop}[Differentiating optimal transport maps, \citeinfo{Proposition 4.1}{cordero2001riemannian}]
	\label{defn:differentiate_optimal_transport_map}
	Let $(M, \matheuvm{g})$ be a complete Riemannian manifold.
	Given a $c$-concave function $\phi$ defined on $\widebar {\mathcal{X}} \subset M$
	with $\mathcal{X}$ a bounded open set,
	we set $F := \exp(- \nabla \phi)$,
	which is $\operatorname{Vol}$-almost everywhere well-defined on $\mathcal{X}$.
	Fix a point $x \in \mathcal{X}$ such that $\operatorname{Hess}_x \phi$
	exists (\ref{equa:Hessian}).
	Then the point $y : = F(x)$ is not in the cut locus of $x$,
	$\nabla \phi(x) = \nabla d^2_y /2 (x)$,
	and $\operatorname{Hess}_x d^2_{y} / 2 - \operatorname{Hess}_x \phi$
	is positive semi-definite.
	Define the differential $\di_x F : T_xM \rightarrow T_y M$ of $F$ at $x$ as
	\begin{align}
		\label{equa:differentiate_optimal_transport_map}
		\di_x F: = [\di_{- \nabla \phi(x)} \exp_x]
		\circ (\operatorname{Hess}_x d^2_{y} / 2 - \operatorname{Hess}_x \phi),
	\end{align}
	and define $\operatorname{Jac} F(x): = \det \di_x F$ as
	the Jacobian determinant of $\di_x F$.
\end{prop}

The Jacobian determinant of the differential
$\di_x F$, as defined in \Cref{defn:differentiate_optimal_transport_map},
is calculated with respect to normal coordinate systems of
the tangent spaces $T_xM$ and $T_yM$ \cite[Lemma 2.1]{cordero2001riemannian}.
By \cite[Claim 4.5]{cordero2001riemannian}, these algebraic Jacobians
are equivalent to their geometric counterparts,
which results in the following change of variables formula.
For further details, see \cite[p.364 of Chapter 14]{villani2009optimal}.

\begin{prop}[Interpolation and change of variable formula]
	\label{prop:Jacobian_optimal_transport_map}
	Let $(M, \matheuvm{g})$ be a complete Riemannian manifold.
	Fix two absolutely continuous measures $\mu, \nu \in \mathcal{W}_2(M)$ with
	supports contained in two bounded open sets
	$\mathcal{X}$ and $\mathcal{Y}$ respectively.
	Let $F := \exp (-\nabla \phi)$ be the optimal transport map
	that pushes $\mu$ forward to $\nu$, where $\phi \in \mathcal{I}^c(\widebar{ \mathcal{X}}, \widebar {\mathcal{Y}})$ is a $c$-concave function
	given by \Cref{thm:optimal_transport_manifold}.

	Denote by $\phi^c \in \mathcal{I}^c(\widebar{\mathcal{Y}} , \widebar{\mathcal{X}})$
	the $c$-conjugate of $\phi$. The set
	\[
		\Omega : = \left\{ x \in \mathcal{X} \mid F(x) \in \mathcal{Y},\,
		\operatorname{Hess}_x \phi \text{ and } \operatorname{Hess}_{F(x)} \phi^c
		\text{ exist }\right\}
	\]
	satisfies the following properties:
	\begin{enumerate}
		\item $\mu(\Omega) = 1$;
		\item  defining $F^t : = \exp(- t \nabla \phi)$ for $0 \le t \le 1$,
		      we have $\operatorname{Jac} F^t > 0$ on $ \Omega$;
		\item denote by $f$ and $g$ the density functions of $\mu$ and $\nu$ respectively;
		      there exists a measurable subset $N \subset \Omega$
		      depending on these two density functions
		      such that $\mu(N) = 1$ and for $x \in N$,
		      \[
			      f(x) = g(F(x)) \operatorname{Jac} F(x) > 0;
		      \]
		\item for any Borel function $A$ on $[0, + \infty)$ with $A(0) = 0$,
		      with the set $N$ as in Property 3,
		      \begin{equation}
			      \label{equa:change_of_variable}
			      \int_{M} A(g)\diff \operatorname{Vol}
			      = \int_{N} A\left(\frac{f}{\operatorname{Jac} F}\right)
			      \operatorname{Jac} F \diff \operatorname{Vol}.
		      \end{equation}
		      (Either both integrals are undefined or both
		      take the same value in $\mathbb{R}\cup\{+\infty, -\infty\}$.)
	\end{enumerate}
\end{prop}

\begin{proof}
	All the statements follow from \cite[Claim 4.4, Theorem 4.2, Corollary 4.7]{cordero2001riemannian} except Property 2 for $t \in (0,1)$.
	Recall that  for any $c$-concave function $\phi$,
	we always have $\det [\di_{-t\nabla \phi(x)} \exp_x ] > 0$
	since $\exp_x (-t\nabla \phi(x)) $
	is not in the cut locus of $x$. As $t \phi$ is $c$-concave for $0 < t < 1$,
	it suffices to show that
\begin{equation}
	\label{equa:proof_property_2}
	\det[\operatorname{Hess}_x d^2_{F(x)} / 2 - \operatorname{Hess}_x \phi] > 0
	\implies
	\forall \, t \in (0, 1),\,
	\det[\operatorname{Hess}_x d^2_{F^t(x)} / 2 - t \operatorname{Hess}_x \phi] > 0.
\end{equation}

	Indeed, $\operatorname{Hess}_x d^2_{F^t(x)}/2 - t \operatorname{Hess}_x d^2_{F(x)}/2$
	is positive semi-definite for $0 < t < 1$ \cite[Lemma 2.3]{cordero2001riemannian},
	which implies (\ref{equa:proof_property_2}) according to
	Minkowski's determinant inequality
	\cite[(5.23)]{villani2021topics}.
	%
\end{proof}


\subsection{Proof of Hessian equality}

The Hessian equality (\ref{equa:hessian_equality}) to prove is a second-order relation.
We first demonstrate a first-order counterpart of this equality using the conclusion of \Cref{prop:construction_Wasserstein_barycenter}
that relates barycenters in manifolds to Wasserstein barycenters.

\begin{thm}[Hessian equality for Wasserstein barycenters]
	\label{thm:hessian_equality_Wasserstein_barycenter}
	Let $(M, \matheuvm{g})$ be a complete Riemannian manifold.
	Given an integer $n \ge 2$,
	let $\lambda_i > 0, 1 \le i \le n$, be $n$ positive real numbers
	such that $\sum_{i=1}^n \lambda_i =1$ and
	let $\mu_i \in \mathcal{W}_2(M), 1 \le i \le n$, be $n$
	probability measures with compact support.
	We assume that $\mu_1$ is absolutely continuous.
	The unique barycenter $\widebar \mu$ of $\mathbb{P}: =\sum_{i=1}^n \lambda_i\, \delta_{\mu_i}$
	is absolutely continuous with compact support.
	For $ 1 \le i \le n $, let $F_i = \exp( - \nabla \phi_i)$ be the optimal
	transport map pushing $\widebar{\mu}$ forward to $\mu_i$,
	where $\phi_i$ is a $c$-concave function given by \Cref{thm:optimal_transport_manifold}.

	For $\widebar \mu$-almost every $x \in M$, $x$ is a barycenter of
	$\sum_{i=1}^n \lambda_i \,\delta_{F_i(x)}$, and we have the Hessian equality
	\begin{align}
		\label{equa:hessian_equality}
		\sum_{i=1}^n \lambda_i \operatorname{Hess}_x \phi_i & = 0.
	\end{align}
\end{thm}

\begin{proof}
	By \Cref{thm:absolute_continuity_discrete},
	$\widebar \mu$ is absolutely continuous with compact support.
	We now apply \Cref{prop:construction_Wasserstein_barycenter} to $\mathbb{P}$.
	Since $\widebar \mu$ is the unique barycenter of $\mathbb{P}$,
	it coincides with the barycenter constructed in \Cref{prop:construction_Wasserstein_barycenter}.
	Consider the identity map $\operatorname{Id}: (M, \mathcal{B}(M), \widebar \mu) \rightarrow M$
	as a random variable taking values in $M$.
	It has law $\widebar \mu$, and the random variable
	$F_i = F_i \circ \operatorname{Id}$ has law $\mu_i$ for $1 \le i \le n$.
	\Cref{prop:construction_Wasserstein_barycenter} implies that
	for $\widebar \mu$-almost every $x \in M$,
	$x$ is a barycenter of $\sum_{i=1}^n \lambda_i\, \delta_{F_i(x)}$.

	Let $\Omega$ be a Borel subset of $M$ with
	$\widebar \mu(\Omega) = 1$ such that for $x \in \Omega$,
	$\nabla \phi_i(x)$ exists for $1 \le i \le n$ and
	$x$ is a barycenter of $\sum_{i=1}^n \lambda_i\, \delta_{F_i(x)}$.
	Fix a point $x \in \Omega$.
	By definition, $x$ reaches the minimum of the function
	\[
		h:	w \in M \mapsto W_2(\delta_w, \sum_{i=1}^n \lambda_i\, \delta_{F_i(x)})^2
		= \sum_{i=1}^n \lambda_i \,d_{\matheuvm{g}}(w, F_i(x))^2.
	\]
	By \Cref{lem:barycenter_discrete_measure_out_of_cut_locus},
	the fixed point $x$ is out of the cut locus of any point $F_i(x)$ for $ 1 \le i \le n$.
	We can thus differentiate $h$ at $w = x$ and get
	$\nabla h |_{w = x} = 0$.
	Since $\nabla \phi_i (x) = \frac{1}{2}\nabla d^2_{F_i(x)} |_{w = x}$ holds
	as both gradients exist \cite[Lemma 3.3]{cordero2001riemannian},
	it follows that $ \sum_{i=1}^n \lambda_i \nabla \phi_i (x) = \frac{1}{2} \nabla h |_{w = x} = 0$.

	Define $f : = \sum_{i=1}^n \lambda_i\,\phi_i$ on a neighborhood of $\Omega$
	that is a common domain for $\phi_i, 1 \le i \le n$.
	The function $f$ is locally semi-concave as each $\phi_i$ is so,
	and for $x \in \Omega$,
	$\nabla f(x) = \sum_{i=1}^n \lambda_i \nabla \phi_i(x) =  0 \in T_xM$
	by the previous arguments.
	Let $\Omega_1 \subset \Omega$ be the set where the Hessians of
	$f$ and $\phi_i, 1 \le i \le n$, all exist.
	Let $\Omega_2$ be the set of density points of $\Omega$.
	We have $\operatorname{Vol}(\Omega \setminus \Omega_1) = 0$
	by \Cref{prop:existence_Hessian_semi_concave_functions},
	and $\operatorname{Vol}(\Omega \setminus \Omega_2) = 0$ by \cite[Theorem 1.35]{evans2018measure}.

	For $x \in \Omega_1$, using the linearity of the Hessian operator, we get
	$\operatorname{ Hess}_x f = \sum_{i=1}^n \lambda_i\, \operatorname{Hess}_x \phi_i$
	by (\ref{equa:Hessian}). Besides, noting that $\nabla f$ is constant on $\Omega$, we infer from the last statement of
	Lemma \ref{defn:approximate_derivative_manifolds} that for $x \in \Omega_2 \cap \Omega$, $\operatorname{Hess}_x f = 0$.
	It follows that for $x \in \Omega_1 \cap \Omega_2$,
	$ \sum_{i=1}^n \lambda_i\, \operatorname{Hess}_x \phi_i = 0$.
	This proves the theorem since $\widebar \mu(\Omega_1 \cap \Omega_2) = 1$
	thanks to the absolute continuity of $\widebar \mu$.
\end{proof}

%


\section{Lower Ricci curvature bounds and displacement functionals}
\label{sec:ricci_curvature}

In this section, we introduce a class of displacement functionals exploiting
the Hessian equality in \Cref{thm:hessian_equality_Wasserstein_barycenter}.
This is one of the primary difference between our approach and the one proposed by Kim and Pass
\cite{kim2017wasserstein} regarding the absolute continuity of the barycenter.


In \Cref{sec:hessian_equality},
the notion of Hessian plays a central role in
differentiating optimal transport maps.
There is also the following widely used connection
between $\operatorname{Hess}_x \phi$ and Jacobi equations
involving $\exp(-\nabla \phi)$,
which is demonstrated in various works including
Sturm \cite{sturm2005convex},
Lott and Villani \cite[\S 7]{lott2009ricci},
Cordero-Erausquin et al.\@ \cite{cordero2006prekopa}
and Villani \cite[Chapter 14]{villani2009optimal}.
The function $J(t)$ defined below is actually
$\di_x \exp(- \nabla t\,\phi)$ using (\ref{equa:differentiate_optimal_transport_map}).
By convention, for a function $f$ with variable $t$,
we denote by $\dot f$ its derivative with respect to $t$.

\begin{prop}
	\label{prop:Jacobi_equation}
	Let $(M, \matheuvm{g})$ be an $m$-dimensional complete Riemannian manifold and
	let $\phi$ be a $c$-concave function defined on $\widebar {\mathcal{X}} \subset M$
	with $\mathcal{X}$ a bounded open set.
	Fix a point $x \in \mathcal{X}$ such that $\operatorname{Hess}_x \phi$
	(\Cref{prop:existence_Hessian_semi_concave_functions}) exists.
	Then $t \in [0, 1] \mapsto \gamma(t) = \exp(-t \nabla \phi)(x)$ is a minimal geodesic.
	Define
	\[
	J:	t \in [0, 1] \mapsto
		\di_{- t\nabla \phi(x)} \exp_x \cdot (\operatorname{Hess}_x d^2_{\gamma(t)}/ 2
		- t \operatorname{Hess}_x \phi).
	\]
	Denote by $\Delta \phi(x)$ the trace of $\operatorname{Hess}_x \phi$ and
	by $\det J(t), 0 \le t \le 1$ the determinant of $J(t)$
	calculated in coordinates using
	orthonormal bases of $T_xM$ and $T_{\gamma(t)} M$.
	If $- K \in \mathbb{R}$ is a lower Ricci curvature bound of $M$
	along $\gamma$ and $\det J > 0$, then $\ell : = - \log \det J$
	defined on $[0,1]$ satisfies
	\[
		\ddot \ell \ge \dot \ell ^2 / m - K \Vert \nabla \phi(x) \Vert^2
	\]
	with $\ell(0) = 0$ and $\dot \ell(0) = \Delta \phi(x)$.
	In particular,
	\[
		l \ge \Delta \phi(x) -  K \| \nabla \phi(x) \|^2 / 2,
	\]
	where we define $l := \ell(1) = - \log \det J(1)$.
\end{prop}

The following displacement functionals
$f \diff \operatorname{Vol} \in \mathcal{W}_2(M) \mapsto \int G(f) \diff \operatorname{Vol}$
are inspired by the entropy functional, where $G(x):= x \log x$.
To uniformly bound (from above) their values of the sequence of barycenter measures
in the law of large numbers for Wasserstein barycenters,
we add the assumption of bounded derivatives.
Examples of $G$ can be constructed according to
\Cref{lem:modified_de_la_Vallee_Poussin_theorem}.

\begin{prop}[Displacement functionals]
	\label{prop:entropy_estimation}
	Let $(M, \matheuvm{g})$ be an $m$-dimensional complete Riemannian manifold
	with a lower Ricci curvature bound $ - K$ ($K \ge 0$).
	Given an integer $n \ge 2$,
	let $\lambda_i > 0, 1 \le i \le n$, be $n$ positive real numbers
	such that $\sum_{i=1}^n \lambda_i = 1$ and let $\mu_i \in \mathcal{W}_2(M), 1 \le i \le n$, be $n$
	probability measures with compact support.
	Assume that there is an integer $ 1\le k \le n$ such that
	for any index $ 1 \le i \le k$, $\mu_i $ is absolutely continuous
	with density function $g_i$.
	Denote by $\widebar \mu$ the unique Wasserstein barycenter
	of $\mathbb{P}: = \sum_{i=1}^{n} \lambda_i\, \delta_{\mu_i} \in (\mathcal{W}_2(\mathcal{W}_2(M)), \mathbb{W}_2)$,
	which is absolutely continuous, and we denote
	by $f$ its density function.

	Let $G$ be a function on $[0, \infty)$ such that
	$G(0) = 0$, and the function
	$H: x \in \mathbb{R} \mapsto {G(e^x)}\,{e^{-x}}$ is
	continuously differentiable with non-negative derivative
	bounded above by some constant $L_H > 0$.
	The following inequality holds,
	\begin{equation}
		\label{equa:entropy_estimation_on_M}
		\int_M G(f) \diff \operatorname{Vol}
		\le \sum_{i=1}^k \frac{\lambda_i}{\Lambda}
		\int_M G(g_i) \diff \operatorname{Vol} + \frac{L_H K}{2 \Lambda}
		\mathbb{W}_2(\mathbb{P}, \delta_{\widebar{\mu}})^2
		+ \frac{L_H}{2 \Lambda}({m^2} + 2m),
	\end{equation}
	where we define the constant $\Lambda := \sum_{i=1}^k \lambda_i$.
\end{prop}

\begin{rmk}
	The following example helps to understand (\ref{equa:entropy_estimation_on_M}).
	Take $\mathbb{P} = \lambda \, \delta_{\mu_1} + (1- \lambda) \delta_{\mu_2}$
	with $0 < \lambda < 1$ and absolutely continuous measures $\mu_1, \mu_2 \in \mathcal{W}_2(M)$.
	Set $G(x) := x \log x$. Since $H(x) = x$, we choose $L_H = 1$.
	Define $\operatorname{Ent}(f \cdot \operatorname{Vol}) := \int_M G(f) \diff \operatorname{Vol}$.
	The inequality (\ref{equa:entropy_estimation_on_M}) becomes
	\[
		\operatorname{Ent}(\widebar{\mu}) \le
		\lambda \operatorname{Ent}(\mu_1) + (1-\lambda) \operatorname{Ent}(\mu_2)
		+ \frac{K}{2} \lambda(1-\lambda) W_2(\mu_1, \mu_2)^2
		+ \frac{m^2}{2} + m,
	\]
	which has exactly one additional term ${L_H}(m^2 + 2m)/(2 \Lambda)$
	compared to the $\lambda$-convexity expression of $\operatorname{Ent}$
	used to define lower Ricci curvature bound $-K$ for metric measure spaces
	in \cite[\S 4,2]{sturm2006geometryI} and \cite[Definition 0.7]{lott2009ricci}.

	Moreover, ${L_H}(m^2 + 2m)/(2 \Lambda)$ is also the only additional term
	when we compare inequality (\ref{equa:entropy_estimation_on_M}) with
	the Wasserstein Jensen’s inequality proven by Kim and Pass \cite[Theorem 7.11]{kim2017wasserstein},
	which corresponds to the case $k=n$.
	However, our inequality (\ref{equa:entropy_estimation_on_M}) for the case $k < n$
	is crucial to the proof of our main result in the next section.
\end{rmk}

\begin{proof}[Proof of \Cref{prop:entropy_estimation}]
	For $1 \le i \le n$, let $F_i := \exp(- \nabla \phi_i)$ be the optimal transport map
	from $\widebar \mu$ to $\mu_i$ with $\phi_i$ a $c$-concave function
	given by \Cref{thm:optimal_transport_manifold}.
	According to \Cref{thm:hessian_equality_Wasserstein_barycenter} and
	\Cref{prop:Jacobian_optimal_transport_map},
	there exists a Borel set $\Omega \subset M$ with $\widebar \mu (\Omega) = 1$
	such that $\sum_{i=1}^n \lambda_i \operatorname{Hess}_x \phi_i = 0$ for $x \in \Omega$,
	$\operatorname{Jac} \exp(- t \nabla \phi_i) > 0$ on $\Omega$
	for $t \in [0, 1]$ and $1 \le i \le k$, and
	\begin{equation}
		\label{equa:change_of_variable_with_G_and_f}
		\int_{M} G\left(g_i\right)\diff \operatorname{Vol}
		= \int_{N_i} G\left(\frac{f}{\operatorname{Jac} F_{i}}\right)
		\operatorname{Jac} F_{i} \diff \operatorname{Vol}, \quad 1 \le i \le k,
	\end{equation}
	where $N_i \subset \Omega$ for $1 \le i \le k$ are Borel sets
	such that $\widebar{\mu}(N_i) = 1$ and
	$f= g_i(F_i) \operatorname{Jac} F_i > 0$ on $ N_i$.
	Hence, $\log f$ is well-defined on $\cup_{i=1}^k N_i$.
	Define $l_i(x) := - \log \operatorname{Jac} F_i(x)$ on $\Omega$.
	It follows from (\ref{equa:change_of_variable_with_G_and_f}) that
	\begin{equation}
		\label{equa:rewriting_integral_G_g_i}
		\int_{M} G(g_i) \diff \operatorname{Vol}
		= \int_{N_i} H(\log f + l_i)\diff \widebar {\mu},
		\quad 1 \le i \le k.
	\end{equation}
	Applying \Cref{prop:Jacobi_equation} to $\phi_i$ for $ 1 \le i \le k$, 
	we have on $\Omega$,
	\begin{equation}
		\label{equa:lower_bound_l_i}
		l_i \ge \Delta \phi_i -  K \| \nabla \phi_i \|^2 / 2,
		\quad 1 \le i \le k.
	\end{equation}

	For $x \in \Omega$ and $1 \le i \le n$, since $\operatorname{Hess}_x d^2_{F_i(x)} /2
		- \operatorname{Hess}_x \phi_i$ is positive semi-definite
	(\Cref{defn:differentiate_optimal_transport_map}),
	we can also bound $\Delta \phi_i(x)$ from above using
	the upper bound of the Laplacian of distance functions
	observed by Kim and Pass \cite[Lemmma 2.7]{kim2017wasserstein}:
	\begin{align}
		\label{equa:upper_bound_for_Delta_phi}
		\Delta \phi_i(x) \le \Delta d^2_{F_i(x)} /2
		 & \le m \frac{\sqrt K d_{\matheuvm{g}}(x, F_i(x))}{\tanh (\sqrt K d_{\matheuvm{g}}(x, F_i(x)))} \nonumber \\
		 & \le m(1 + \sqrt{K} d_{\matheuvm{g}}(x, F_i(x)))
		\le m + m^2/ 2 + K\, \| \nabla \phi_i(x) \|^2 / 2,
	\end{align}
	where we used the general inequality
	$\alpha / \tanh \alpha \le 1 + \alpha$ for $\alpha \ge 0$
	\footnote{
		Since $\lim_{\alpha \downarrow 0} \frac{\alpha}{\tanh \alpha} = 1$,
		it suffices to show that the function
		$f(\alpha): =\sinh \alpha + \alpha \sinh \alpha - \alpha \cosh \alpha$ is non-negative
		for $\alpha \ge 0$.
		As $f(0) = 0$ and $f^\prime(\alpha) = \sinh \alpha + \alpha (\cosh \alpha - \sinh \alpha)
			= \sinh\alpha + \alpha \, e^{-\alpha}$,
		we have $f^\prime(\alpha) \ge 0$ and thus $f(\alpha) \ge f(0) = 0$.
	},
	applied the inequality of arithmetic and geometric means
	to $ \sqrt{K \,d_{\matheuvm{g}}(x, F_i(x))^2 } \cdot \sqrt{m^2}$,
	and employed the equality $d_{\matheuvm{g}}(x, F_i(x)) = \|\nabla \phi_i(x) \|$ for $x \in \Omega$.
	With our assumptions on $H$, (\ref{equa:lower_bound_l_i}) and
	(\ref{equa:upper_bound_for_Delta_phi}) imply
	that for $1 \le i \le k$,
	on the set $\cup_{i=1}^k N_i$ (where $\log f$ is well-defined),
	\begin{align}
		H(\log f + l_i) - H(\log f) = H^\prime (\xi) \,l_i   \nonumber & \ge
		H^\prime (\xi)  [\Delta \phi_i - K \| \nabla \phi_i \|^2 / 2] \nonumber \\        & \ge
		         H^\prime (\xi)  [\Delta \phi_i -  K \| \nabla \phi_i \|^2 / 2
		- m -{m^2}/{2}] \nonumber                                               \\        & \ge
		         \label{equa:H_derivative_comparison}
		L_H( \Delta \phi_i - K \| \nabla \phi_i \|^2 / 2) - L_H(m + m^2/2),
	\end{align}
	where we applied the mean value theorem to $H$ that gave
	the real number $\xi$ between $\log f + l_i$ and $\log f$.
	Sum up $k$ inequalities as (\ref{equa:H_derivative_comparison})
	with coefficients $\lambda_i / \Lambda$ on the set $\cup_{i=1}^k N_i$,
	\begin{align}
		H(\log f)
		 & \le \sum_{i=1}^k \frac{\lambda_i}{\Lambda} H(\log f + l_i)
		- \frac{L_H}{\Lambda} \sum_{i=1}^k \lambda_i
		(\Delta \phi_i -  K \| \nabla \phi_i \|^2 / 2) +
		L_H(m + m^2/2) \nonumber                                      \\
		 & = \sum_{i=1}^k \frac{\lambda_i}{\Lambda} H(\log f + l_i)
		+ \frac{L_H}{\Lambda} \sum_{i> k}^n \lambda_i\,\Delta \phi_i
		+ \frac{L_H K}{2 \Lambda} \sum_{i=1}^k \lambda_i \,\| \nabla \phi_i \|^2
		+ L_H(m+m^2/2)\nonumber                                       \\
		\label{equa:sumed_up_H_inequality}
		 & \le \sum_{i=1}^k \frac{\lambda_i}{\Lambda} H(\log f + l_i)
		+ \frac{L_H K}{2 \Lambda} \sum_{i=1}^n \lambda_i \,\|\nabla \phi_i \|^2 +
		\frac{L_H}{2\Lambda}(m^2 + 2m),
	\end{align}
	where we used $\sum_{i=1}^n\lambda_i \,\Delta \phi_i = 0$ derived from the Hessian equality
	for the first equality and used
	(\ref{equa:upper_bound_for_Delta_phi}) for the last inequality.
	Finally, (\ref{equa:entropy_estimation_on_M}) follows from
	(\ref{equa:rewriting_integral_G_g_i}) after
	integrating (\ref{equa:sumed_up_H_inequality})
	over $N_1 \cap \ldots \cap N_k$ against
	$\widebar{\mu}$ since $\widebar{\mu}(N_i) = 1$ for $ 1 \le i \le k$
	and $W_2(\widebar{\mu}, \mu_i)^2
		= \int_M \| \nabla \phi_i \|^2 \diff \widebar \mu$ for $ 1 \le i \le n$.
\end{proof}


\section{Proof of our main result}
\label{sec:measure_theory}

In this section, we prove our main result, i.e., the following theorem.
\begin{thm}
	\label{thm:final_theorem_manifolds}
	Let $(M, \matheuvm{g})$ be a complete Riemannian manifold
	with a lower Ricci curvature bound.
	If a probability measure $\mathbb{P} \in \mathcal{W}_2(\mathcal{W}_2(M))$
	gives mass to the set of absolutely continuous probability
	measures on $M$, then its unique Wasserstein barycenter
	is absolutely continuous.
\end{thm}
New auxiliary results in this section no longer require Riemannian structure,
so we usually consider a Polish metric space equipped with a $\sigma$-finite Borel measure.



\subsection{Wasserstein barycenters' absolute continuity by approximation}

We first deduce an intermediate result by applying the law of
large numbers for Wasserstein barycenters to the displacement functionals
introduced in \Cref{prop:entropy_estimation}.

The following lemma,
taken from Santambrogio \cite[Proposition 7.7, Remak 7.8]{Santambrogio2015},
originates from Buttazzo and Freddi \cite[Theorem 2.2]{buttazzo1991functionals},
which was slightly generalized later in
\cite[Theorem 2.34]{ambrosio2000functions}.
One can find another slightly generalized version by
Ambrosio et al.\@ \cite[Theorem 15.8, Theorem 15.9]{ambrosio2021lectures}
with a proof for the case of Euclidean spaces.

\begin{lem}
	\label{lem:lower_semi_continuity_entropy}
	Let $E$ be a Polish metric space with a $\sigma$-finite Borel measure $\mu$.
	Let $G$ be a function defined on $[0, \infty)$ such that
	\begin{enumerate}
		\item $G(x) \ge 0$;
		\item $G$ continuous and convex;
		\item $\displaystyle \lim_{x \rightarrow \infty} {G(x)}/{x} = \infty$.
	\end{enumerate}
	With respect to the reference measure $\mu$,
	if there is a sequence of absolutely continuous probability measures
	$ \nu_i = f_i \diff \mu, \, i \ge 1$
	converging weakly to a probability measure $\nu$
	such that
	$ \displaystyle \liminf_{i \rightarrow \infty} \int_{E} G(f_i) \diff \mu$ is finite,
	then $\nu$ is also absolutely continuous and
	\begin{equation}
		\label{equa:integral_functional_lower_semicontinuity}
		\int_{E} G(f) \diff \mu
		\le \liminf_{i \rightarrow \infty} \int_{E} G(f_i) \diff \mu < \infty,
	\end{equation}
	where $f$ is the density of $\nu$.
\end{lem}

Since convergence in Wasserstein metric implies weak convergence,
\Cref{lem:lower_semi_continuity_entropy} ensures that
the set below is closed in $\mathcal{W}_2(E)$.

\begin{defn}[$\bset(G,L)$ sets]
	\label{defn:bounded_entropy_subset}
	Let $E$ be Polish space with a $\sigma$-finite Borel measure $\mu$.
	Let $G$ be a function on $[0, \infty)$ such that
	\begin{enumerate}
		\item $G$ is non-negative and $G(x) = 0$ for $x \in [0, 1]$;
		\item $G$ is non-decreasing, continuous and convex;
		\item $\displaystyle \lim_{x \rightarrow \infty} {G(x)}/{x} = \infty$;
		\item the function $H(x) : = G(e ^ x) / e^ x$
		      has continuous non-negative bounded derivative.
	\end{enumerate}
	Given $L > 0$, the following set
	of measures,
	\[
		\bset(G, L) : = \left\{ \nu \in \mathcal{W}_2(E) \mid
		\nu = f\cdot \mu, \, \int_M G(f)
		\diff \mu \le L \right\},
	\]
	is a closed subset of $\mathcal{W}_2(E)$.
\end{defn}

The function $\widehat G: x \mapsto x\, \log x$ on $[0, +\infty)$
is not always positive and non-decreasing, so it fails to meet the above assumptions.
Since $\widehat G(e^{-1}) = -e^{-1}$ is the minimum value of $\widehat G$,
we can consider the function that is equal to $0$ on $[0, 1]$ and
is equal to $\widehat G(x / e) + e^{-1}$ on $x \in [1, +\infty)$,
which is a valid example.
Indeed, we include the property that $G$ is non-decreasing
to ensure that each element in $\bset(G, L)$ can be approximated
by elements in $\bset(G, L + 1)$ with compact support,
as shown in the following lemma.

\begin{lem}
	\label{lem:B_G_L_sets_with_compact_support}
	Let $(E, d)$ be a proper metric space equipped with a $\sigma$-finite Borel measure $\mu$.
	Fix a $\bset(G, L)$ set as defined in \Cref{defn:bounded_entropy_subset}.
	For any probability measure $\nu \in \bset(G, L)$,
	there exists a sequence of probability measures
	in $\bset(G, L + 1)$ with compact support
	that converges to $\nu$ with respect to the Wasserstein metric.
\end{lem}

\begin{proof}
	Let $f$ be the density function of $\nu$ with respect to $\mu$, i.e., $\nu = f \cdot \mu$.
	Since the integral $\int_E f \diff \mu = 1$ is non-zero,
	there exists a positive number $l > 0$ such that
	the set $\{ x \in E \mid f(x) \le l\}$ is not $\mu$-negligible.
	Since $\mu$ is $\sigma$-finite, there exists a bounded subset $Y \subset E$
	such that $f(y) \le l$ for $y \in Y$ and $0 < \mu(Y) < + \infty$.
	We define for $(k, x) \in \mathbb{N}^* \times E$,
	\begin{equation}
		\label{equa:approximate_B_G_L_density}
		g(k, x) :=
		f(x) \mathbbm{1}_{\widebar{B}(x_0, k)} (x) +
		\alpha_k\, \mathbbm{1}_{Y \cap \widebar{B}(x_0, k)} (x),
	\end{equation}
	where we set $\alpha_k := 0$ if $\mu(Y \cap \widebar{B}(x_0, k)) = 0$
	and $\alpha_k := [ 1 - \nu(\widebar{B}(x_0, k))] / \mu(Y \cap \widebar{B}(x_0, k))$
	if $\mu(Y \cap \widebar{B}(x_0, k)) > 0$.
	Since $\lim_{k \rightarrow \infty}  \mu(Y \cap \widebar{B}(x_0, k)) = \mu(Y) > 0$,
	for $k$ sufficiently large such that $a_k > 0$,
	the sequence $\alpha_k$ is decreasing with $\lim_{k \rightarrow +\infty} \alpha_k = 0$.
	Let $k_0 \in \mathbb{N}^*$ be the smallest integer such that $\alpha_{k_0} > 0$.
	Our choices of $\alpha_k$ and $k_0$ ensure that for $n \in \mathbb{N}^*$,
	$\alpha_{k_0 + n} > 0$ and $g(k_0 + n, \cdot)$
	is a probability density function with respect to $\mu$.
	Define $\nu_n := g(k_0 + n, \cdot) \cdot \mu$.
	Since $(E, d)$ is a proper metric space,
	$Y$ is pre-compact set, which implies that
	$\nu_n$ is a probability measure with compact support
	and thus $\nu_n \in \mathcal{W}_2(E)$.
	We now prove the convergence $\nu_n \rightarrow \nu$ with respect to $W_2$
	using test functions.
	For a continuous function $\phi: E \rightarrow \mathbb{R}$ such that
	$|\phi(x)| \le 1 + d(x_0, x)^2$, note that
	\[
		|\phi (x) \,g(k_0 + n, x) | \le
		\left(1 + d(x_0, x)^2\right)
		\cdot \left(f(x) + \alpha_{k_0}\,\mathbbm{1}_Y(x)\right)
		\quad\text{and}\quad
		\lim_{n \rightarrow \infty} g(k_0+n, x) = f(x).
	\]
	As $Y$ is pre-compact with $\mu(Y) < + \infty$ and $\nu \in \mathcal{W}_2(E)$,
	it follows from the dominated convergence theorem that
	\begin{equation*}
		\lim_{n \rightarrow \infty}	\int_E \phi \diff \nu_n
		=  \lim_{n \rightarrow \infty} \int_E \phi(x)\, g(k_0 +n, x) \diff \mu(x)
		=  \int_E \phi(x)\, f(x) \diff \mu(x)
		= \int_E \phi \diff \nu,
	\end{equation*}
	which implies $\lim_{n \rightarrow \infty} W_2(\nu_n, \nu) = 0$
	according to the characterization of $W_2$ using the test function $\phi$
	\cite[(iv) of Definition 6.8 and Theorem 6.9]{villani2009optimal}.

	Since $f(y) \le l$ for $y \in Y$ and $G$ is non-decreasing,
	we have
	\begin{equation}
		\label{equa:dominated_convergence_G}
		\forall\,(n, x) \in \mathbb{N}^* \times E,\quad
		G(g(k_0 +n, x)) \le G(f(x)) + G(l + \alpha_{k_0})\, \mathbbm{1}_Y(x).
	\end{equation}
	Since $G$ is a continuous function and $\mu(Y) < 0$,
	we can apply the dominated convergence theorem to (\ref{equa:dominated_convergence_G})
	and obtain
	\[
		\lim_{n \rightarrow \infty} \int_E G(g(k_0 +n, x)) \diff \mu(x)
		= \int_E G(f(x)) \diff \mu(x).
	\]
	Hence, for $n$ sufficiently large, $\nu_n \in \bset(G, L + 1)$,
	which concludes the proof.
\end{proof}

As the assumptions in \Cref{defn:bounded_entropy_subset} include the ones
we used to construct displacement functionals in \Cref{prop:entropy_estimation},
we obtain the following intermediate result.

\begin{prop}
	\label{thm:absolute_continuity_main_theorem}
	Let $(M, \matheuvm{g})$ be a complete Riemannian manifold
	with a lower Ricci curvature bound.
	If $\mathbb{P} \in \mathcal{W}_2(\mathcal{W}_2(M))$
	gives mass to some closed set $\bset(G, L)$
	defined in \Cref{defn:bounded_entropy_subset},
	i.e., $\mathbb{P}(\bset(G, L)) > 0$,
	then the unique barycenter of $\mathbb{P}$ is absolutely continuous.
\end{prop}

\begin{proof}
	Write $\mathbb{P} = \mathbb{P}(\bset(G,L)) \,\mathbb{P}^1 +
		(1 - \mathbb{P}(\bset(G,L))\, \mathbb{P}^2$
	with $\mathbb{P}^1, \mathbb{P}^2 \in \mathcal{W}_2(\mathcal{W}_2(M))$ such that
	$\mathbb{P}^1$ is supported in $\bset(G,L)$.
	We approximate $\mathbb{P}$ in the Wasserstein metric $\mathbb{W}_2$
	with finitely supported measures
	$\mathbb{P}_j \in \mathcal{W}_2(\mathcal{W}_2(M))$ by approximating $\mathbb{P}^1$
	and $\mathbb{P}^2$ as follows.

	Since $\bset(G, L)$ equipped with the Wasserstein metric $\mathbb{W}_2$
	is a non-empty closed subspace of $\mathcal{W}_2(M)$,
	we can construct the Wasserstein space $\mathcal{W}_2(\bset(G, L))$
	and treat $\mathbb{P}^1$ as an element in it.
	Recall that the set of finitely supported measures is dense
	in Wasserstein spaces \cite[Theorem 6.18]{villani2009optimal}.
	Applying this property to the Wasserstein spaces
	$\mathcal{W}_2(\bset(G, L))$ and $\mathcal{W}_2(\mathcal{W}_2(M))$,
	we obtain two sequences of finitely supported
	probability measures $\{\mathbb{P}^1_j\}_{j \ge 1}$ and $\{\mathbb{P}^2_j\}_{j \ge 1}$
	satisfying $\mathbb{W}_2(\mathbb{P}_j^1, \mathbb{P}^1) \rightarrow 0$,
	$\mathbb{W}_2(\mathbb{P}_j^2, \mathbb{P}^2) \rightarrow 0$ when $j \rightarrow \infty$.
	Furthermore, thanks to \Cref{lem:B_G_L_sets_with_compact_support},
	we can further refine the two approximating sequences to
	ensure that all $\mathbb{P}_j^1, \mathbb{P}_j^2$ for $j \ge 1$
	are supported in probability measures with compact support
	and $\mathbb{P}^1_j(\bset(G, L + 1)) = 1$.
	Define $\mathbb{P}_j := \mathbb{P}(\bset(G,L)) \,\mathbb{P}_j^1 +
		(1 - \mathbb{P}(\bset(G,L))\, \mathbb{P}_j^2$.
	It follows that $\mathbb{W}_2(\mathbb{P}_j, \mathbb{P}) \rightarrow 0$ as $j \rightarrow \infty$.

	Consider the displacement functional
	$\mathcal{G}: f \cdot \operatorname{Vol} \mapsto \int_M G(f) \diff \operatorname{Vol}$.
	\Cref{prop:entropy_estimation} implies the following estimate
	of $\mathcal{G}(\mu_{\mathbb{P}_j})$ at the barycenter $\mu_{\mathbb{P}_j}$ of $\mathbb{P}_j$,
	\begin{equation}
		\label{equa:entropy_estimation_for_approximating_sequence}
		\mathcal{G}(\mu_{\mathbb{P}_j}) \le 
		\int_{\mathcal{W}_2(M)} \mathcal{G}(\nu) \diff \mathbb{P}^1_j(\nu)
		+ \frac{L_H K}{2 \Lambda}
		\mathbb{W}_2(\mathbb{P}_j, \delta_{\mu_{\mathbb{P}_j}})^2
		+ \frac{L_H}{2 \Lambda}({m^2} + 2m),
	\end{equation}
	where $\Lambda := \mathbb{P} (\bset(G,L))$, $-K$ is a lower Ricci curvature bound of $M$,
	$m$ is the dimension of $M$,
	and $L_H$ is an upper bounded of the $H^\prime$ with $H(x) : = G(e^x) e^{-x}$.
	Denote by $\mu_{\mathbb{P}}$ the unique barycenter of $\mathbb{P}$,
	\Cref{thm:law_of_large_numbers_Wasserstein_barycenter}
	implies that $W_2(\mu_{\mathbb{P}_j} , \mu_{\mathbb{P}}) \rightarrow 0$
	and thus $\mathbb{W}_2(\mathbb{P}_j, \delta_{\mu_{\mathbb{P}_j}})
		\rightarrow \mathbb{W}_2(\mathbb{P}, \delta_{\mu_{\mathbb{P}}})$ as $j \rightarrow \infty$.
	Since the support of $\mathbb{P}_j^1$ is a subset of $\bset(G, L + 1)$
	and $\mathbb{W}_2(\mathbb{P}_j, \delta_{\mu_{\mathbb{P}_j}})$ is bounded for $j \ge 1$,
	by setting
	\[
		L^\prime : = (L + 1)
		+ \frac{L_H K}{2 \Lambda}
		\sup_{j \ge 1} \mathbb{W}_2(\mathbb{P}_j, \delta_{\mu_{\mathbb{P}_j}})^2
		+ \frac{L_H}{2 \Lambda}({m^2} + 2m),
	\]
	we have $\mu_{\mathbb{P}_j} \in \bset(G, L^\prime)$
	for all $j \ge 1$.
	It follows from \Cref{lem:lower_semi_continuity_entropy}
	that $\mu_{\mathbb{P}}$ is absolutely continuous.
\end{proof}

We replace the assumption $\mathbb{P}(\bset(G, L)) > 0$
by a more natural one in the next subsection.


\subsection{Compactness using Souslin space theory}
\label{sec:topology_density_functions}

The last step towards our main result
is to show that the closed
subset $\bset(G, L)$ needed in \Cref{thm:absolute_continuity_main_theorem} always exists
if $\mathbb{P}$ gives mass to the set of absolutely continuous measures.
Our inspiration is the criterion of uniform integrability by
Charles-Jean de la Vall\'ee Poussin.
This criterion \cite[Theorem 4.5.9]{bogachev2007measure} constructs
a functional $ f \mapsto \int G(f) \diff \mu$
that is uniformly bounded for a family of uniformly
integrable functions.
We have enough freedom in its construction
to impose the properties required by \Cref{defn:bounded_entropy_subset}
on the function $G$.
Pre-compact sets of measures with respect to the topology $\tau$ defined below are
closely related to uniformly integrable families.

\begin{defn}[The set $\mathbb{A}$ and four topologies $\tau_w, \tau_W, \tau, \tau_L$]
	\label{defn:set_of_absolutely_continuous_measures}
	Let $E$ be a Polish space with a $\sigma$-finite reference measure $\mu$.
	Pick a point $x_0 \in E$ and
	define the following set of measurable functions on $E$,
	\begin{equation}
		\label{equa:defn_density_proba_ball}
		\mathbb{A} :=
		\left\{ f \in L^1(\mu) \,\bigg|\,
		f \ge 0,\,  \int_E f \diff \mu = 1, \,
		\int_E d(x_0, x)^2 f(x) \diff \mu(x) < \infty \right\},
	\end{equation}
	which is independent of the chosen point $x_0$.
	The set $\mathbb{A}$ is naturally identified via $f \leftrightarrow f \cdot \mu$
	with the set of probability measures
	in $\mathcal{W}_2(E)$ that are absolutely continuous with respect to $\mu$.
	We introduce the following four topologies.
	Denote by $\tau_w$ the topology 
	on $\mathcal{W}_2(E)$ with respect to the weak convergence,
	denote by $\tau_W$ the topology of the Wasserstein space $\mathcal{W}_2(E)$,
	denote by $\tau$ the weak topology on $L^1(\mu)$
	induced by its dual space $L^\infty(\mu)$
	\cite[Theorem 4.4.1]{bogachev2007measure}
	and denote by $\tau_L$ the topology of the Lebesgue space $L^1(\mu)$.
	By definition, $\tau_w \subset \tau_W$ and $\tau \subset \tau_L$.
	Denote by $(\mathbb{A}, \tau_w), (\mathbb{A}, \tau_W), (\mathbb{A}, \tau)$ and $(\mathbb{A}, \tau_L)$
	the four topological subspaces induced
	by these topologies on the set $\mathbb{A}$.
\end{defn}

Consider the case when $E$ is a complete Riemannian manifold and $\mu$ is the volume measure on $E$.
By \Cref{lem:measurable_set_absolutely_continous_measures},
$\mathbb{A}$ is a Borel set for the topology $\tau_W$.
Given a probability measure $\mathbb{P} \in \mathcal{W}_2(\mathcal{W}_2(E))$
such that $\mathbb{P}(\mathbb{A}) > 0$,
our goal is to find a compact subset $\mathcal{F}$ in $(\mathbb{A}, \tau)$
with $\mathbb{P}(\mathcal{F}) > 0$.
If we can accomplish this, then $\mathcal{F}$ forms
a family of uniformly integrable functions by the Dunford-Pettis theorem
(\Cref{thm:Dunford_Pettis_theorem}),
bringing us closer to the main result.
To find such an $\mathcal{F}$, a direct but problematic approach is
to argue that $\mathbb{P}$ is a Radon measure.
However, this argument overlooks that crucial point that
$\mathbb{P}$ (restricted on $\mathbb{A}$) must be a Borel measure
with respect to the Borel sets of $(\mathbb{A}, \tau)$.
To address this issue, we revisit the Souslin space theory.
Our main reference is Bogachev \cite[Section 6.6, Section 6.7, Section 7.4]{bogachev2007measure}.
\begin{defn}[Souslin space]
	A set in a Hausdorff space is called Souslin if it is
	the image of a Polish metric space under a continuous map.
	A Souslin space is a Hausdorff space that is a Souslin set. The empty set is
	Souslin as well.
\end{defn}

By definition, Polish spaces are Souslin.
Here are some properties of Souslin spaces:
\begin{enumerate}
	\item Every Borel subset of a Souslin space is a Souslin
	      space \cite[Theorem 6.6.7]{bogachev2007measure};
	\item Let $E$ and $F$ be Souslin spaces and let $f: E \mapsto F$ be
	      a measurable map. If $f$ is bijective,
	      then $E$ and $F$ share the same Borel sets, see
	      \cite[Proposition 423F]{fremlin2000measure} or \cite[Theorem 6.7.3]{bogachev2007measure};
	\item If $E$ is a Souslin space, then every finite Borel
	      measure $\mu$ on $E$ is Radon \cite[Theorem 7.4.3]{bogachev2007measure}.
\end{enumerate}
These properties are used in the following lemma
to justify the previous arguments with Radon measures.

\begin{lem}
	\label{lem:density_function_topologies_induce_same_Borel_sets}
	Let $(E, d)$ be a Polish space
	with an outer regular and $\sigma$-finite Borel measure $\mu$ on $E$.
	Let $\mathbb{A}$ be as (\ref{equa:defn_density_proba_ball}).
	The four topological subspaces,
	 $(\mathbb{A}, \tau_w), (\mathbb{A}, \tau_W), (\mathbb{A}, \tau)$, and $(\mathbb{A}, \tau_L)$
	share the same Borel sets.
	
	In particular,
	if $\mathbb{P} \in \mathcal{W}_2(\mathcal{W}_2(E))$
	gives mass to the set $\mathbb{A}$,
	then it gives mass to a compact subset of $(\mathbb{A}, \tau)$.
\end{lem}

\begin{proof}
	For spaces $(\mathbb{A}, \tau_w)$ and $(\mathbb{A}, \tau_W)$,
	the first statement is already proven in \cite[Lemma 2.4.2]{panaretos2020invitation},
	and we recall its arguments here.
	By \Cref{lem:measurable_set_absolutely_continous_measures},
	$\mathbb{A}$ is a Borel set for both $\tau_w$ and $\tau_W$.
	Since $(\mathcal{W}_2(E), W_2)$ is a Polish space,
	$(\mathbb{A}, \tau_W)$ is then a Souslin space as a Borel
	subset of $(\mathcal{W}_2(E), W_2)$.
	Consider the identity map $\operatorname{Id}:
		(\mathbb{A}, \tau_W) \rightarrow (\mathbb{A}, \tau_w)$,
	it is continuous and bijective.
	By definition, $(\mathbb{A}, \tau_w)$
	is a Souslin space as the image of
	the Souslin space $(\mathbb{A}, \tau_W)$ under
	the continuous map $\operatorname{Id}$.
	Moreover, $(\mathbb{A}, \tau_W)$ and $(\mathbb{A}, \tau_w)$
	share the same Borel sets since
	the measurable map $\operatorname{Id}$ is bijective.

	We claim that $(\mathbb{A}, \tau_L)$ is also a Souslin space.
	We first prove that the Lebesgue space
	$L^1(\mu)$ is complete and separable using the assumption
	that $E$ is Polish.
	$L^1(\mu)$ is complete for any measurable space $E$
	\cite[Theorem 4.1.3]{bogachev2007measure}.
	Its separability is asserted in Br\'ezis \cite[Theorem 4.13]{brezis2011functional} and Bogachev
	\cite[Section 1.12(iii), Corollary 4.2.2, Exercise 4.7.63]{bogachev2007measure} but only proven
	for the case of Euclidean spaces.
	Here is a brief proof of it.
	Every Polish space is homeomorphic to a closed subspace
	of $\mathbb{R}^\infty$ \cite[Theorem 6.1.12]{bogachev2007measure}.
	Moreover, one can show that $L^1(\mu)$ is separable when $E = \mathbb{R}^\infty$
	using the same arguments for Euclidean spaces.
	It follows that $L^1(\mu)$ is a Polish space.
	We then prove that $\mathbb{A}$ is a Borel set for the topology $\tau_L$.
	Fix a point $x_0 \in E$.
	Define the following sets for integers
	$k, j \ge 1$,
	\[
		A_{k,j} :=
		\left\{ f \in L^1(\mu) \,\bigg|\,
		f \ge 0,\,  \int_E f \diff \mu = 1, \,
	\int_E \min\{ d(x_0, x)^2, k \} f(x) \diff \mu(x) \le j \right\}.
	\]
	Fix two integers $k, j \ge 1$.
	We show that the set $A_{k,j}$ is a closed subset of $L^1(\mu)$.
	Let $\{f_i\}_{i \ge 1} \subset A_{k,j}$ be a sequence converging
	to $f \in L^1(\mu)$ in $L^1(\mu)$.
	Since $\{f_i\}_{i \ge 1}$ has a subsequence
	converging almost everywhere to $f$,
	$f$ is non-negative ($\mu$-almost everywhere).
	It follows that
	$\int_E f \diff \mu = \| f \|_{L^1(\mu)} =
	\lim_{i \rightarrow \infty} \| f_i \|_{L^1(\mu)} = 1$.
	Noting that as $i \rightarrow \infty$,
	\[
		\| \min\{d(x_0, \cdot)^2, k\} f_i - \min\{d(x_0, \cdot)^2, k \} f\|_{L^1(\mu)}
		\le k \| f_i - f \|_{L^1(\mu)} \rightarrow 0,
	\]
	which implies that $f \in A_{k,j}$.
	Hence, $A_{k,j}$ is a closed subset of $L^1(\mu)$.
	By the monotone convergence theorem, we have
	$\mathbb{A} = \cup_{j \ge 1} \cap_{k \ge 1} A_{k,j}$,
	which proves that $\mathbb{A}$ is a Borel set.
	Finally, $(\mathbb{A}, \tau_L)$ is a Souslin space
	as $\mathbb{A}$ is a Borel set of the Polish space $L^1(\mu)$.

	By definition of $\tau_w$ and $\tau$,
	we have the topological inclusions
	$(\mathbb{A}, \tau_w) \subset (\mathbb{A}, \tau) \subset (\mathbb{A}, \tau_L)$.
	Using the identity map as before, we conclude that
	the three topological spaces,
	$(\mathbb{A}, \tau_w)$, $(\mathbb{A}, \tau)$ and $(\mathbb{A}, \tau_L)$, share the same Borel sets
	since $(\mathbb{A}, \tau_L)$ is a Souslin space.

	$\mathbb{P}$, restricted on $\mathbb{A}$, is then a Radon measure with respect
	to the common Borel sets for the four topological subspaces since
	finite Borel measures on Souslin spaces are Radon.
	Hence, $\mathbb{P}(\mathbb{A}) > 0$ can be approximated
	by the $\mathbb{P}$ measure of compact subsets of $(\mathbb{A}, \tau)$.
\end{proof}


We now recall some results from functional analysis to
prove \Cref{thm:Dunford_Pettis_theorem} later,
a slightly generalized Dunford-Pettis theorem that
connects uniform integrability and the weak topology $\tau$.
The Eberlein–Šmulian theorem characterizes compact sets with respect to
the weak topology of a Banach space.
For its proof, see \cite[Theorem 1.6.3]{albiac2006topics} or \cite[Theorem II.3]{li2017introduction}.

\begin{thm}[Eberlein–Šmulian theorem]
	\label{thm:Eberlein_Šmulian_theorem}
	A subset $K$ of a Banach space $E$ is pre-compact with respect to the weak topology $\sigma(E^*, E)$
	if and only if, from each sequence of elements of $K$, we can extract a weakly convergent subsequence.
\end{thm}

To deal with set-wise convergence of countably additive set functions (i.e., signed measures),
we recall the following Vitali–Hahn–Saks theorem.
For its proof, see \cite[\S 3.14]{swartz1994measure}, \cite[Theorem 4.6.3]{bogachev2007measure}
or \cite[Theorem A8.15]{alt2016linear}.

\begin{thm}[Vitali–Hahn–Saks theorem]
	\label{thm:Vitali_Hahn_Saks_theorem}
	Let $(\Omega, \mathcal{B})$ be a measurable space with a probability measure $\mu$ on it.
	Let $\nu_n: \mathcal{B} \rightarrow \mathbb{R},\, n\in \mathbb{N}$ be a sequence of
	real-valued countably additive set functions
	such that
	\begin{enumerate}
		\item the limit $\lim_{n \rightarrow \infty} \nu_n(A) \in \mathbb{R}$ exists and is finite
		      for any $A \in \mathcal{B}$;
		\item each $\nu_n$ is absolutely continuous with respect to $\mu$,
		      i.e., for $A \in \mathcal{B}$, $\mu(A) = 0$ implies $\nu_n(A) = 0$.
	\end{enumerate}
	Then $\{ \nu_n \}_{n \in \mathbb{N}}$ is uniformly absolutely continuous
	with respect to $\mu$, i.e.,
	\[
		\sup_{n \in \mathbb{N}} |\nu_n(A)| \rightarrow 0
		\text{ as } \mu(A) \rightarrow 0.
	\]
\end{thm}

As a corollary, we illustrate how to
apply \Cref{thm:Vitali_Hahn_Saks_theorem} with a $\sigma$-finite measure $\mu$.

\begin{coro}
	\label{coro:uniformly_absolute_continuity}
	Let $(\Omega, \mathcal{B})$ be a measurable space with a $\sigma$-finite measure $\mu$ on it.
	Let $\{f_n\}_{n \in \mathbb{N}} \subset L^1(\mu)$ be a sequence of $\mu$-integrable function
	such that there exists a $\mu$-integrable function $f \in L^1(\mu)$ satisfying
	\[
		\forall\, A \in \mathcal{B},\quad
		\lim_{n \rightarrow \infty} \int_A f_n \diff \mu = \int_A f \diff \mu.
	\]
	Then for any $\epsilon > 0$, there exists $\delta > 0$ such that for $A \in \mathcal{B}$,
	\[
		\mu(A) < \delta \implies \sup_{n \in \mathbb{N}} \int_A f_n \diff \mu < \epsilon.
	\]
\end{coro}

\begin{proof}
	Since $\mu$ is $\sigma$-finite, there exists an at most countable family
	of pairwise disjoint measurable sets,
	$\{ E_j, j \in J \}$ ($J \subset \mathbb{N}$),
	such that $0< \mu(E_j) < +\infty$ and $\mu(\Omega \setminus \cup_{j \in J} E_j) = 0$.
	Define the measure $\eta: = \sum_{j \in J} \lambda_j \frac{1}{\mu(E_j)} \mu |_{E_j}$
	with $\lambda_j : = 2^{-j} / \sum_{k \in J} 2^{-k}$.
	Since $\sum_{j \in J} \lambda_j = 1$, $\eta$ is a probability measure satisfying
	\begin{equation}
		\label{equa:define_eta}
		\forall\, A \in \mathcal{B},\quad
		\eta (A) = \int_A \sum_{j \in J} \frac{\lambda_j }{\mu(E_j)} \mathbbm{1}_{E_j} \diff \mu.
	\end{equation}
	As $\mu(\Omega \setminus \cup_{j \in J} E_j) = 0$,
	(\ref{equa:define_eta}) implies that
	$\eta(A) = 0$ if and only if $\mu(A) = 0$.

	For $n \in \mathbb{N}$, define the countably additive function
	$\nu_n: \mathcal{B} \rightarrow \mathbb{R}$,
	\[
		\nu_n (A): = \int_A f_n \diff \mu,\quad A \in \mathcal{B}.
	\]
	As $f \in L^1(\mu)$, the limit $\lim_{n \rightarrow \infty} \mu_n(A) = \int_A f \diff \mu$
	always exists and is finite.
	Since $\eta(A) = 0$ implies $\mu(A) = 0$ and thus $\nu_n(A) = 0$,
	\Cref{thm:Vitali_Hahn_Saks_theorem} is applicable to $\{ \nu_n \}_{n \in \mathbb{N}}$
	with the probability measure $\eta$, which implies that
	$\sup_{n \in \mathbb{N}} |\nu_n(A)| \rightarrow 0$ as $\eta(A) \rightarrow 0$.
	Moreover, since $\eta$ is finite measure that is absolutely continuous
	with respect to $\mu$,
	the convergence $\mu(A) \rightarrow 0$ implies $\eta(A) \rightarrow 0$ \cite[Lemma 4.2.1]{cohn2013measure}.
	Hence, $\mu(A) \rightarrow 0$ implies $\sup_{n \in \mathbb{N}} |\nu_n(A)| \rightarrow 0$,
	which concludes the proof.
\end{proof}

We are ready to prove the following Dunford-Pettis theorem,
where the $\sigma$-finiteness of $\mu$ is not a standard assumption.

\begin{prop}[Dunford-Pettis theorem]
	\label{thm:Dunford_Pettis_theorem}
	Let $(E, \mathcal{B})$ be a measurable space with
	a $\sigma$-finite Borel measure $\mu$ on it.
	Let $\mathcal{F} \subset L^1(\mu)$ be a set of $\mu$-integrable functions.
	If $\mathcal{F}$ has compact closure in the weak topology induced
	by the dual space $L^\infty(\mu)$ of $L^1(\mu)$, then $\mathcal{F}$ is
	\emph{uniformly integrable}, i.e.,
	\[
		\lim_{C \rightarrow \infty} \sup_{f \in \mathcal{F}}
		\int_{ \{ | f | > C \} } | f | \diff \mu = 0.
	\]
\end{prop}

\begin{proof}
	We need the assumption of $\mu$ being $\sigma$-finite
	to ensure that $L^{\infty}(\mu)$ is the dual space of $L^1(\mu)$,
	see \cite[Theorem 4.4.1]{bogachev2007measure} and \cite[Exercise 6.12]{rudin1987real}.
	The above definition of uniform integrability is
	taken from Bogachev \cite[Definition 4.5.1]{bogachev2007measure}.
	When $\mu$ is finite, the equivalence
	between pre-compactness in the weak topology and uniform integrability
	is already proven by Bogachev \cite[Theorem 4.7.18]{bogachev2007measure}.
	The following arguments for the general case are based on his proof.

	We prove our statement for $\sigma$-finite measures by contradiction.
	Suppose that $\mathcal{F}$ has compact closure in the weak topology, but
	is not uniformly integrable.
	Then, there are $\epsilon > 0$ and a sequence $\{f_n\}_{n \ge 1} \subset \mathcal{F}$
	such that
	\begin{equation}
		\label{equa:not_uniformly_integrable}
		\inf_{ n \ge 1} \int_{\{ | f_{n} | > n \}}
		|f_{n}| \diff \mu \geq \epsilon.
	\end{equation}
	Applying the Eberlein–Šmulian theorem (\Cref{thm:Eberlein_Šmulian_theorem})
	to $\{f_n\}$ and the Banach space $L^1(\mu)$ \cite[Theorem 4.8]{brezis2011functional},
	we obtain a subsequence $\{f_{n_k}\}_{ k \ge 1 }$
	convergent to some function $f \in L^1(\mu)$ in the weak topology.
	In particular, for every measurable set $A \in \mathcal{B}$ we have
	\begin{equation}
		\label{equa:convegence_set_wise}
		\lim_{k \rightarrow \infty} \int_A f_{n_k} \diff \mu = \int_A f \diff \mu.
	\end{equation}
	It follows from the Vitali–Hahn–Saks theorem (\Cref{coro:uniformly_absolute_continuity})
	that sequence $\{f_{n_k}\}_{k \ge 1}$
	has uniformly absolutely continuous integrals,
	i.e., for every $\epsilon > 0$, there exists $\delta > 0$ such that
	\begin{equation}
		\label{equa:uniform_absolute_continous_integral}
		\mu(A) < \delta \implies
		\sup_{k \ge 1} \int_A | f_{n_k} | \diff \mu < \epsilon.
	\end{equation}
	Via the isometric embedding of $L^1(\mu)$ into the dual space of
	$L^\infty(\mu)$ \cite[Corollary 1.4]{brezis2011functional},
	the Banach–Steinhaus theorem \cite[Theorem 4.4.3]{bogachev2007measure}
	is applicable to the Banach space $L^\infty(\mu)$
	and the convergent sequence of functional $\{f_{n_k}\}_{k \ge 1}$,
	which implies that
	$C : = \sup_{k \ge 1} \| f_{n_k} \|_{L^1(\mu)} < \infty$ is finite.
	Take the $\delta$ given by (\ref{equa:uniform_absolute_continous_integral})
	for the $\epsilon$ in (\ref{equa:not_uniformly_integrable}), and
	let $n$ be an integer bigger than $C / \delta$.
	Then by Chebyshev’s inequality,
	\[
		\sup_{k \ge 1} \mu (\{| f_{n_k} | > n\}) \le \frac{1}{n} \sup_{k \ge 1}
		\| f_{n_k} \|_{L^1(\mu)} < \delta,
	\]
	which leads to a contradiction between (\ref{equa:not_uniformly_integrable})
	and (\ref{equa:uniform_absolute_continous_integral}).
\end{proof}

We also generalize the de la Vall\'ee Poussin criterion to construct the function $G$
in \Cref{defn:bounded_entropy_subset}.
In the following proposition, the $\sigma$-finiteness of $\mu$ allows us to apply Fubini's theorem.

\begin{thm}[De la Vall\'ee Poussin criterion]
	\label{lem:modified_de_la_Vallee_Poussin_theorem}
	Let $(E, \mathcal{B})$ be a measurable space with a $\sigma$-finite Borel measure $\mu$ on it.
	A subset $\mathcal{F} \subset L^1(\mu)$ is uniformly integrable, i.e.,
	\[
		\lim_{C \rightarrow \infty} \sup_{f \in \mathcal{F}}
		\int_{ \{ | f | > C \} } | f | \diff \mu = 0
	\]
	if and only if there exists a function $G$ defined on $[0, +\infty)$
	such that
	\begin{enumerate}
		\item $G(x) = 0$ for $ 0 \le x \le 1$;
		\item $G$ is a non-decreasing and convex function that is smooth on $(0, +\infty)$;
		\item $\sup_{f \in \mathcal{F}} \int_E G(| f |) \diff \mu \le 1$;
		\item if we define the function $H(x): =G(e^x) e^{-x}$ on $\mathbb{R}$,
		      then $\displaystyle \lim_{x \rightarrow + \infty } H(x) = + \infty$, and its derivative
		      ${H}^\prime$ is smooth with $ 0 \le H^{\prime}(x) \le 1$.
	\end{enumerate}
\end{thm}

\begin{proof}
	If we have the asserted function $G$ for some subset $\mathcal{F} \subset L^1(\mu)$,
	then for every $\epsilon > 0$, we can find a real number $C > 0$ such that
	$G(t) / t \ge 2 / \epsilon$ for any $t > C$.
	It implies that $ | f(x) | \le \epsilon \, G(| f(x) |) / 2$
	for all $ f \in \mathcal{F}$ when $| f(x) | > C$. Hence,

	\[	\int_{\{ | f | > C \}} | f | \diff \mu
		\leq \frac{\varepsilon}{2}
		\int_{\{ | f | > C \}}\, G \circ | f | \diff \mu
		\le \epsilon,
	\]
	which shows that $\mathcal{F}$ is uniformly integrable.

	Now assume that we are given a uniformly integrable subset $\mathcal{F} \subset L^1(\mu)$.
	To better motivate our construction of $G$,
	we postpone the definition of a smooth function $H$
	with $H(x) =0, x \le 0$ to (\ref{equa:definition_H})
	but use it here to define $G(x) := H(\log x) \, x$.
	Differentiate this equation twice,
	we obtain ${G}^{\prime \prime}(x) = [ H^\prime(\log x) + H^{\prime\prime}(\log x)] / x$.
	By our requirements on $H$, $G(x) = 0$ for $0 \le x \le 1$.
	Hence, we have
	$G(x) = \int_0^x \int_0^s G^{\prime\prime}(t) \diff t \diff s$
	for $x > 0 $ and thus
	\begin{align}
		\int_E G(|f|) \diff \mu & =
		\int_E \int_0^{|f|} \int_0^s G^{\prime\prime}(t)
		\diff t \diff s \diff \mu =
		\int_E \int_\mathbb{R} \int_\mathbb{R} {G}^{\prime \prime}(t) \cdot
		\mathbbm{1}_{0 < t < s < |f|}
		\diff t \diff s \diff \mu \nonumber                                                      \\
		                        & = \int_\mathbb{R} \int_\mathbb{R} {G}^{\prime \prime}(t) \cdot
		\mathbbm{1}_{0 < t < s} \cdot \mu(|f| > s)
		\diff t \diff s \nonumber                                                                \\
		                        & = \int_\mathbb{R} {G}^{\prime \prime}(t) \cdot
		\mathbbm{1}_{t > 0} \int_{t}^{\infty} \mu(|f| > s)
		\diff s \diff t \nonumber                                                                \\
		                        & = \int_0^\infty \frac{H^\prime(\log t)
			+ H^{\prime \prime}(\log t)}{t} \int_{t}^{\infty} \mu(|f| > s)
		\diff s \diff t \nonumber                                                                \\
		                        & = \label{equa:fubini_rewriting_G}
		\int_\mathbb{R} [{H}^\prime(y) + {H}^{\prime\prime}(y)] \int_{e^y}^\infty
		\mu(|f| > s) \diff s \diff y,
	\end{align}
	where we applied Fubini's theorem twice and a change of variable $y: = \log t$.
	According to (\ref{equa:fubini_rewriting_G}), we need to
	control ${H}^\prime + {H}^{\prime\prime}$ and the
	integral of $\mu(|f| > s)$ at the same time.
	For the integral, note that by Fubini's theorem again,
	we have for $t > 0$ and $f \in L^1(\mu)$ that

\begin{align}
	\int_{ \{ |f| > t \} } |f| \diff \mu & =
	\int_{ \{ |f| > t \} } \int_\mathbb{R} \mathbbm{1}_{0 < s < |f|} \diff s \diff \mu =
	\int_{\mathbb{R}} \int_{ E }
	\mathbbm{1}_{ |f| > t} \cdot \mathbbm{1}_{0 < s < |f|}
	\diff \mu \diff s        \nonumber       \\
	\label{equa:fubini_integrate_subset}
	& =
	\int_\mathbb{R} \int_E
	\mathbbm{1}_{0 < s < t < |f|} +
	\mathbbm{1}_{0 < t \le s < |f|} \diff \mu \diff s \nonumber \\
	&= t\, \mu(|f| > t) + \int_t^\infty \mu(|f| > s) \diff s.
\end{align}

	Let $\alpha: \mathbb{N} \rightarrow \mathbb{N}$ be a strictly increasing function
	such that $\alpha(0) \ge 0$ and
	\[
		\sup_{f \in \mathcal{F}} \int_{e^{\alpha(n)}}^\infty \mu(|f| > s) \diff s \le
		\sup_{f \in \mathcal{F}} \int_{ \{ |f| > e^{\alpha(n)} \} }
		|f| \diff \mu \le 2^{-(n + 1)},
	\]
	where we used (\ref{equa:fubini_integrate_subset}) for the first inequality
	and the uniform integrability of $\mathcal{F}$ for the second one.
	It follows that
	\begin{equation}
		\label{equa:summation_with_alphan}
		\sup_{f \in \mathcal{F}} \sum_{n \ge 0}
		\int_{e^{\alpha(n)}}^\infty \mu(|f| > s) \diff s \le 1.
	\end{equation}
	For the term ${H}^\prime + {H}^{\prime\prime}$ in (\ref{equa:fubini_rewriting_G}),
	we bound it from above with a function that is non-zero only
	on selected intervals based on our choice of $\alpha(n)$,
	allowing us to convert the integral of
	$\int_{e^y}^{\infty} \mu(|f| > s) \diff s$
	into the series summation (\ref{equa:summation_with_alphan}).
	To achieve this, we first select a
	smooth function $\gamma: \mathbb{R} \rightarrow [0,1]$ such that
	$\gamma(x) = 1$ for $x \in [\alpha(n)+ 1/3, \alpha(n) + 2/3]$
	and $\gamma(x) = 0$ for $x \notin (\alpha(n), \alpha(n)+1)$.
	Then we define
	\begin{equation}
		\label{equa:definition_H}
		H(x) := \begin{cases}
			\int_0^x e^{-s} \int_0^s \gamma(t) e^t \diff t \diff s, & x>0     \\
			0,                                                      & x \le 0
		\end{cases}.
	\end{equation}
	In this way, we have ${H}^{\prime\prime}(x) + {H}^\prime(x) = \gamma(x)$.
	Using this construction,
	(\ref{equa:fubini_rewriting_G}) and (\ref{equa:summation_with_alphan}) imply that
	\[
		\sup_{f \in \mathcal{F}} \int_E G(|f|) \diff \mu
		= \sup_{f \in \mathcal{F}}
		\sum_{n \ge 0 } \int_{\alpha(n)}^{\alpha(n) + 1} \gamma(y)
		\int_{e^y}^\infty \mu(|f| > s) \diff s \diff y
		\le \sup_{f \in \mathcal{F}} \sum_{n \ge 0} \int_{e^{\alpha(n)}}^\infty
		\mu(|f| > s) \diff s \le 1.
	\]
	For the first derivative of $H$, we have
	\[
		0 \le {H}^\prime(x) = e^{-x} \int_0^x \gamma(t) e^t \diff t \le e^{-x} (e^x -1) \le 1.
	\]
	And by direct calculation we have that the difference
	\[
		H(\alpha(n) + 1) - H(\alpha(n)) >
		\int_{\alpha(n) + \frac{2}{3}}^{\alpha(n) + 1} e^{-s}
		\int^{\alpha(n) + \frac{2}{3}}_{\alpha(n) + \frac{1}{3}} e^{t}
		\diff t \diff s
		= (1 - e^{- \frac{1}{3}})^2
	\]
	is bigger than a constant independent of $n$,
	which implies that $\displaystyle \lim_{x \rightarrow + \infty} H(x) = +\infty$
	since $H$ is non-decreasing.
	It follows from $0 \le \gamma \le 1$ that
	$G$ is non-decreasing and convex as
	${G}^{\prime\prime}(x) = \gamma(\log x) / x \ge 0$ for $ x > 1$
	and $G(x) = 0$ for $ 0 \le x \le 1$.
\end{proof}

\subsection{Final step of the proof}

To prove \Cref{thm:final_theorem_manifolds},
it remains to combine the previous auxiliary propositions
to replace the assumption in
\Cref{thm:absolute_continuity_main_theorem} that $\mathbb{P}(\bset(G, L)) > 0$
for some set $\bset(G, L)$ (\Cref{defn:bounded_entropy_subset}).

As in \Cref{defn:set_of_absolutely_continuous_measures},
we denote by $\mathbb{A}$ the set of absolutely continuous measures in $\mathcal{W}_2(M)$.
If $\mathbb{P}(\mathbb{A}) > 0$,
then \Cref{lem:density_function_topologies_induce_same_Borel_sets} provides
a compact subset $\mathcal{F}$ of $(\mathbb{A}, \tau)$ such that $\mathbb{P}(\mathcal{F}) > 0$.
Applying the Dunford-Pettis theorem (\Cref{thm:Dunford_Pettis_theorem})
to $\mathcal{F}$ with $\mu: = \operatorname{Vol}$,
we see that $\mathcal{F}$ is uniformly integrable.
Then the de la Vall\'ee Poussin criterion (\Cref{lem:modified_de_la_Vallee_Poussin_theorem})
asserts the existence of a smooth function $G$
such that $\mathcal{F} \subset \bset(G, 1) \subset \mathbb{A}$.
Therefore,
our theorem follows from \Cref{thm:absolute_continuity_main_theorem} and
the property $\mathbb{P}(\bset(G, 1)) \ge \mathbb{P}(\mathcal{F}) > 0$.


\vspace{0.5cm}
\noindent
\textbf{Acknowledgments}\quad
\small{
	This work is based on the author's Master 2 internship
	at IMT (Institut de Math\'ematiques de Toulouse),
	and it is finished during the author's preparation for his thesis.
	In both periods, the author is funded by LabEX CIMI
	(Centre International de Math\'ematiques et d’Informatique)
	under the supervision of Dr. J\'er\^ome Bertrand, to whom the author
	would like to express his sincere gratitude for his patience and invaluable
	advice and help on how to present the work properly.
}

\vspace{0.3cm}
\noindent
\textbf{Conflicts of Interest Statement}\quad
\small{
	The author declares that they have no known competing financial interests or personal relationships that could have appeared to influence the work reported in this paper.
}

\vspace{0.3cm}
\noindent
\textbf{Data Availability Statement}\quad
\small{
	Data sharing is not applicable to this article as no datasets were generated or analysed during the current study.
}

\bibliography{bibliography.bib}

@book{alt2016linear,
	title={Linear Functional Analysis: An Application-Oriented Introduction},
	author={Alt, H.W. and N{\"u}rnberg, R.},
	isbn={9781447172802},
	lccn={2016944464},
	series={Universitext},
	year={2016},
	doi={10.1007/978-1-4471-7280-2},
	publisher={Springer},
	address={London},
}

@book{li2017introduction,
	title={Introduction to Banach Spaces: Analysis and Probability},
	author={Li, Daniel and Queff{\'e}lec, Herv{\'e}},
	series={Cambridge Studies in Advanced Mathematics},
	volume={166},
	year={2017},
	doi={10.1017/CBO9781316675762},
	publisher={Cambridge University Press},
	address={Cambridge CB2 8BS, United Kingdom}
}

@book{albiac2006topics,
	title={Topics in Banach space theory},
	author={Albiac, Fernando and Kalton, Nigel John},
	series={Graduate Texts in Mathematics},
	volume={233},
	year={2006},
	isbn={978-0387-28141-4},
	publisher={Springer},
	address={New York}
}

@book{sakai1996riemannian,
	title={Riemannian Geometry},
	author={Sakai, Takashi},
	isbn={9780821889565},
	series={Translations of mathematical monographs},
	address={Providence, Rhode Island},
	url={https://books.google.fr/books?id=aqgaZVr94xMC},
	year={1996},
	publisher={American Mathematical Soc.}
}

@book{niculescu2018convex,
	title={Convex Functions and Their Applications: A Contemporary Approach},
	author={Niculescu, C.P. and Persson, L.E.},
	isbn={9783319783376},
	series={CMS Books in Mathematics},
	year={2018},
	edition={2},
	doi={10.1007/978-3-319-78337-6},
	publisher={Springer International Publishing},
	address={Cham, Switzerland}
}

@book{kobayashi1996foundations,
	title={Foundations of Differential Geometry, Volume 1},
	author={Kobayashi, Shoshichi and Nomizu, Katsumi},
	volume={1},
	year={1996},
	publisher={John Wiley \& Sons},
	address={Hoboken},
}

@article{whitehead1932convex,
	title={Convex regions in the geometry of paths},
	author={Whitehead, John Henry Constantine},
	journal={The Quarterly Journal of Mathematics},
	number={1},
	pages={33--42},
	year={1932},
	publisher={Oxford University Press}
}

@book{lee2012introduction,
	title={Introduction to Smooth Manifolds},
	author={Lee, John M.},
	isbn={9781441999825},
	lccn={2012945172},
	series={Graduate Texts in Mathematics},
	volume={218},
	year={2013},
	edition={2},
	doi={10.1007/978-1-4419-9982-5},
	publisher={Springer SZcience+Business Media},
	address={New York}
}

@article{clarke1995proximal,
	title={Proximal smoothness and the lower-$\mathcal{C}^2$ property},
	author={Clarke, Francis H and Stern, Ronald J and Wolenski, Peter R},
	journal={Journal of Convex Analysis},
	volume={2},
	number={1-2},
	pages={117--144},
	year={1995}
}

@article{vial1983strong,
	title={Strong and weak convexity of sets and functions},
	author={Vial, Jean-Philippe},
	journal={Mathematics of Operations Research},
	volume={8},
	number={2},
	pages={231--259},
	doi={10.1287/moor.8.2.231},
	year={1983},
	publisher={INFORMS}
}

@article{bangert1979analytische,
	title={{Analytische Eigenschaften konvexer Funktionen auf Riemannschen Mannigfaltigkeiten}},
	pages={309--324},
	doi={doi:10.1515/crll.1979.307-308.309},
	year={1979},
	journal={Journal für die reine und angewandte Mathematik},
	author={Bangert, Victor},
	year={1979},
	publisher={Walter de Gruyter, Berlin}
}

@book{federer2014geometric,
	title={Geometric measure theory},
	author={Federer, Herbert},
	year={1996},
	doi={10.1007/978-3-642-62010-2},
	publisher={Springer},
	address={Berlin},
}

@incollection{bourbaki2004extension,
	title={Extension of a measure. ${L}^p$ spaces},
	author={Bourbaki, Nicolas},
	booktitle={Elements of Mathematics: Integration I},
	pages={94--241},
	doi="10.1007/978-3-642-59312-3_5",
	year={2004},
	address="Berlin",
	publisher={Springer}
}

@book{lee2018introduction,
	title={Introduction to Riemannian manifolds},
	author={Lee, John M.},
	edition={2},
	doi={10.1007/978-3-319-91755-9},
	year={2018},
	publisher={Springer International Publishing AG},
	address={Chanm, Switzerland}
}

@book{swartz1994measure,
	title={Measure, integration and function spaces},
	author={Swartz, Charles W},
	doi={10.1142/2223},
	year={1994},
	publisher={World Scientific},
	address={Singapore},
}

@article{figalli2008absolute,
	title={Absolute continuity of {W}asserstein geodesics in the {H}eisenberg group},
	author={Figalli, Alessio and Juillet, Nicolas},
	journal={Journal of Functional Analysis},
	volume={255},
	number={1},
	pages={133--141},
	doi={10.1016/j.jfa.2008.03.006},
	year={2008},
	publisher={Elsevier}
}

@article{bernard2007optimal,
	title={Optimal mass transportation and {M}ather theory},
	author={Bernard, Patrick and Buffoni, Boris},
	journal={Journal of the European Mathematical Society},
	volume={9},
	number={1},
	doi={10.4171/JEMS/74},
	pages={85--121},
	year={2007}
}

@article{fathi2010optimal,
	title={Optimal transportation on non-compact manifolds},
	author={Fathi, Albert and Figalli, Alessio},
	journal={Israel Journal of Mathematics},
	volume={175},
	pages={1--59},
	doi={10.1007/s11856-010-0001-5},
	year={2010},
	publisher={Springer}
}

@article{gigli2016optimal,
	title="Optimal maps and exponentiation on finite-dimensional spaces with {R}icci curvature bounded from below",
	author={Gigli, Nicola and Rajala, Tapio and Sturm, Karl-Theodor},
	journal={The Journal of geometric analysis},
	volume={26},
	pages={2914--2929},
	doi={10.1007/s12220-015-9654-y},
	year={2016},
	publisher={Springer}
}

@article{cavalletti2017optimal,
	title={Optimal maps in essentially non-branching spaces},
	author={Cavalletti, Fabio and Mondino, Andrea},
	journal={Communications in Contemporary Mathematics},
	volume={19},
	number={06},
	pages={1750007},
	doi={10.1142/S0219199717500079},
	year={2017},
	publisher={World Scientific}
}

@article{gigli2012optimal,
	title="Optimal maps in non branching spaces with {R}icci curvature bounded from below",
	author={Gigli, Nicola},
	journal={Geom. Funct. Anal},
	volume={22},
	number={4},
	doi={10.1007/s00039-012-0176-5},
	pages={990--999},
	year={2012}
}

@book{billingsley1999convergence,
	title = {Convergence of probability measures},
	series = {Wiley series in probability and statistics Probability and statistics section},
	isbn = {9780471197454},
	year = {1999},
	edition = {2},
	language = {eng},
	author = {Billingsley, Patrick},
	doi={10.1002/9780470316962},
	lccn = {978-0-471-19745-4},
}

@book{bredon2013topology,
	title={Topology and geometry},
	author={Bredon, Glen E},
	volume={139},
	series={Graduate Texts in Mathematics},
	year={1993},
	doi={10.1007/978-1-4757-6848-0},
}

@inproceedings{fremlin2006measurable,
	title={Measurable selections and measure-additive coverings},
	author={Fremlin, DH},
	booktitle={Measure Theory Oberwolfach 1981},
	pages={425--431},
	series={Lecture Notes in Mathematics},
	volume={945},
	doi={10.1007/BFb00},
	year={2006},
}

@article{koumoullis1983ramsey,
	title={The {R}amsey property and measurable selections},
	author={Koumoullis, G and Prikry, K},
	journal={Journal of the London Mathematical Society},
	volume={s2-28},
	number={2},
	month={10},
	pages={203--210},
	year={1983},
	doi={10.1112/jlms/s2-28.2.203},
	publisher={Oxford University Press},
}

@article{sturm2005convex,
	title = {Convex functionals of probability measures and nonlinear diffusions on manifolds},
	journal = {Journal de Mathématiques Pures et Appliquées},
	volume = {84},
	number = {2},
	pages = {149-168},
	year = {2005},
	issn = {0021-7824},
	doi = {https://doi.org/10.1016/j.matpur.2004.11.002},
	url = {https://www.sciencedirect.com/science/article/pii/S0021782404001485},
	author={Sturm, Karl-Theodor}
}

@book{rudin1987real,
	title={Real and Complex Analysis},
	author={Rudin, Walter},
	year={1987},
	isbn={9780070542341},
	publisher={McGraw-Hill Book},
	edition={3},
	address={Singapore}
}

@book{brezis2011functional,
	title={Functional analysis, Sobolev spaces and partial differential equations},
	author={Br{\'e}zis, Haim},
	year={2011},
	series={Universitext},
	doi={10.1007/978-0-387-70914-7},
	publisher={Springer},
	address={New York}
}

@article{mccann1997convexity,
	title={A convexity principle for interacting gases},
	author={McCann, Robert J},
	journal={Advances in mathematics},
	volume={128},
	number={1},
	pages={153--179},
	doi={10.1006/aima.1997.1634},
	year={1997},
	publisher={Elsevier}
}

@article{sturm2006geometryI,
	title={On the geometry of metric measure spaces. {I}},
	author={Sturm, Karl-Theodor},
	journal={Acta Math},
	volume={196},
	pages={65--131},
	doi={/10.1007/s11511-006-0002-8},
	year={2006}
}

@article{sturm2006geometryII,
	title={On the geometry of metric measure spaces. {II}},
	author={Sturm, Karl-Theodor},
	journal={Acta Math},
	volume={196},
	doi={10.1007/s11511-006-0003-7},
	pages={133--177},
	year={2006}
}

@article{gigli2011inverse,
	title="On the inverse implication of {Brenier-McCann} theorems and the structure of ${(\mathcal{P}_2(M), W_2)}$",
	author={Gigli, Nicola},
	journal={Methods and Applications of Analysis},
	volume={18},
	number={2},
	pages={127--158},
	doi={10.4310/MAA.2011.v18.n2.a1},
	year={2011},
	publisher={International Press of Boston}
}

@article{kell2017transport,
	title="Transport maps, non-branching sets of geodesics and measure rigidity",
	journal={Advances in Mathematics},
	volume={320},
	pages={520-573},
	year={2017},
	issn={0001-8708},
	doi={https://doi.org/10.1016/j.aim.2017.09.003},
	url={https://www.sciencedirect.com/science/article/pii/S0001870817302347},
	author={Martin Kell}
}

@article{bertrand2008existence,
	title="Existence and uniqueness of optimal maps on {A}lexandrov spaces",
	author={Bertrand, Jérôme},
	journal={Advances in Mathematics},
	doi={10.1016/j.aim.2008.06.008},
	volume={219},
	number={3},
	pages={838--851},
	year={2008},
	publisher={Elsevier}
}

@book{ambrosio2021lectures,
	title="Lectures on Optimal Transport",
	author={Ambrosio, L. and Bru{\'e}, E. and Semola, D.},
	isbn={9783030721626},
	series={UNITEXT},
	doi={10.1007/978-3-030-72162-6},
	year={2021},
	volume={130}
}

@book{fremlin2000measure,
	title={Topological Measure Theory},
	author={Fremlin, David Heaver},
	series={Measure Theory},
	volume={4},
	publisher={Torres Fremlin},
	address={Colchester CO3 3AT, England},
	year={2005},
}

@book{panaretos2020invitation,
	title="An invitation to statistics in {W}asserstein space",
	author={Panaretos, Victor M and Zemel, Yoav},
	year={2020},
	series={SpringerBriefs in Probability and Mathematical Statistics},
	doi={10.1007/978-3-030-38438-8},
	publisher={Springer},
	address={Cham, Switzerland}
}

@book{cohn2013measure,
	title="Measure theory",
	author={Cohn, Donald L.},
	edition={2},
	series={Birkhäuser Advanced Texts Basler Lehrbücher},
	doi={10.1007/978-1-4614-6956-8},
	year={2013},
	publisher={Birkhäuser},
	address={New York}
}

@book{ambrosio2000functions,
	title="Functions of bounded variation and free discontinuity problems",
	author={Ambrosio, Luigi and Fusco, Nicola and Pallara, Diego},
	year={2000},
	series={Oxford Mathematical Monographs},
}

@article{buttazzo1991functionals,
	title="Functionals defined on measures and applications to non equi-uniformly elliptic problems",
	author={Buttazzo, Giuseppe and Freddi, Lorenzo},
	journal={Annali di matematica pura ed applicata},
	volume={159},
	number={1},
	pages={133--149},
	year={1991},
	doi={10.1007/BF01766298},
	publisher={Springer-Verlag}
}

@article{lott2009ricci,
	title="Ricci curvature for metric-measure spaces via optimal transport",
	author={Lott, John and Villani, C{\'e}dric},
	journal={Annals of Mathematics},
	pages={903--991},
	year={2009},
	volume = {169},
	number = {3},
	doi={10.4007/annals.2009.169.903},
	publisher={JSTOR}
}

@book{petersen2016riemannian,
	title="Riemannian Geometry",
	author={Petersen, Peter},
	volume={171},
	series={Graduate Texts in Mathematics},
	year={2016},
	doi={10.1007/978-3-319-26654-1},
}

@book{evans2018measure,
	title="Measure theory and fine properties of functions",
	author={Evans, Lawrence C and Garzepy, Ronald F},
	year={2018},
	publisher={Routledge},
	address={Oxfordshire},
}

@book{taylor2006measure,
	title="Measure theory and integration",
	author={Taylor, Michael Eugene},
	year={2006},
	series={Graduate Studies in Mathematics},
	volume={76},
	publisher={American Mathematical Society},
	address={Michigan},
}

@article{cordero2006prekopa,
author = {Dario Cordero-Erausquin and Robert J. McCann and Michael Schmuckenschl\"ager},
title = {Pr\'ekopa{\textendash}Leindler type inequalities on {Riemannian} manifolds, {Jacobi} fields, and optimal transport},
journal = {Annales de la Facult\'e des sciences de Toulouse : Math\'ematiques},
pages = {613--635},
publisher = {Universit\'e Paul Sabatier, Institut de Math\'ematiques},
address = {Toulouse},
volume = {Ser. 6, 15},
number = {4},
year = {2006},
doi = {10.5802/afst.1132},
mrnumber = {2295207},
zbl = {1125.58007},
language = {en}
}

@book{gallot2004riemannian,
	title="Riemannian Geometry",
	author={Gallot, Sylvestre and Hulin, Dominique and Lafontaine, Jacques},
	year={2004},
	doi={10.1007/978-3-642-18855-8},
	series={Universitext},
	address={Berlin},
	publisher={Springer Science \& Business Media}
}

@article{mccann2001polar,
	title="Polar factorization of maps on {R}iemannian manifolds",
	author={McCann, Robert J},
	journal={Geometric \& Functional Analysis GAFA},
	volume={11},
	number={3},
	doi={10.1007/PL00001679},
	pages={589--608},
	year={2001},
	publisher={Springer}
}

@book{villani2021topics,
	title="Topics in optimal transportation",
	author={Villani, C{\'e}dric},
	volume={58},
	year={2003},
	series={Graduate Studies in Mathematics},
	publisher={American Mathematical Society},
	address={Providence, Rhode Island}
}

@book{villani2009optimal,
	title="Optimal transport: old and new",
	author={Villani, C{\'e}dric},
	volume={338},
	year={2009},
	doi={10.1007/978-3-540-71050-9},
	series={Grundlehren der mathematischen Wissenschaften},
}

@article{ohta2012barycenters,
	title="Barycenters in {A}lesxandrov spaces of curvature bounded below",
	author={Ohta, Shin-Ichi},
	journal={Advances in geometry},
	volume={12},
	number={4},
	doi={10.1515/advgeom-2011-058},
	pages={571--587},
	year={2012},
	publisher={De Gruyter}
}

@article{cordero2001riemannian,
	title={A {R}iemannian interpolation inequality à la {B}orell, {B}rascamp and {L}ieb},
	author = {Dario Cordero-Erausquin and Robert J. McCann and Michael Schmuckenschl\"ager},
	journal={Inventiones mathematicae},
	year={2001},
	month={Nov},
	day={01},
	volume={146},
	number={2},
	pages={219-257},
	issn={1432-1297},
	doi={10.1007/s002220100160},
	url={https://doi.org/10.1007/s002220100160}
}

@article{jiang2017absolute,
	title="Absolute continuity of {W}asserstein barycenters over {A}lesxandrov spaces",
	author={Jiang, Yin},
	journal={Canadian Journal of Mathematics},
	volume={69},
	number={5},
	pages={1087--1108},
	doi={10.4153/CJM-2016-035-8},
	year={2017},
	publisher={Cambridge University Press}
}

@article{le2017existence,
	title="Existence and consistency of {W}asserstein barycenters",
	author={Le Gouic, Thibaut and Loubes, Jean-Michel},
	journal={Probability Theory and Related Fields},
	volume={168},
	number={3},
	doi={10.1007/s00440-016-0727-z},
	pages={901--917},
	year={2017},
	publisher={Springer}
}

@article{agueh2011barycenters,
	title="Barycenters in the {W}asserstein space",
	author={Agueh, Martial and Carlier, Guillaume},
	journal={SIAM Journal on Mathematical Analysis},
	volume={43},
	number={2},
	pages={904--924},
	year={2011},
	doi={10.1137/100805741},
	publisher={SIAM}
}

@article{kim2015multi,
	title="Multi-marginal optimal transport on {R}iemannian manifolds",
	author={Kim, Young-Heon and Pass, Brendan},
	journal={American Journal of Mathematics},
	volume={137},
	number={4},
	pages={1045--1060},
	doi={10.1353/ajm.2015.0024},
	year={2015},
	publisher={Johns Hopkins University Press}
}

@book{Santambrogio2015,
	author="Santambrogio, Filippo",
	title="Optimal Transport for Applied Mathematicians: Calculus of Variations, PDEs, and Modeling",
	year="2015",
	series={Progress in Nonlinear Differential Equations and Their Applications},
	volume={87},
	doi={10.1007/978-3-319-20828-2},
}

@article{kim2017wasserstein,
	title="Wasserstein barycenters over {R}iemannian manifolds",
	journal={Advances in Mathematics},
	volume={307},
	pages={640-683},
	year={2017},
	issn={0001-8708},
	doi={https://doi.org/10.1016/j.aim.2016.11.026},
	url={https://www.sciencedirect.com/science/article/pii/S0001870815304643},
	author={Kim, Young-Heon and Pass, Brendan},
}

@book{ambrosio2008gradient,
	title="Gradient flows: in metric spaces and in the space of probability measures",
	author={Ambrosio, Luigi and Gigli, Nicola and Savar{\'e}, Giuseppe},
	year={2005},
	address = {Basel},
	publisher={Birkhäuser},
	series={Lectures in Mathematics. ETH Zürich},
	doi={10.1007/b137080},
}

@book{bogachev2007measure,
	title="Measure theory",
	author={Bogachev, Vladimir Igorevich},
	doi={10.1007/978-3-540-34514-5},
	year={2007},
	publisher={Springer},
	address={Berlin},
}
\end{document}